\documentclass[12pt,a4paper]{amsart}

\usepackage[utf8]{inputenc}
\usepackage{amsmath}
\usepackage{amssymb}
\usepackage{amsbsy} 
\usepackage{amsfonts}
\usepackage{amsthm}
\usepackage[all]{xy}
\usepackage{tikz}
\usepackage{hyperref}
\usepackage[shortlabels]{enumitem}
\usepackage{multicol}
\usepackage[T1]{fontenc}
\usepackage{mathrsfs} 
\usepackage{tikz}
\usepackage{tikz-cd}
\usepackage{ wasysym }
\usetikzlibrary{patterns}

\theoremstyle{plain}
\newtheorem*{atw*}{Theorem}
\newtheorem{atw}{Theorem}[section]
\newtheorem*{lemma*}{Lemma}
\newtheorem{lemma}[atw]{Lemma}
\newtheorem{pro}[atw]{Proposition}
\newtheorem*{pro*}{Proposition}
\newtheorem{cor}[atw]{Corollary}
\newtheorem*{cor*}{Corollary}

\theoremstyle{remark}
\newtheorem{rem}[atw]{Remark}
\newtheorem*{rem*}{Remark}
\newtheorem{ex}[atw]{Example}
\newtheorem*{ex*}{Example}

\theoremstyle{definition}
\newtheorem{adf}[atw]{Definition}
\newtheorem*{adf*}{Definition}

\newcommand{\ws}{{\underline{ws}}}
\newcommand{\sw}{{\underline{sw}}}
\newcommand{\wu}{{\underline{w}}}

\def\epsilon{\varepsilon}

\def\DL#1{{\mathcal{T}^{\textsc r}_{#1}}}
\def\DLL#1{{\mathcal{T}^{\textsc l}_{#1}}}
\def\DLq#1{{\mathcal{T}^{{\textsc r},q}_{#1}}}
\def\DLLq#1{{\mathcal{T}^{{\textsc l},q}_{#1}}}

\newcommand{\mC}{\operatorname{mC}^\T}
\newcommand{\mCloc}{\operatorname{mC}_{\rm loc}^\T}
\newcommand{\St}{\operatorname{stab}_{\mathfrak{C_+}}}
\newcommand{\StC}{\operatorname{stab}}

\newcommand{\Hom}{\operatorname{Hom}}

\newcommand{\Pic}{\operatorname{Pic}}

\newcommand{\rank}{\operatorname{rk}}

\newcommand{\KTh}{\operatorname{K}}
\newcommand{\eu}{\operatorname{eu}}

\newcommand{\Z}{\mathbb{Z}}

\newcommand{\Q}{\mathbb{Q}}
\newcommand{\R}{\mathbb{R}}
\newcommand{\C}{\mathbb{C}}

\newcommand{\T}{\mathbb{T}}
\newcommand{\PP}{\mathbb{P}}

\newcommand{\ttt}{{\mathfrak t}}

\def\O{\mathcal{O}}
\def\Ls{\mathcal{L}_s}
\def\LL{\mathcal{L}}
\def\id{{\rm id}}

\newcommand{\xtto}[1]{\stackrel{#1}{\longrightarrow}}
\newcommand{\xto}[1]{{\xrightarrow{#1}}}
\newcommand{\mono}{\hookrightarrow}

\def\Tf#1#2{{\mathfrak T}_{#1,#2}}
\def\Tcc#1#2{{\mathscr T}_{#1,#2}}

\def\Tfr#1#2{{\mathfrak T}^{\textsc r}_{{#1,#2}}}
\def\Tfl#1#2{{\mathfrak T}^{\textsc l}_{#1,#2}}
\def\Tcr#1#2{{\mathscr T}^{ \textsc r}_{#1,#2}}
\def\Tcrr#1#2{\widetilde{\mathscr T}^{ \textsc r}_{#1,#2}}
\def\Tcl#1#2{{\mathscr T}^{\textsc l}_{#1,#2}}

\def\BB{{\mathscr B}}
\def\EE{{\mathscr E}}
\def\XX{{\mathcal X}}
\def\ZL#1{{\mathcal{Z}^{\textsc{l}}_{#1}}}
\def\ZR#1{{\mathcal{Z}^{\textsc{r}}_{#1}}}
\def\GL{{\rm GL}}
\def\SL{{\rm SL}}
\def\TD{{\T^2}}
\def\oo{\circ}
\def\eheta{{\varepsilon}}
\newcommand{\peheta}{\eheta}
\newcommand{\ehet}{{\widehat\varepsilon}}
\newcommand{\pehet}{\ehet}
\newcommand{\mCD}{\operatorname{mC}^{\T^2}}
\newcommand{\mCm}{\operatorname{mC}^{\T^2}}
\newcommand{\mClocD}{\operatorname{mC}_{\rm loc}^{\T^2}}
\def\mClocDL#1#2{{\operatorname{mC}_{\rm loc}^{\T^2}(\ZL{#1},#2)}}
\def\mClocDR#1#2{{\operatorname{mC}_{\rm loc}^{\T^2}(\ZR{#1},#2)}}
\newcommand{\dv}{\operatorname{div}}
\newcommand{\End}{\operatorname{End}}

\title[Hecke action on twisted motivic Chern classes]
{Hecke algebra action on twisted motivic Chern classes and K-theoretic stable envelopes}
\author{Jakub Koncki}
\address{Institute of Mathematics, Polish Academy of Sciences, Poland}
\address{Institute of Mathematics, University of Warsaw, Poland}
\email{j.koncki@mimuw.edu.pl}

\author{Andrzej Weber}
\address{Institute of Mathematics, University of Warsaw, Poland}
\email{aweber@mimuw.edu.pl}
\thanks{
	AW is supported by the National Science Centre (Poland) grant 2022/47/B/ST1/01896,
	JK is supported by the National Science Centre (Poland) grant 2023/48/C/ST1/00002.
	Data sharing is not applicable to this article as no datasets were generated or analysed during the current study. The authors have no conflicts of interest to declare that are relevant to the content of this article. We would like to thank the anonymous referee for helpful comments.}
\begin{document}

\begin{abstract} Let $G$ be a linear semisimple algebraic group and $B$ its Borel subgroup. Let $\T\subset B$ be the maximal torus. 
We study the inductive construction of Bott-Samelson varieties to obtain recursive formulas for the twisted motivic Chern classes of Schubert cells in $G/B$. To this end we introduce two families of operators acting on the equivariant K-theory $\KTh_\T(G/B)[y]$, the right and left Demazure-Lusztig operators depending on a parameter. 
The twisted motivic Chern classes coincide (up to normalization) with the K-theoretic stable envelopes. 
Our results imply wall-crossing formulas for a change of the weight chamber and slope parameters.
The right and left operators generate a twisted double Hecke algebra. We show that in the type $A$ this algebra acts on the Laurent polynomials. This action is a natural lift of the action on $\KTh_\T(G/B)[y]$ with respect to the Kirwan map. We show that the left and right twisted Demazure-Lusztig operators provide a recursion for twisted motivic Chern classes of matrix Schubert varieties.  
\end{abstract}
\maketitle	

\section{Introduction}
Schubert varieties and their cohomological invariants are important objects of enumerative geometry.
These varieties are usually singular, yet they admit a well-studied resolution of singularities called the Bott-Samelson resolution.
In many cases the inductive construction of Bott-Samelson varieties gives rise to  recursive formulas for various cohomological  invariants. These formulas allow us to compute the class of a Schubert cell from classes of smaller cells.

The study of inductive properties of various characteristic classes of Schubert varieties is widely presented in the literature. Starting from formulas for fundamental classes in cohomology \cite{BGG} or in K-theory \cite{LascouxSchutz, KostantKumar} the recursion based on word length became a standard feature of cohomological study of homogenous varieties. Among further important contributions we mention \cite{Brion2, Kn} for fundamental classes, \cite{AM,MNS} for $c_{SM}$ classes, \cite{AMSS,MNS, AMSwhit} for motivic Chern classes, \cite{AMSSnew} for Hirzebruch-Todd classes, \cite{SZZ,SZZ2} for stable envelopes, \cite{RW,KRW} for elliptic classes and \cite{MNS} for classes in the quantum cohomology. Most of the mentioned results are nicely reviewed in \cite{MNS}. In \cite{RWLag} the study of such recurrences is used to observe an instance of mirror symmetry. In this paper we study the inductive properties of the twisted motivic Chern class.

The motivic Chern class $\operatorname{mC}$ was defined in \cite{BSY}. It generalizes several other classes such as Chern-Schwartz-MacPherson $c_{SM}$ class \cite{CSM}, Baum-Fulton-MacPherson Todd class \cite{BFM}, or L-class (see \cite{SYp} for a broad survey). It may be thought of as a relative, K-theoretic version of the Hirzebruch-Todd genus \cite{chiy}, which is also defined for singular varieties. The equivariant versions are due to \cite{Oh2,WeHir,FRW,AMSS}. In \cite{KonW} the twisted version of the motivic Chern class was defined. It takes into account a chosen fractional line bundle. The motivic Chern class (and its twisted version) of a locally closed subvariety $X\subset M$ can be explicitly computed in terms of a resolution of singularities of the closure of $X$.

We consider semisimple, simply connected algebraic group $G$ with a chosen Borel subgroup $B$ and a maximal torus $\T$. The rational Picard group $\Pic(G/B)\otimes\Q$ is isomorphic to the space of rational characters $\Hom(\T,\C^*)\otimes\Q$. We study the twisted motivic Chern class $\mC(w,\lambda)=\mC(X_w,\partial X_w;\LL(\lambda))$ of the Schubert variety $X_w$ for a Weyl group element $w$ and a fractional character $\lambda$ (see Section~\ref{s:mC} for a definition). To compute such classes we use the Bott-Samelson resolution. 
It assigns to a reduced word decomposition $\wu$ of a Weyl group element $w$ a resolution of singularities of the Schubert variety $X_w$
$$p_\wu:Z_\wu \to X_w \subset G/B\,.$$
Let $s$ be a simple reflection, $\alpha_s$ the corresponding simple root and $P_s$ the corresponding minimal parabolic subgroup. Suppose that $w$ is a Weyl group element such that $ws$ is longer than $w$. The Bott-Samelson variety $Z_\ws$ can be constructed inductively as a fiber product
$$
\xymatrix{
Z_\ws \ar[rr]^-{p_\ws} \ar[d]&&G/B\ar[d]\\
Z_\wu\ar[r]^-{p_\wu}&G/B\ar[r]&G/P_s
}
$$

This construction induces a recursive formula for the twisted class. We define the twisted Demazure-Lusztig operator

$$\DL{s,\lambda}=\frac{1+y\Ls^*}{1-\Ls}\cdot s^{\textsc r} -\frac{(1+y)\cdot\Ls^{\lceil -\langle\lambda,\alpha_s^\vee\rangle\rceil}}{1-\Ls} \cdot \id_{\KTh_\T(G/B)[y]} \,,$$
where  $\Ls$ is the relative tangent bundle of the projection $G/B\to G/P_s$ and $s^{\textsc r}$ denotes the right Weyl group action of $s$ (see Section \ref{s:Weyl}).
Our first result is the following theorem.
\begin{atw*}[\ref{rem:tw1}] 
	Let $w \in (G/B)^\T\simeq W$ be a fixed point and $\lambda$ a general enough  fractional character. Then
	$$(-y)^{\frac{1}{2}(l(w)+1-l(ws))}\cdot\mC(ws,s\lambda)=\DL{s,\lambda}(\mC(w,\lambda))\,.$$
\end{atw*}
We obtain also a left counterpart of the above recursive formula. We define the left twisted Demazure-Lusztig operator
$$\DLL {s,\lambda}:=
\frac{1+y\alpha^{-1}_s}{1-\alpha^{-1}_s}\cdot s^{\textsc l} -\frac{(1+y)\cdot\alpha_s^{-\lceil \langle \lambda,\alpha^\vee_s\rangle\rceil}}{1-\alpha^{-1}_s} \cdot \id_{\KTh_\T(G/B)[y]}\,.$$
We use the recursive formula \ref{rem:tw1} to  obtain the following theorem.
\begin{atw*}[\ref{tw:L1}] 
	Let $w \in  (G/B)^\T\simeq W$ be a fixed point and $\lambda$ a general enough  fractional character. Then
	$$(-y)^{\frac{1}{2}(l(w)+1-l(sw))}\mC(sw,\lambda)=\DLL {s,w\lambda}(\mC(w,\lambda))\,.$$
\end{atw*}

One of the important features of the twisted motivic Chern class  is its connection with the K-theoretic stable envelope. Stable envelopes are characteristic classes defined initially for symplectic resolutions in three types: cohomological \cite{OM}, K-theoretic \cite{O2,OS} and elliptic \cite{OA}. Their definition is still evolving see e.g. \cite{O3} for recent progress.
In this paper we consider only the K-theoretic stable envelopes. They depend on a fractional line bundle called a slope.

Stable envelopes for the cotangent variety of a homogeneous space were extensively studied see e.g \cite{Su15,RTV0,RTV,RSVZ, SZsurv}.
They are tightly connected to the characteristic classes mentioned earlier. See \cite{FR,RV,AMSS0} for a comparison with the $c_{SM}$ class in cohomology. In the K-theory, a comparison between the stable envelope for a small anti-ample slope and the motivic Chern class was carried out in \cite{AMSS,FRW,Kon2}. The twisted class of \cite{KonW} is a generalization of the motivic Chern class which agrees with the stable envelope for a general enough slope.
It allows us to define the stable envelope in terms of a resolution of singularities of the Schubert variety.
Our recursive formulas may be restated
 in the language of stable envelopes. \begin{cor*}[\ref{cor:stabind}, \ref{cor:stabind2}] 
 	Let $w \in  (G/B)^\T\simeq W$ be a fixed point. Consider a  general enough  fractional character $\lambda$. Then
 	\begin{align*}
 		\notag&q^{1/2}\cdot\St^{s\lambda}(ws)=\DLq{s,\lambda}(\St^\lambda(w))\,, \\
 		&q^{1/2}\cdot\St^{\lambda}(sw)=\DLL {s,w\lambda}(\St^\lambda(w))\,.
 	\end{align*}
 	where $\DLq{s,\lambda}$ and $\DLLq {s,w\lambda}$ denote operators $\DL{s,\lambda}$ and $\DLL {s,w\lambda}$ after the substitution $y=-q$, respectively.
 \end{cor*}
\begin{cor*}[{\ref{cor:stabind3}}] 
	Let $w \in  (G/B)^\T\simeq W$ be a fixed point. Suppose that $ws$ is longer than $w$. Consider an arbitrary  fractional character $\lambda$ and a small anti-ample character $\lambda^-$. For a small enough positive real number $\varepsilon$ we have
	$$
	q^{1/2}\cdot\St^{s\lambda+\varepsilon\lambda^-}(ws)=\DLq{s,\lambda}(\St^{\lambda+\varepsilon\lambda^-}(w))\,.
	$$
\end{cor*}
The best known
recursive formulas computing the stable envelope for $G/B$ concern only specific values of the slope parameter (see e.g. \cite[Proposition 3.3 and Theorem~3.5]{SZZ} and \cite[Theorem 5.4]{SZZ2}). Our formulas generalize these results and work for an arbitrary slope.

One of the important notions in the theory of K-theoretic stable envelopes are R-matrices (also called wall-crossing formulas).
They describe the behavior of the stable envelope for elementary modifications of parameters (weight chamber and slope).
They were studied in e.g. \cite{OS,SZZ2,GorNeg,SmirHilb2}. 
In \cite{SZZ2} the slope R-matrix was fully computed for generalized flag varieties.

Our recursive formulas allow for an alternative, geometric approach to the wall-crossing problem based on the study of the Bott-Samelson resolution.
The right recursion leads to the wall-crossing formula related to a change of the slope (Theorem \ref{tw:2}), which is equivalent to the results of \cite{SZZ2}. The left induction implies the wall-crossing formula related to a change of the weight chamber (Theorem \ref{tw:wallL}).

Recursive formulas in the K-theory surprisingly rely on simple algebraic operators involving characters and the classes of line bundles. It turns out that one can build an algebra generated by natural lifts of operators $\DL{s,\lambda}$ and $\DLL{s,\lambda}$, denoted by $\Tcr{s}{\lambda}$ and $\Tcl{s}{\lambda}$, or their variants depending on a scalar $\Tfr{s}{a}$ and $\Tfl{s}{a}$. That algebra acts on the representation ring $R(\T^2)$ extended by the formal variable $y$. Practically this means, that to check identities in this algebra it is enough to perform calculus involving Laurent polynomials in two sets of variables and $y$. This algebra formally resembles the construction of Hecke algebra of Ginzburg-Kapranov-Vasserot, but there is a difference: the algebra of \cite{GKV} acts on $R(\T\times\T^\vee)$, where $\T^\vee$ is the dual torus. On the other hand we have a dependence on $\lambda$, which means that our operators can be treated as functions on $\ttt^*$. 
It is remarkable that the braid relations hold for the lifted operators. Braid relations can be deduced from the properties of the motivic Chern classes or simply checked  by hand performing elementary transformation of rational functions. Our twisted Hecke algebra is presented in full generality in the last section. Before, starting from Section \ref{sec:upgrade} we analyze the $A_n$ case. The braid relations have the form (satisfied by left and right operators) 
$$\Tcc{i}{s_{i+1}s_{i}\lambda}\circ \Tcc{i+1}{s_i \lambda}\circ \Tcc{i}{\lambda}= 
 \Tcc{i+1}{s_{i}s_{i+1}\lambda}\circ \Tcc{i}{s_{i+1} \lambda}\circ \Tcc{i+1}{\lambda}$$
or
$$\Tf{i}{a}\circ \Tf{i+1}{a+b}\circ \Tf{i}{b}=  \Tf{i+1}{b}\circ \Tf{i}{a+b}\circ \Tf{i+1}{a}\,,$$
see Lemma \ref{universal_braid} and Propositions \ref{uniwersalny_braid}, \ref{universalny_braid_l}.
The quadratic relations hold:
if $\langle \lambda ,\alpha_i^\vee\rangle\not\in\Z$, $a\not\in \Z$  
$$\Tcl{i}{s_i\lambda}\circ \Tcl{i}{\lambda}=-y \,\id\,,$$
$$\Tfl{i}{-a}\circ \Tfl{i}{a}=-y \,\id\,.$$
For integer values of the parameter $a$ the relation is different, we have
$$\Tfl{i}{1-a}\circ \Tfl{i}{a}=-y \,\id\,,$$
see Lemma \ref{quadratic}.
Dependence on the parameter is not a surprise. Felder in \cite{Felder}, see also \cite{EitingofVarchenko}, stated braid relations in their incarnation of the Yang-Baxter equation for R-matrices
$$
R^{(12)}(z)
R^{(13)}(z+w)
R^{(23)}(w)
=
R^{(23)}(w)
R^{(13)}(z+w)
R^{(12)}(z)\,.$$
\medskip

A natural question arises: How the operators $\Tcl{s}{\lambda}$ and $\Tcr{s}{\lambda}$ are related to geometry? What do they compute? We give an answer in the $A_n$ case. We prove that after a suitable normalization the operators $\Tcl{s}{\lambda}$ and $\Tcr{s}{\lambda}$ provide recursive formulas for the twisted motivic Chern classes of matrix Schubert varieties $\XX_w$. We consider only the maximal rank matrix Schubert varieties in $\Hom(\C^n,\C^n)$, i.e. $B\times B$-orbits of the permutation matrices. We prove     
\begin{atw*}[\ref{indukcja_macierzowa}] The left and right recursions hold
$$\Tcl{i}{w\lambda}\big(\BB^{-1} \mCD(\XX_w,\partial \XX_w;D_{w,\lambda})\big)=\BB^{-1}\mCD(\XX_{s_iw},\partial \XX_{s_iw};D_{s_iw,\lambda})\quad \text{if }~~l(s_iw)>l(w)\,,$$
$$\Tcr{i}{\lambda}\big(\BB^{-1} \mCD(\XX_w,\partial \XX_w;D_{w,\lambda})\big)=\BB^{-1}\mCD(\XX_{ws_i},\partial \XX_{ws_i};D_{ws_i,s_i\lambda})\quad \text{if }~~l(ws_i)>l(w)\,.$$
\end{atw*}
Here $\BB$ is a certain fixed rational function (see the formula \eqref{eq:BB}) and $D_{w,\lambda}$ is the distinguished divisor contained in the boundary of the matrix Schubert variety (Definition \ref{dywizor}).
The factor $\BB$ is responsible for the motivic Chern class of the fiber of the projection $\GL_n\to \GL_n/B$, and it is present in a similar formula for Rim\'anyi-Tarasov-Varchenko weight function, \cite[Section 6.1]{RTV0}. For the proof we apply two resolutions of matrix Schubert varieties: the left resolution and the right resolution, giving rise to the left and right Demazure-Lusztig operators.

\tableofcontents
\section{Notations}
All considered varieties are complex and quasi-projective. We consider an algebraic torus \hbox{$\T\simeq(\C^*)^{\rank \T}$.}

\subsection{Line bundles and divisors}Let $X$ be a quasiprojective $\T$-variety. An equivariant line bundle is a line bundle  $\LL\to X$ together with a linearization, i.e. a lift of the $\T$-action to $\LL$ which is linear on fibers. For any equivariant line bundle $\LL\to X$ and a fixed point $x\in X^\T$, the fiber ${\LL}_{|x}$ is a representation of the torus $\T$.

Suppose that $X$ is a smooth $\T$-variety, $x\in X^\T$ a fixed point and $D$ a $\T$-invariant codimension one subvariety. The line bundle $\O_X(D)$ (treated as a subsheaf of meromorphic functions) has the natural linearization such that the weight of the representation $\O_X(D)_{|x}$ is:
\begin{itemize}
	\item  trivial when $x\notin D$,
	\item  equal to the normal weight to $D$ at $x$ when $x$ is a smooth point of $D$. 
\end{itemize}
Suppose now that $D=\sum_{i=1}^{n} a_iD_i$ is an arbitrary $\T$-invariant divisor on $X$. The isomorphism
$$\O_X(D) \simeq \bigotimes_{i=1}^n \O_X(D_i)^{a_i}$$
induces the natural linearization of the bundle $\O_X(D)$.
\subsection{Round-up divisor}
\begin{adf}[{\cite[Definition 9.1.2]{multiplier}}]
	Let $D=\sum q_iD_i$ be a $\Q$-divisor. The round-up divisor $\lceil D\rceil$ is given by
	$$\lceil D\rceil=\sum\lceil q_i\rceil D_i \,.$$
\end{adf}
In general rounding-up does not commute with pullbacks (cf. \cite[Remark 9.1.4]{multiplier}). The following proposition shows commutation in a special case. 
\begin{pro} \label{pro:round}
	Let $X$ be a smooth variety and $\pi:Y\to X$ a smooth morphism (for example a $\PP^1$-bundle). Let $D$  be a~$\Q$-divisor on $X$. Then
	$$\lceil \pi^* D\rceil=\pi^*\lceil D\rceil \,.$$
\end{pro}
The proof is straightforward.
\subsection{Equivariant K-theory}
For a quasiprojective  $\T$-variety $X$ we consider the equivariant K-theory of coherent sheaves $\operatorname{G}^\T(X)$ and the equivariant K-theory of vector bundles~$\KTh_\T(X)$. For a smooth $\T$-variety $X$ these two notions coincide i.e. we have a canonical isomorphism
$$\KTh_\T(X) \simeq  \operatorname{G}^\T(X)\,,$$
induced by taking the sheaf of sections.
\\
The equivariant K-theory of a point is isomorphic to the ring of Laurent polynomials
$$\KTh_\T(pt) \simeq \Z[\Hom(\T,\C^*)] \simeq \Z[t_1^\pm,\dots,t_{\rank \T}^\pm]\,.$$
Let $S\subset \KTh_\T(pt)$ be the multiplicative system consisting of all nonzero elements. By the localized K-theory of $X$ we denote the ring $S^{-1}\KTh_\T(X)$. The localization theorem \cite[Theorem 2.1]{Tho} implies that the restriction map induces an isomorphism
$$S^{-1}\KTh_\T(X) \simeq S^{-1}\KTh_\T(X^\T)\,.$$
Suppose that a $\T$-variety $X$ is smooth and projective. Suppose moreover that the fixed point set $X^\T$ is finite. Then the equivariant K-theory of $X$ is a free $\KTh_\T(pt)$-module\footnote{The proof is fairly easy induction on Bia{\l}ynicki-Birula skeleta, as in \cite[Section 9]{FRW}. For the structure of the K-theory of homogeneous spaces see \cite{KostantKumar} or \cite{Uma}.} 
In this case we have an inclusion of rings 
$$\KTh_\T(X) \subset S^{-1}\KTh_\T(X) \simeq S^{-1}\KTh_\T(X^\T)$$
and
$$\KTh_\T(X^\T)\simeq \bigoplus_{x\in X^\T}\Z[t_1^\pm,\dots,t_{\rank \T}^\pm]\,.$$
For an equivariant complex vector bundle $E$ over a $\T$-variety $X$ we define $\lambda_y$ class as
$$\lambda_{y}(E):=\sum_{k=0}^{{\rm rk}(E)}y^k[\Lambda^k E] \in \KTh_\T(X)[y]\,.$$
The K-theoretic Euler class of $E$ is defined as
$$\eu(E):=\lambda_{-1}(E^*)=\sum_{k=0}^{{\rm rk}(E)}(-1)^k[\Lambda^k E^*] \in \KTh_\T(X)\,.$$
The Lefschetz-Riemann-Roch formula (LRR for short) allows us to compute push-forwards using restriction to the fixed point set. Let us remind it here.
\begin{atw}[{\cite[Theorem 3.5]{Tho} and \cite[Theorem 5.11.7]{ChGi}}] \label{tw:LRR}
	Let $X$ and $Y$ be smooth $\T$-varieties. Let $p:X \to Y$ be an equivariant proper morphism. Consider a component of the fixed point set $F \subset Y^{\T}$. Let $I$ be the set of components of the fixed point set $p^{-1}(F)^\T$.
	For $\alpha \in S^{-1}\KTh_\T(X)$ we have
	$$
	\frac{(p_*\alpha)_{|F}}{\eu(\nu(F\subset Y))}=
	\sum_{F'\in I}(p_{F'})_*\frac{\alpha_{|F'}}{\eu(\nu(F'\subset X))} \in S^{-1}\KTh_\T(F)\,,
	$$
where $\nu(F\subset Y)$ is the normal bundle and $p_{F'}\colon F'\to F$ is the restriction of $p$. 
\end{atw}
We will apply the LRR theorem for the Bott-Samelson resolution, where the fixed point sets are finite, hence the pushforward maps $(p_{F'})^*$ are isomorphisms. Due to the identification $F'\simeq F\simeq pt$, they can be treated as identities.

\section{Generalized flag varieties}
We summarize the facts from Lie theory and about Coxeter groups which we will further use.
\subsection{Notation and assumptions}\label{notation}
\begin{itemize}
	\item $G$ is a connected, semisimple, simply connected complex Lie group 
	with a chosen Borel subgroup $B$ and a maximal torus $\T$, such that
	$\T\subset B\subset G\,.$
	\item $W$ is the Weyl group of $G$. We use the identification $W\simeq(G/B)^\T$.
	\item For a Weyl group element $w\in W$ let $X_w^\oo \subseteq G/B$ be its $B$-orbit. We call this orbit the Schubert cell of $w$. The Schubert variety
	$X_w$ is the closure of the Schubert cell~$X_w^\oo $ in $G/B$. Denote by $\partial X_w$ the boundary
	$$\partial X_w=X_w \setminus X_w^\oo  \,. $$
	\item $\iota_w$ denotes the inclusion $\iota_w\colon X_w \mono G/B\,.$
	\item $\Hom(\T,\C^*)$ is the group of characters.
	\item $\ttt$ is the real part of Lie algebra of the torus $\T$, and
	$$\ttt^*\simeq \Hom(\T,\C^*)\otimes_\Z\R\,. $$
	The Weyl group acts on $\ttt^*$ from the left.	
	\item We fix a $W$-invariant inner product $\langle- ,- \rangle$ on $\ttt^*$.
	\item For a root $\alpha \in \ttt^*$ of $G$ let
	$$ \alpha^\vee=\frac{2\alpha}{\langle\alpha,\alpha\rangle}$$
	be the dual root. If a root appears as an element of $\ttt^*$ then we use the additive notation. If a root appears as an element of the representation ring $R(\T)$, then the multiplicative notation is applied. For example the dual of the line representation given by the root $\alpha$ is denoted $\alpha^{-1}$, which in some sources would be denoted by $e^{-\alpha}$.
	\item We call a reflection $s\in W$ simple if it corresponds to a simple root.
	\item For $w\in W$ let $l(w)$ be its length. It is equal to the dimension of the Schubert variety of $w$, i.e.
	$$\dim(X_w)=l(w)\,.$$
	\item For a simple reflection $s \in W$ we consider the corresponding minimal parabolic subgroup $P_s$ containing $B$ and a lift of $s$ to $N\T$. Let
	$$\pi_s:G/B \to G/P_s $$ be the projection
	\item For a character $\lambda \in \Hom(\T,\C^*)\simeq\Hom(B,\C^*)$ let $\LL(\lambda) \in \Pic(G/B)$ be the line bundle of the form 
	$$\LL(\lambda) \simeq G\times_B \C_{-\lambda}\,. $$ 
	\item The bundle $\LL(\lambda)$ has a natural linearization such that for a fixed point $w\in  (G/B)^\T\simeq W$ we have $$\LL(\lambda)_{|w}=- w \lambda \in \Hom(\T,\C^*)\,.$$ 
	\item For a $\Q$-character $\lambda \in \Hom(\T,\C^*)\otimes_\Z \Q$ let $\LL(\lambda)\in \Pic(G/B)\otimes_\Z\Q$ be the corresponding fractional line bundle. This assignment induces an isomorphism 
	$$\Hom(\T,\C^*)\otimes\Q \simeq \Pic(G/B)\otimes\Q\,. $$
\item We say that $\lambda\in\ttt^*$ is general enough if $\langle \lambda,\alpha^\vee\rangle\not\in\Z$ for all roots $\alpha$. 
\item We say that $\lambda\in\ttt^*$ is anti-ample if $\langle \lambda,\alpha^\vee\rangle<0$ for all simple roots $\alpha$.

\end{itemize}
\subsection{Weyl group actions} \label{s:Weyl}
The Weyl group acts on the maximal torus $\T$ by conjugation. This induces a left Weyl group action on the characters $\Hom(\T,\C^*)$. The action extends naturally to an action on fractional characters and the equivariant K-theory of a point
$$\KTh_\T(pt) \simeq \Z[\Hom(\T,\C^*)]\,.$$

Consider a generalized flag variety $G/B$. There are two actions (the left and the right action) of the Weyl group on the equivariant K-theory $\KTh_\T(G/B)$.
\begin{itemize}
	\item The normalizer of the torus $N\T$ acts on $G/B$ by the left multiplication. This induces the left action of $W=N\T/\T$ on $\KTh_\T(G/B)$. This action does not preserve the torus action. It is invariant when twisted by the torus automorphism given by the $W$-action.
	\item The right action is obtained using the identification
	$G/B\simeq K/T$, where $K\subset G$ is a maximal compact group and $T\subset K$ is a maximal compact torus. We can assume that $T\subset \T$. The Weyl group $W=NT/T$ acts on $K/T$ by the right multiplication. This induces the right action of $W$ on $\KTh_\T(G/B)$, by the pullback. This action is torus equivariant, but it is not holomorphic.
\end{itemize}

 See \cite[Section 3.1]{MNS} or \cite{Kn} for the definition and a detailed discussion of these actions. For an element $w\in W$ we denote the left action of $w$ by $w^{\textsc l}$ and the right one by $w^{\textsc r}$. The weights of the restrictions to the fixed point set are related by the following formula.
\begin{pro}
	For $w,\sigma\in W$ and $x\in \KTh_\T(G/B)$ we have
	\begin{align*}
		w^{\textsc r}(x)_{|\sigma}&=x_{|\sigma w}\,, \\
		w^{\textsc l}(x)_{|\sigma}&=w(x_{|w^{-1}\sigma })\,.
	\end{align*}
\end{pro}
\begin{rem} The action of $W$ on $\KTh_\T(G/B)$ via $w^{\textsc r}$ satisfies $(uw)^{\textsc r}=u^{\textsc r}w^{\textsc r}$ i.e.~is an action from the left, despite its name.
\end{rem}
\subsection{Reflections}

The Weyl group can be identified with the group generated by the reflections in the simple roots of the Lie algebra. We record the following proposition, which is a direct consequence of the properties of root systems (cf. \cite[Proposition in Section~1.2]{Hump}).

\begin{pro} \label{lem:Chev}
	Let $\alpha$ and $\beta$ be roots and $s_\alpha,s_\beta$ the corresponding reflections. Let $w\in W$ be a Weyl group
	element. Then
	$$s_\beta=ws_\alpha w^{-1}  \iff \beta=\pm w\alpha \in \ttt^*\,.$$
\end{pro}

We will need further properties of the Weyl group. 
\begin{pro} \label{lem:Weyl}
	 Let $s \in W$ be any reflection. There exists $w\in W$ such that $wsw^{-1}$ is a simple reflection.
\end{pro}
\begin{proof}
	Let $\alpha$ be the positive root corresponding to $s$. By Proposition \ref{lem:Chev} it is enough to show that there exists $w\in W$ such that $w \alpha$ is a simple root. This is a standard fact, see e.g. \cite[Proposition 5.12]{CarLieAlg}. 
\end{proof}

\begin{pro}  \label{lem:Wall1}
	Suppose $q_1,q_2,q_3\in W$ are reflections such that $q_3=q_1q_2q_1$, $q_1$ is simple and $q_1\neq q_2$. Let $w\in W$ be any Weyl group element. Then
	$$l(wq_1)<l(wq_1q_2) \iff l(w)<l(wq_3)\,.$$
\end{pro}
\begin{proof}
First let us note that $l(w)<l(wq_i)$ if and only if the weight of the tangent space of the one dimensional $\T$-orbit joining $w$ with $wq_i$ at $w$ is a negative root.
	Let $\alpha_2,\alpha_3$ be positive roots corresponding to reflections $q_2$ and $q_3$, respectively. The reflection $q_1$ is simple and $q_1\neq q_2$, therefore $q_1\alpha_2$ is also a positive root (see e.g. \cite[Lemma 5.9]{CarLieAlg}). Thus $q_1\alpha_2=\alpha_3$ (cf. Proposition \ref{lem:Chev}).
	The tangent weight at $w$ to the one-dimensional $\T$-orbit connecting $wq_3$ and $w$ is equal to $-w\alpha_3$. The tangent weight at $wq_1$ to the one-dimensional $\T$-orbit connecting $wq_1q_2$ and $wq_1$ is equal to $-wq_1\alpha_2$. These two weights are equal which proves the lemma.
\end{proof}
\begin{pro} \label{pro:walpha}
	Let $s$ be a simple reflection, $\alpha_s$ the corresponding simple root and $w\in W$ a Weyl group element.
	\begin{itemize}
		\item Suppose that $l(w)<l(ws)$ then $w\alpha_s$ is a positive root.
		\item Suppose that $l(w)<l(sw)$ then $w^{-1}\alpha_s$ is a positive root.
	\end{itemize}
\end{pro}
\begin{proof}
	The first part follows from the fact that $-w\alpha_s$ is the tangent weight at $w$ of the orbit connecting $w$ and $ws$. The second part follows from the first since $l(w)<l(sw)$ if and only if $l(w^{-1})<l(w^{-1}s)$.
\end{proof}

 \section{Bott-Samelson resolution}\label{BottSa_resolution}
 The construction of the Demazure-Lusztig operators are based on the geometry of Bott-Samelson resolution. Below we briefly summarize what we need.
 \subsection{Inductive construction}
 Let $\wu$
 be a reduced word decomposition of $w\in W$. We denote by
 $$p_\wu:Z_\wu \to X_w \subset G/B $$
 the corresponding Bott--Samelson resolution of singularities (see e.g. \cite[Section~2.2]{BK},
 \cite[Section 2]{AMSS} or \cite[Section 3]{RW}).
 Let
 $$\partial Z_\wu =p_{\wu}^{-1}(\partial X_w), \qquad Z^\oo _\wu=Z_\wu\setminus \partial Z_\wu \,.$$ 
 It is a standard fact that
 $$p_{\wu}:(Z_\wu,\partial Z_\wu) \to (X_w, \partial X_w)\,.$$
 is a simple normal crossing resolution of singularities. 
 The boundary $\partial Z_\wu$ consists of~$l(w)$ components corresponding to omitting a single letter in $\wu$. Denote these components by~$\partial_j Z_\wu$ for $j\in\{1,2,\dots ,l(w)\}$.
 
 The Bott--Samelson resolution may be constructed inductively. Consider a simple reflection $s$. Suppose that $l(ws)>l(w)$. Then we have a pullback diagram
 \begin{align}
 \xymatrix{
 	Z_\ws \ar[rr]^{p_\ws} 
 	\ar[d]^{\pi_\ws}
 	&&X_{ws} \ar[r]^{\iota_{ws}} 
 	& G/B \ar[d]^{\pi_s} \\
 	Z_\wu \ar[r]^{p_\wu}&X_{w}\ar[r]^{\iota_{w}}&
 	G/B \ar[r]^{\pi_s}&
 	 G/P_s
 }
 \end{align}
 The map $\pi_\ws: Z_\ws \to Z_\wu$ has a $\T$-invariant section $i:Z_\wu \to Z_\ws$. We have an equality
  $$\partial Z_\ws=\pi_\ws^{-1}(\partial Z_\wu) \cup i(Z_\wu) \,.$$ 
  One may use the above construction to inductively describe the fixed point set $Z_\wu^\T$. It can be identified with the set of binary sequences of length~$l(w)$. For the empty word identification is obvious, the variety $Z_\varnothing$ is a single point. Suppose that $l(ws)>l(w)$. 
  Let $\eheta$ be a binary sequence of length~$l(w)$. By the inductive assumption it corresponds to a fixed point in $Z_\wu^\T$. The preimage $\pi^{-1}_{\ws}(\eheta)$ is a projective line $\PP^1$ equipped with a nontrivial action of $\T$. The set $\pi^{-1}_{\ws}(\eheta)^\T$ consists of two points. We identify these fixed points with $(\eheta,0)$ and $(\eheta, 1)$. We assume that $(\eheta,0)=i(\eheta)$ is the point that lies in the image of section $i$. For a detailed discussion of the fixed point set and an example see \cite[Section 3.2]{RW}.
  \subsection{Pull-backs of line bundles}
  Consider a fractional character $\lambda\in \Hom(\T,\C^*)\otimes_\Z \Q$. We are interested in the pullback $p^*_\wu\LL(\lambda)$ of the corresponding line bundle to the Bott--Samelson variety. Let us recall several classical results.
  \begin{pro}[Kempf lemma {\cite[Section 2, Lemma 3]{Kempf}}] \label{pro:kem}
  	Let $\wu$ be a reduced word decomposition of $w\in W$. Suppose that $s$ is the reflection corresponding to a simple root~$\alpha_s$. Assume that $l(ws)>l(w)$. Then, we have an isomorphism
  	$$
  	p^*_\ws\LL(s\lambda)\simeq \pi_\ws^*p_\wu^* \LL(\lambda) \otimes \O_{Z_{\ws}}(-\langle\lambda,\alpha^\vee_s\rangle\cdot i(Z_\wu)) \in \Pic(Z_\ws)\,.
  	$$
  \end{pro}
\begin{cor} \label{cor:Kem}
		The Kempf lemma holds equivariantly. Consider the natural  $\T$-linearization of bundle $\LL(\lambda)$ (i.e. such that the weight at $\id\in G/B$ is equal to $-\lambda$) and the natural linearization of $\O_{Z_{\ws}}(i(Z_\wu))$. We have an isomorphism
		$$
		p^*_\ws\LL(s\lambda)\simeq \pi_\ws^*p_\wu^* \LL(\lambda) \otimes \O_{Z_{\ws}}(-\langle\lambda,\alpha^\vee_s\rangle\cdot i(Z_\wu)) \in \Pic^\T(Z_\ws)\,.
		$$
	\end{cor}
\begin{proof}
	We only need to prove that $\T$-linearizations of both sides agree. Due to \cite[Proposition~2.10]{Brion} 
	it is enough to check that the weights at some chosen fixed point agree. On the maximal fixed point $p^{-1}_\ws(ws)$ both bundles have weight $-w\lambda.$
\end{proof} \label{pro:Chev0}
	\begin{pro}[Chevalley formula {\cite[Paragraph 4.3]{Demazure}}, see also {\cite[Proposition 4.1]{KRW}} for this formulation] 
		Let \hbox{$\wu=s_{i_1},s_{i_2},\dots,s_{i_{l(w)}}$} be a reduced word decomposition of $w\in W$. Let $\alpha_{j}$ be the simple root corresponding to the reflection $s_{i_j}$. Consider elements
		
		$$\gamma_j =s_{i_{l(w)}}s_{i_{l(w)-1}}\dots s_{i_{j+1}}\cdot \alpha_j \in \ttt^*\,. $$
		There is an isomorphism
		$$
		p^*_\wu\LL(\lambda)\simeq  \O_{Z_{\wu}}\Big(
		\sum_{j=1}^{l(w)}\langle\lambda,\gamma_j^\vee\rangle\cdot\partial_jZ_\wu
		\Big) \in \Pic(Z_\wu)\,.
		$$
	\end{pro}
	\begin{pro} \label{pro:Chev}
		The coefficients $\langle \lambda, \gamma_j^\vee\rangle$ from the Chevalley formula belong to the set
		$$ \{\langle\lambda,\alpha^\vee\rangle|\alpha \text{ is a root, } l(w)>l(ws_\alpha)\} \,,$$
		where $s_\alpha$ is the reflection corresponding to the root $\alpha$.
	\end{pro}
\begin{proof}
Let $s_{\gamma_j}$ be the reflection corresponding to the root ${\gamma_j}$. We need to prove that
	$$l(w)>l(ws_{\gamma_j})\,.$$
	Let $w_{>j}=s_{i_{j+1}}s_{i_{j+2}}\dots s_{i_{l(w)}}\in W$. We have $\gamma_j=w_{>j}^{-1}\cdot\alpha_{j}$. Proposition \ref{lem:Chev} implies that
	$$s_{\gamma_j}=w^{-1}_{>j}s_{i_j} w_{>j} \in W\,.$$
	It follows that $ws_{\gamma_j}$ is represented by the word $\wu$ with $s_{i_j}$ omitted. Therefore \hbox{$l(w)>l(ws_{\gamma_j})\,.$}
\end{proof}
\begin{rem}
	In Proposition \ref{pro:Chev} we do not assume that the roots are positive. Of course the set of coefficients may be much smaller, e.g. all appearing coefficients are nonnegative if $\lambda$ is dominant.
\end{rem}
\section{Twisted motivic Chern class}
The K-theoretic characteristic classes of singular varieties are our main protagonists. Below we recall the twisted version of motivic Chern classes.
\subsection{Motivic Chern class}
The motivic Chern class was defined in \cite{BSY} (see also \cite{SYp} for a survey). Its equivariant version is due to \cite{FRW,AMSS}. Here we remind only the definition of $\T$-equivariant motivic Chern class. Consult \cite{FRW,AMSS} for a detailed account.
\begin{adf}[after {\cite[Section 2.3]{FRW}}]
	The motivic Chern class assigns to every $\T$-equivariant map of quasi-projective $\T$-varieties $f:Z \to Y$ an element
	$$\mC(f)=\mC(Z \xtto{f} Y) \in \operatorname{G}^\T(Y)[y]\,,$$
	such that the following properties are satisfied
	\begin{description}
		\item[1. Additivity] Let $Z$ be a $\T$-variety and $U \subset Z$ an invariant open subvariety.  Then
		$$\mC(Z\xtto{f} Y)=\mC(U\xtto{f_{|U}} Y)+\mC(Z\setminus U\xtto{f_{|Z\setminus U}} Y)\,.$$
		\item[2. Functoriality] For an equivariant proper map $g:Y\to Y'$ we have
		$$\mC(Z\xtto{g\circ f} Y')=g_*\mC(Z\xtto{f} Y) \,.$$
		\item[3. Normalization] For a smooth $\T$-variety $X$ we have
		$$\mC(\id_X)=\lambda_y(T^*X)
		\,.$$ 
	\end{description}
	The equivariant motivic Chern class is the unique assignment satisfying the above properties. For a smooth $\T$-variety $X$ we may consider the class $\mC(Z\to X)$ as an element of $\KTh_\T(X)[y]$.
\end{adf}
See \cite[Theorem 4.2]{AMSS} for an equivalent definition in terms of a natural transformation of functors. The above definition is meaningful also in the non-equivariant setting, i.e. for $\T$ equal to the trivial group (cf. \cite{BSY}). 
\subsection{Twisted class}
The twisted motivic Chern class was defined in \cite{KonW}. We repeat its definition here.
Let $(Y,\partial Y)$ be a pair, consisting of an algebraic quasiprojective $\T$-variety $Y$ and an invariant closed subvariety $\partial Y \subset Y$. We assume that the complement $Y^\oo =Y\setminus \partial Y$ is smooth.
Let $\Delta$ be an invariant $\Q$-Cartier divisor on $Y$ with support contained in the boundary $|\Delta|\subset \partial Y$. 

\begin{adf}[{\cite[Definition 2.2]{KonW}}]\label{df:twistedmC} 
	The twisted motivic Chern class of a triple $(Y,\partial Y; \Delta)$ as above is defined by the formula
	$$\mC(Y,\partial Y;\Delta)=f_*\Big(\O_Z(\lceil f^*(\Delta)\rceil) \cdot \mC(Z^\oo \subset Z)\Big) \,,$$
	where  $f:(Z,\partial Z)\to (Y,\partial Y)$ is a resolution of singularities such that $\partial Z=f^{-1}(\partial Y)$ is a simple normal crossing divisor and $Z^\oo =Z\setminus \partial Z$.
	We assume that
	$f_{|Z^\oo }:Z^\oo \to Y^\oo $
	is an isomorphism.
	\end{adf}
	It is proved in \cite[Corollary 4.6]{KonW}
	that this class is well defined, i.e. it does not depend on a choice of a SNC resolution of singularities.
	\begin{rem}
		The twisted motivic Chern class may be obtained as a limit of the equivariant version of Borisov-Libgober elliptic class \cite{BoLi1}, whenever it is defined, see \cite[Section 3]{KonW} for a detailed discussion. 
	\end{rem}
	\subsection{Twisted classes of Schubert cells and stable envelopes} \label{s:mC}
	The K-theoretic stable envelopes were defined in \cite{O2,OS}. They depend on a fractional line bundle called slope. It turns out that for homogeneous varieties the stable envelopes for a small anti-ample slope coincide (up to normalization) with the motivic Chern classes of Schubert varieties, see \cite[Theorem 8.5]{AMSS}, \cite{FRW} and \cite[Theorem 6.2]{Kon2}. The twisted class is a class depending on a fractional line bundle which generalizes the motivic Chern class and coincides with the stable envelope for an arbitrary slope, see \cite[Theorem~7.1]{KonW}. \medskip \\
	Consider a generalized flag variety $G/B$. 
	Let $\lambda \in \Hom(\T,\C^*)\otimes_\Z\Q$ be a fractional character and $w\in W$ a fixed point.
	\begin{pro}[{\cite[Section 10]{KonW}}]
		There exists a unique $\T$-invariant $\Q$-Cartier divisor $\Delta_{w,\lambda}$ on~$X_w$ such that
		\begin{enumerate}
			\item The divisor $\Delta_{w,\lambda}$ represents $\LL(\lambda)_{|X_w}$.
			\item The support of $\Delta_{w,\lambda}$ is contained in the boundary $\partial X_w$.
		\end{enumerate}
	\end{pro}
	\begin{adf}\label{mc-defndf}
		The twisted motivic Chern class of a Schubert cell is defined as
		\begin{equation}\label{mc-defn}\mC(w,\lambda):=\iota_{w*}\mC(X_w,\partial X_w; \Delta_{w,\lambda})\in \KTh_\T(G/B)[y]\,. \end{equation}
	\end{adf}
	\begin{rem}
		To compute the class $\mC(w,\lambda)$ one may use the Bott-Samelson resolution. The divisor $p_\wu^*\Delta_{w,\lambda}$ is the unique $\T$-invariant $\Q$-divisor representing $p_\wu^*\LL(\lambda)$ with support contained in $\partial Z_\wu$. Its coefficients are computed by the Chevalley formula (cf. Propositions~\ref{pro:Chev0} and \ref{pro:Chev}).
	\end{rem}
	\begin{rem}[{\cite[Proposition 2.5]{KonW}}]
		The twisted motivic Chern class generalizes the motivic Chern class, i.e. for a small anti-ample $\lambda$ we have
		$$\mC(w,\lambda)=\mC(w,0)=\mC(X_w^\oo \subset G/B) \,.$$
	\end{rem}
	\begin{atw}[{\cite[Theorem 7.1 and Remark 6.4]{KonW}}] \label{tw:otoczki}
		Suppose that a fractional character~$\lambda$ is general enough. Then the class
		$$q^{-\frac{1}{2} \dim X_w}\operatorname{mC}^\T_{-q}(w,\lambda) \in \KTh_\T(G/B)[q^{1/2},q^{-1/2}]$$
		is equal to the stable envelope
		$ \St^\lambda(w)\,.$
	\end{atw}
	\begin{rem}
		We write $\operatorname{mC}^\T_{-q}$ for the image of the class $\mC$ under the map
		$$\rho:\KTh_\T(G/B)[y] \to \KTh_\T(G/B)[q^{-1/2},q^{1/2}]\,,$$
		which sends $y$ to $-q$.
	\end{rem}
	\begin{cor} \label{cor:otoczki}
		Let $\lambda$ be an arbitrary rational character. Let $\lambda^-$ be a small anti-ample fractional character and
		$\varepsilon$ a small enough positive rational number. Then the class
		$$q^{-\frac{1}{2} \dim X_w}\operatorname{mC}^\T_{-q}(w,\lambda) \in \KTh_\T(G/B)[q^{1/2},q^{-1/2}]$$
		is equal to the stable envelope
		$ \St^{\lambda+\varepsilon\lambda^-}(w)\,.$
	\end{cor}
	\begin{rem}
		We do not include the definition of the stable envelope here. For a definition adapted to generalized flag varieties see e.g.~\cite{AMSS,SZZ,KonW}  
	\end{rem}
	\begin{rem} 
		We write $\operatorname{stab}_{\mathfrak{C}}^\lambda(w)$ for the stable envelope of the fixed point $w \in W$ for the weight chamber $\mathfrak{C}$, slope $\lambda$ and the tangent polarization.
		We consider stable envelopes only for 
		the tangent bundle $TG/B$ as a polarization. 
		This restriction does not reduce the generality of the results.
		Stable envelopes corresponding to different
		polarizations are related by a shift of slope and renormalization, see \cite[Paragraph 9.1.12]{O2}. 
		We usually consider stable envelopes for the positive weight chamber $\mathfrak{C^+}$.
	\end{rem}

\section{Twisted Demazure-Lusztig operators}
Many cohomological invariants of Schubert varieties can be computed inductively due to the inductive construction of the Bott-Samelson resolution of singularities. By application of the Demazure or the Bernstein-Gelfand-Gelfand operators, starting from the class of 0-dimensional Schubert variety, the fundamental class can be computed, \cite{BGG}. The same is true for the motivic Chern classes in K-theory as proven in \cite{AMSS}, where the Demazure-Lusztig operators are employed. The
twisted version demands a twisted version of operators. 
\subsection{The definitions and main properties}
Let $\alpha_s$ be a simple root and $s\in W$ the corresponding reflection. Denote by $\Ls$ the~relative tangent bundle of the projection $\pi_s:G/B\to G/P_s$. Let us recall the definition of the Demazure-Lusztig operator.
\begin{adf}
	The  Demazure-Lusztig operator $\DL{s}$ is defined by
	\begin{align*}
		\DL{s}&= \lambda_y(\Ls^*)\cdot \pi_s^*\pi_{s*} -\id_{\KTh_{\T}(G/B)[y]} \\
		&= \frac{1+y\Ls^*}{1-\Ls}\cdot s^{\textsc r}-\frac{(1+y)}{1-\Ls}\cdot \id_{\KTh_{\T}(G/B)[y]}\,.
	\end{align*}
They appeared already in Lusztig's article \cite{Lusztig} on the K-theory of flag varieties. In fact,  $\KTh_\T(G/B)[q,q^{-1}]$ is a free, rank one  module over the Hecke algebra, hence it admits both right and left Hecke action.  
The interplay with characteristic classes was exposed in \cite{AMSS}. Because of the relation with Hirzebruch $\chi_y$-genus we use the variable $y$ which is equal to $-q$ from Lusztig's work.
	The left Demazure-Lusztig $\DLL{s}$ operator is defined by (cf. \cite[Section 5.3]{MNS})
	\begin{align*}
		\DLL {s}:=&
		\frac{1+y\alpha^{-1}_s}{1-\alpha^{-1}_s}\cdot s^{\textsc l} -\frac{(1+y)}{1-\alpha^{-1}_s} \cdot \id_{\KTh_\T(G/B)[y]}\,.
	\end{align*}
\end{adf}
Note that here $\alpha_s^{-1}$ denotes the one-dimensional representation, the dual of $\alpha_s$. We use here the multiplicative notations for characters.

\begin{rem}For the flag varieties in type $A$ the  left Demazure-Lusztig operators appeared in the form of  R-matrix relations connecting weight functions \cite[Section 7.2]{RTV0}, \cite[Theorem 3.2]{RTV}. In \cite{FRW} it was shown, that the trigonometric weight functions coincide with the motivic Chern classes of Schubert cells.\end{rem}
 
We define a twisted version of the Demazure-Lusztig operators.
\begin{adf} \label{df:DL} \  
	\begin{enumerate}[\bf 1)]
		\item For $a\in \R$ we define the right twisted  Demazure-Lusztig operator 
		\begin{align*}
			\DL{s,a}:=&
			\frac{1+y\Ls^*}{1-\Ls}\cdot s^{\textsc r} -\frac{(1+y)\cdot\Ls^{\lceil a\rceil}}{1-\Ls} \cdot \id_{\KTh_\T(G/B)[y]}\\
			=&
			\frac{1+y\Ls^*}{1-\Ls}\cdot s^{\textsc r} +\frac{(1+y)\cdot\Ls^{\lceil a\rceil-1}}{1-\Ls^*} \cdot \id_{\KTh_\T(G/B)[y]}\,.
		\end{align*}
	\item Let $\lambda \in \Hom(\T,\C^*)\otimes\Q$ be a fractional character. We define the right twisted Demazure-Lusztig operator $\DL{s,\lambda}:\KTh_\T(G/B)[y] \to \KTh_\T(G/B)[y]$ by
	$$\DL{s,\lambda}=\DL{s,-\langle\lambda,\alpha_s^\vee\rangle} \,.$$
	\item For $a\in \R$ we define the left  twisted Demazure-Lusztig operator
	\begin{align*}
		\DLL {s,a}:=&
		\frac{1+y\alpha^{-1}_s}{1-\alpha^{-1}_s}\cdot s^{\textsc l} -\frac{(1+y)\cdot\alpha_s^{-\lceil a\rceil}}{1-\alpha^{-1}_s} \cdot \id_{\KTh_\T(G/B)[y]}\\
		=&
		\frac{1+y\alpha^{-1}_s}{1-\alpha^{-1}_s}\cdot s^{\textsc l} +\frac{(1+y)\cdot\alpha_s^{1-\lceil a\rceil}}{1-\alpha_s} \cdot \id_{\KTh_\T(G/B)[y]}
		\,.
	\end{align*}
	\item Let $\lambda \in \Hom(\T,\C^*)\otimes\Q$ be a fractional character. We define the twisted left Demazure-Lusztig operator $\DLL{s,\lambda}:\KTh_\T(G/B)[y] \to \KTh_\T(G/B)[y]$ by
	$$\DLL{s,\lambda}=\DLL{s,\langle\lambda,\alpha_s^\vee\rangle} \,. $$
	\end{enumerate}
\end{adf}
Formally the definitions make sense after the inversion of the elements $1-\LL$ and $1-\alpha$, but it will be clear later, that the operations lead to $\KTh_\T(G/B)[y]$, see Remark \ref{rem:calkowitosc}.
\subsection{Right Demazure-Lusztig operators}
We note several straightforward consequences of the above definition.
\begin{pro}\label{pro:defcor}
	Let $s$ be a simple reflection, $\lambda$ a fractional character and  $a \in \R$ a real number.
	\begin{enumerate}
		\item For $a\in (-1,0]$ the twisted operators are equal to the Demazure-Lusztig operators 
		$$\DL{s}=\DL{s,a}\,, \qquad \DLL{s}=\DLL{s,a}\,.$$
		\item The operators $\DL{s,\lambda}$ are morphisms of $\KTh_\T(pt)[y]$-modules. 
		\item For a small enough positive real number $\varepsilon>0$ we have
		$$ \DL{s,a}=\DL{s,a-\varepsilon}\,, \qquad \DLL{s,a}=\DLL{s,a-\varepsilon}\,.$$
		\item Let $\lambda^-\in \Hom(\T,\C^*)\otimes\Q$ be an anti-ample fractional character. 
For a small enough positive real number $\varepsilon>0$ we have
		$$ \DL{s,a}=\DL{s,\lambda-\varepsilon\lambda^-}\,, \qquad \DLL{s,a}=\DLL{s,\lambda+\varepsilon\lambda^-}\,.$$
		\item Let $\lambda$ be a fractional character. The twisted Demazure-Lusztig operator is equal to
		\begin{align*}\DL{s,\lambda}&=\lambda_y(\Ls^*)\cdot\pi_s^*\pi_{s*} -\left(1+\frac{(1+y)(1-\Ls^{\lceil -\langle\lambda,\alpha_s\rangle\rceil})}{1-\Ls}\right)\cdot \id\\
&=\DL{s}-\frac{(1+y)(1-\Ls^{\lceil -\langle\lambda,\alpha_s\rangle\rceil})}{1-\Ls}\cdot \id
\,.\end{align*}	
	\end{enumerate}
	
\end{pro}

\begin{rem}\label{rem:calkowitosc}
	For every integer $n \in \Z$ the element $1-\Ls^{n}$ is divisible by $1-\Ls$ in the equivariant K-theory $\KTh_\T(G/B)$. Therefore, the twisted Demazure-Lusztig operators define maps
	$$\DL{s,a}:\KTh_\T(G/B)[y]\to \KTh_\T(G/B)[y]\,.$$
Similarly $\DLL{s,a}$ acts on $\KTh_\T(G/B)[y]$. This action does not preserve the $\KTh_\T(pt)$-module structure. On the other hand, unlike $\DL{s,a}$, it descends to an action on $\KTh_\T(G/P)[y]$ for any parabolic group $P\supset B$.
\end{rem}
We prove a quadratic relation for the twisted Demazure-Lusztig polynomials.
\begin{pro}\label{pro:kwadrat1}Let $s\in W$ be a simple reflection and $a\in \R$ a real number. Suppose that $a\notin \Z$. Then
		\begin{align*}
			\DL{s,a}\circ \DL{s,-a}=-y \,\id_{\KTh_\T(G/B)[y]}\,.
		\end{align*}
\end{pro}

\begin{proof}
	For two real numbers $a,b$ and $\xi\in \KTh_\T(G/B)[y]$ the composition $\DL{s,a}\circ \DL{s,b}$ evaluated at $\xi$ is of the form
	$$\frac{1+y\Ls^{*} }{1-\Ls }\left(\frac{1+y\Ls 
	}{1-\Ls^{* }}\xi-\frac{(1+y) 
		\Ls^{-\lceil b\rceil
	}}{1-\Ls^{*}}s^{\textsc r}\xi\right)
	-\frac{(1+y) \Ls ^{\lceil a\rceil }}{1-\Ls }
	\left(\frac{
		1+{y}{\Ls^* }}{1-\Ls }s^{\textsc r}\xi-\frac{(1+y)  \Ls ^{\lceil b\rceil
	}}{1-\Ls }\xi\right) \,.$$
	Suppose that $b=-a$ and the number $a$ is not an integer, then 	
	$$-\lceil b\rceil=-\lceil -a\rceil=\lfloor a\rfloor
	=\lceil a\rceil-1 \,. $$
	In the expression above the coefficient of $s^{\textsc r}\xi$ is equal to
	$$-\frac{1+y\Ls^{*} }{1-\Ls }\cdot\frac{(1+y) 
		\Ls^{\lceil a\rceil-1
	}}{1-\Ls^{*}}
	-\frac{(1+y) \Ls ^{\lceil a\rceil }}{1-\Ls }
	\cdot
	\frac{1+{y}{\Ls^* }}{1-\Ls }=0\,.$$
	The coefficient of $\xi$ is equal to
	$$\frac{1+y\Ls^{*} }{1-\Ls }\cdot\frac{1+y\Ls 
	}{1-\Ls^{* }}
	+\frac{(1+y) \Ls ^{\lceil a\rceil }}{1-\Ls }
	\cdot   
	\frac{(1+y)  \Ls ^{1-\lceil a\rceil
	}}{1-\Ls }\,.$$
	Simplifying we obtain $-y$.
\end{proof}
\begin{cor}
	Suppose that $a \in \Z$. Then
	$$\DL{s,a}\circ \DL{s,-a}=-y-(1+y)\Ls^a\cdot \DL{s,-a}\,.$$
	For $a=0$ we obtain the standard relation
	$$(\DL{s}+y)(\DL{s}+1)=0 \,. $$
\end{cor}
\begin{proof}
	Let $\varepsilon>0$ be a small enough positive real number. We have
	$$\DL{s,-a}=\DL{s,-a-\varepsilon}\,, \qquad \DL{s,a}=\DL{s,a+\varepsilon} -(1+y)\Ls^a \,. $$
	Therefore
	$$\DL{s,a}\circ \DL{s,-a}=(\DL{s,a+\varepsilon} -(1+y)\Ls^a)\circ\DL{s,-a-\varepsilon}=
	\DL{s,a+\varepsilon}\circ \DL{s,-a-\varepsilon} -(1+y)\Ls^a\cdot\DL{s,-a}\,.
	$$
	The result follows from Proposition \ref{pro:kwadrat1}.
\end{proof}
\begin{cor}\label{pro:kwadrat} Let $s\in W$ be a simple reflection and $\lambda \in \Hom(\T,\C^*)\otimes\Q$ a fractional character.
	\begin{enumerate}
		\item Suppose that $\langle\lambda,\alpha_s^\vee\rangle\not\in\Z$, then
		$$\DL{s,\lambda}\circ \DL{s,s\lambda}=-y \,\id\,.$$
		\item Suppose that $\langle\lambda,\alpha_s^\vee\rangle\in\Z$. Then
		$$\DL{s,\lambda}\circ \DL{s,s\lambda}=-y\cdot \id-(1+y)\Ls^{-\langle\lambda,\alpha_s\rangle}\cdot \DL{s,s\lambda}\,.
		$$
	\end{enumerate}
	
\end{cor}
Further properties of the twisted Demazure-Lusztig operators are treated in Section~\ref{s:alg}.

\subsection{Bott-Samelson recursion}
In this section we prove our first inductive formula.
Let $\alpha_s$ be a simple root and $s\in W$ the corresponding reflection. Denote by $\Ls$ the relative tangent bundle of the projection $\pi_s:G/B\to G/P_s\,.$ 
\begin{atw} \label{tw:1} 
	Let $\lambda \in \Hom(\T,\C^*)\otimes\Q$ be a fractional character and $w \in W\simeq (G/B)^\T$ a fixed point such that $l(ws)>l(w)$. Then
	\begin{align*}
		\mC(ws,s\lambda)=\DL{s,\lambda}(\mC(w,\lambda))\,.
	\end{align*}
\end{atw}
\begin{rem}
	Theorem \ref{tw:1} is the limit case of the induction for the elliptic classes of \cite{RW}. We prefer to stay entirely in the framework of K-theory and analyze the Bott-Samelson construction from the point of view of twisted motivic Chern classes. 
\end{rem}
Before proving the above theorem we note several corollaries.
\begin{cor}
	For $\lambda=0$ we recover \cite[Corollary 5.2]{AMSS}. 
\end{cor}

\begin{cor} \label{cor:ws} \label{rem:tw1} 
	Let $\lambda \in \Hom(\T,\C^*)\otimes\Q$ be a general enough fractional character, i.e. such that $\langle\lambda,\alpha_s\rangle$ is not an integer. For a fixed point $w \in W$  such that $l(w)>l(ws)$ we have
	$$(-y)\cdot\mC(ws,s\lambda)=\DL{s,\lambda}(\mC(w,\lambda))\,.$$
	Therefore, for an arbitrary $w \in W$ we have
	$$(-y)^{\frac{1}{2}(l(w)-l(ws)+1)}\cdot\mC(ws,s\lambda)=\DL{s,\lambda}(\mC(w,\lambda))\,.$$
\end{cor}
\begin{proof}
	It follows from Theorem \ref{tw:1} and Corollary \ref{pro:kwadrat}.
\end{proof}
We recall that the condition \emph{general enough} means here and elsewhere that $\langle\lambda,\alpha^\vee\rangle\not\in \Z$ for any root $\alpha$.
The above corollary may be restated in the language of stable envelopes.
\begin{cor}\label{cor:stabind}
	Let $w \in W$ be an arbitrary fixed point and $\lambda$ a  general enough  fractional character. Then
	$$
		q^{1/2}\cdot\St^{s\lambda}(ws)=\DLq{s,\lambda}(\St^{\lambda}(w))\,,
	$$
	where $\DLq{s,\lambda}$ denotes the operation $\DL{s,\lambda}$ after the substitution $y=-q$.
\end{cor}
\begin{proof}
	It follows directly from Theorem \ref{tw:otoczki} and Corollary \ref{rem:tw1}.
\end{proof}
We may obtain a new interesting recursion for stable envelopes by choosing a non-generic value of $\lambda$.
\begin{cor} \label{cor:stabind3}
	Let $w \in W$ be a fixed point such that $l(ws)>l(w)$. Let $\lambda$ be an arbitrary character and $\lambda^-\in \Hom(\T,\C^*)\otimes\Q$ a small anti-ample character.
	For a small enough positive real number $\varepsilon>0$ we have
	$$
	q^{1/2}\cdot\St^{s\lambda+\varepsilon\lambda^-}(ws)=\DLq{s,\lambda}(\St^{\lambda+\varepsilon\lambda^-}(w))\,,
	$$
	where $\DLq{s,\lambda}$ denotes operator $\DL{s,\lambda}$ after the substitution $y=-q$.
\end{cor} 
\begin{proof}
	It follows directly from Corollary \ref{cor:otoczki} and Theorem \ref{tw:1}.
\end{proof}
\begin{rem}
	Corollaries \ref{cor:stabind3} and \ref{cor:stabind} produce different recursive formulas, e.g. we may obtain \cite[Theorem~3.5]{SZZ} taking $\lambda=0$ in Corollary \ref{cor:stabind3} and \cite[Theorem 5.4]{SZZ2} from Corollary \ref{cor:stabind} (see Corollary \ref{cor:SZZ}).
\end{rem}
 The rest of this section is devoted to the proof of Theorem \ref{tw:1}. We split the proof into several lemmas. For an arbitrary element $\sigma\in W$ with a reduced word decomposition $\underline{\sigma}$ we define
 \begin{align*}
 	\mCloc(\sigma,\lambda)&:=
 	\frac{\mC(\sigma,\lambda)}{\eu(TG/B)} \in S^{-1}\KTh_\T(G/B)[y]\,, \\
 	\mCloc(\underline{\sigma},\lambda)&:=
 	\frac{\mC(Z^\oo _{\underline{\sigma}}\subset Z_{\underline{\sigma}})\cdot \O_{Z_{\underline{\sigma}}}(\lceil p_{\underline{\sigma}}^*\Delta_{\sigma,\lambda}\rceil)}{\eu(TZ_{\underline{\sigma}})} \in S^{-1}\KTh_\T(Z_{\underline{\sigma}})[y]\,. 
 \end{align*}
 Consider the situation as in Theorem \ref{tw:1}. Fix a reduced word decomposition $\wu$ of $w$. It induces a reduced word decomposition $\ws$ of $ws$. \\
  \begin{lemma} \label{lem:BSind0}
 	Consider the situation as in Theorem \ref{tw:1}.
 	Let $\eheta$ be a binary sequence corresponding to a fixed point in $Z_\wu$. Let $\ehet$ be a binary sequence of the form $\ehet=(\eheta,\delta)$, where $\delta\in\{0,1\}$. We have the following equalities in the localized K-theory of a point $S^{-1}\KTh_\T(pt)[y]$.
 	\begin{enumerate}
 		\item $$\eu(TZ_\ws)_{|\pehet}=\eu(T Z_\wu)_{|\peheta}\cdot(1-\Ls^*)_{|\pehet} \,.$$
 		
 		\item
 		$$
 		\O_{Z_\ws}(\lceil p_\ws^*\Delta_{ws,s\lambda}\rceil)_{|\pehet}=
 		\begin{cases}
 			   \O_{Z_\wu}(\lceil p_\wu^*\Delta_{w,\lambda}\rceil)_{|\peheta}\cdot \LL_{s\,|\pehet}^{\lceil -\langle\lambda,\alpha_s\rangle\rceil}  &  \text{ when } \delta=0\,, \\
 			\O_{Z_\wu}(\lceil p_\wu^*\Delta_{w,s\lambda}\rceil)_{|\peheta}  & \text{ when } \delta=1\,.
 		\end{cases}
 		$$
 		\item
 		$$\mCloc(\ws,s\lambda)_{|\pehet}=
 		\begin{cases}
 			 \mCloc(\wu,\lambda)_{|\peheta}\cdot
 			 \left(\frac{(1+y)\cdot\Ls^{\lceil -\langle\lambda,\alpha_s\rangle\rceil-1}}{1-{\Ls^*}}\right)_{|\pehet} &
 			 \text{ when } \delta=0\,,\\
 			\mCloc(\wu,\lambda)_{|\peheta}\cdot
 			\left(\frac{1+y{\Ls^*}}{1-{\Ls^*}}\right)_{|\pehet} & \text{ when } \delta=1\,.
 		\end{cases}$$

 		\item In particular
$$\mC(Z^\oo _\ws \subset Z_\ws)_{|\pehet}=
 		\begin{cases}
 			\mC(Z_\wu^\oo  \subset Z_\wu)_{|\peheta}\cdot(1+y)\cdot {\Ls^*}_{|\pehet} & \text{ when } \delta=0\,, \\
 			\mC(Z_\wu^\oo  \subset Z_\wu)_{|\peheta}\cdot\left(1+y\Ls^*\right)_{|\pehet} &   \text{ when } \delta=1\,.\\
 		\end{cases}$$
 			\end{enumerate}
 \end{lemma}
\begin{proof}
	{\bf 1)} The bundle $p_{\ws}^*\Ls$ is the relative tangent bundle of the projection $\pi_{\ws}$. Therefore we have a short exact sequence
	$$0\to p_{\ws}^*\Ls\to TZ_\ws \to\pi^*TZ_\wu\to 0\,. $$
	The claim follows from the multiplicative properties of Euler class. \\
	\\
	{\bf 2)} Denote the divisor $i(Z_\wu)\subset Z_\ws$ by $D$. The bundle $\O_{Z_\ws}(D)_{|D}$ is isomorphic to~${\Ls}_{|D}$. The point $\pehet$ belongs to $D$ if and only if $\delta=0$. Therefore
	\begin{align} \label{w:OD}
	\O_{Z_\ws}(D)_{|\pehet}=
	\begin{cases}
		\LL_{s|\pehet} & \text{ when } \delta=0\,, \\
		1 & \text{ when } \delta=1\,.
	\end{cases}
	\end{align}
	The Kempf lemma (Proposition \ref{pro:kem}) implies that
	$$p_\ws^*\Delta_{ws,s\lambda}=\pi^*_\ws p_\wu^*\Delta_{w,\lambda } + \langle-\lambda,\alpha_s\rangle\cdot D\,.$$
	We use this equality to obtain
	\begin{align*}
		\notag \O_{Z_\ws}(\lceil p_\ws^*\Delta_{ws,s\lambda}\rceil)_{|\pehet}&=
		\O_{Z_\ws}(\lceil\pi^*_\ws p_\wu^*\Delta_{w,\lambda}  +\langle-\lambda,\alpha_s\rangle\cdot D\rceil )_{|\pehet} \\
		&=\O_{Z_\ws}(\pi_\ws^*\lceil p_\wu^*\Delta_{w,\lambda}\rceil)_{|\pehet}\cdot\O_{Z_\ws}(\lceil \langle-\lambda,\alpha_s\rangle\rceil\cdot D )_{|\pehet} \\
		&=\O_{Z_\wu}(\lceil p_\wu^*\Delta_{w,\lambda}\rceil)_{|\peheta}\cdot\O_{Z_\ws}(D)^{\lceil \langle-\lambda,\alpha_s\rangle\rceil}_{|\pehet}\,,
	\end{align*}
	where the second equality follows from Proposition \ref{pro:round}. The claim follows from formula~\eqref{w:OD}. \\
	{\bf 4)} Suppose that $\delta=0$. Then $\pehet \in D$. The divisor $\partial Z_\ws$ is SNC. Therefore (cf. \cite[Lemma 9.7]{KonW})
	\begin{align*}
		\mC(Z^\oo _\ws \subset Z_\ws)_{|D} 
		&=(1+y)\cdot \O_{Z_\ws}(-D)_{|D} \cdot \mC(D^\oo  \subset D)\\
		&=(1+y)\cdot {\Ls^*}_{|D} \cdot \mC(Z_\wu^\oo  \subset Z_\wu)\,.
	\end{align*}
	Suppose that $\delta=1$ then
	\begin{align*} 
		\mC(Z^\oo _\ws \subset Z_\ws)_{|\ehet}
		&=
		\mC(\pi_\ws^{-1}(Z^\oo _\wu)  \subset Z_\ws)_{|\ehet} \\
		&=\mC(Z^\oo _\wu  \subset Z_\wu)_{|\eheta} \cdot \lambda_y(\LL^*_s)_{|\ehet}\,,
	\end{align*}
	where the first equality comes from localness of the motivic Chern class \cite[Section~2.3]{FRW} and the second from the Verdier-Riemann-Roch formula \cite[Theorem 4.2 (4)]{AMSS}. \\
	\\
	{\bf 3)} This point follows from {\bf 1--2)} and {\bf 4)}. 
	\end{proof}

\begin{lemma} \label{lem:BSind}
	Consider the situation as in Theorem \ref{tw:1}. For $\sigma \in (G/B)^\T$ we have
	$$\mCloc(ws,s\lambda)_{|\sigma}=
	\mCloc(w,{\lambda})_{|\sigma} \cdot\frac{(1+y)\cdot{\LL_{s\,|\sigma}^{\lceil -\langle\lambda,\alpha_s\rangle\rceil-1}}}{1-{\Ls^*}_{|\sigma}} +
	\mCloc(w,{\lambda})_{|\sigma s} \cdot \frac{1+y{\Ls^*}_{|\sigma}}{1-{\Ls^*}_{|\sigma}} \,.$$
\end{lemma}
 
\begin{proof}
	By definition
	$$\mC(ws,s\lambda)=p_{\ws*} \mC(Z^\oo _{\ws}\subset Z_\ws)\cdot \O_{Z_\ws}(\lceil p_\ws^*\Delta_{ws,\lambda}\rceil)\,. $$
	The LRR formula (Theorem \ref{tw:LRR}) implies that 
	\begin{align} \label{eq:mS}
		\mCloc(ws,s\lambda)_{|\sigma}=&
		\sum_{\ehet\in (Z_\ws)^\T,\ p_{\ws}(\eheta')=\sigma}
		\mCloc(\ws,s\lambda)_{|\eheta'} \,.
	\end{align}
	We have a set decomposition
	$$
	\{\ehet\in Z_\ws^\T|p_{\ws}(\ehet)=\sigma\}=
	\{(\eheta,0)|\eheta\in Z_\wu^\T, \, p_{\wu}(\eheta)=\sigma\}\sqcup
	\{(\eheta,1)|\eheta\in Z_\wu^\T, \, p_{\wu}(\eheta)=\sigma s\}\,.
	$$
	This allows us to split the sum \eqref{eq:mS} into two sums. Lemma \ref{lem:BSind0} (4) and the LRR formula implies that the first part is equal to
	\begin{align*}
	\sum_{\eheta\in (Z_\wu)^\T, \ p_\wu(\eheta)=\sigma}\mCloc(\wu,{\lambda})_{|\eheta} \cdot \frac{(1+y)\cdot{\LL_{s\,|\ehet}^{\lceil -\langle\lambda,\alpha_s\rangle\rceil-1}}}{1-{\Ls^*}_{|\ehet}}
	&=
	\mCloc({w},{\lambda})_{|\sigma} \cdot\frac{(1+y)\cdot{\LL_{s\,|\sigma}^{\lceil -\langle\lambda,\alpha_s\rangle\rceil-1}}}{1-{\Ls^*}_{|\sigma}}\,,
	\end{align*}
	and the second part to
	\begin{align*}
	\sum_{\eheta\in (Z_\wu)^\T, \ p_\wu(\eheta) = \sigma s}
	\mCloc(\wu,{\lambda})_{|\eheta} \cdot \frac{1+y{\Ls^*}_{|\ehet}}{1-{\Ls^*}_{|\ehet}}
	=
	\mCloc(w,{\lambda})_{|\sigma s} \cdot \frac{1+y{\Ls^*}_{ |\sigma}}{1-{\Ls^*}_{|\sigma}} \,.
	\end{align*}
	This proves the lemma.
\end{proof}
\begin{proof}[Proof of Theorem \ref{tw:1}]
	Let $\sigma\in (G/B)^\T$ be an arbitrary fixed point. We multiply the equation from Lemma \ref{lem:BSind} by the Euler class $\eu(T_\sigma G/B)$. After simplification we obtain
	$$\mC(ws,s\lambda)_{|\sigma}=
	\frac{1+y{\Ls^*}_{|\sigma}}{1-{\Ls}_{|\sigma}}\cdot \mC(w,\lambda)_{|\sigma s}
	-\frac{(1+y)\cdot{\LL_{s\,|\sigma}^{\lceil -\langle\lambda,\alpha_s\rangle\rceil}}}{1-{\Ls}_{|\sigma}}\cdot \mC(w,\lambda)_{|\sigma}\,.$$ 
	Theorem \ref{tw:1} follows from the localization isomorphism.
\end{proof}

\section{The Twisted Hecke algebra} \label{s:alg} 
In this section we describe the algebra of operations $\DL{s,a}\in \End(\KTh_\T(G/B)[y])$. We show that one can define the operators $\DL{w,a}$ for any $w\in W$ in a way that the resulting algebra is a deformation of the Hecke algebra and it acts in the expected way on the twisted motivic Chern classes of Schubert cells. Further on, in Section \ref{final}, purely algebraically, we lift this action to the endomorphisms of the representation ring $R(\T\times \T)[y]$. From that point of view both operators $\DL{s,a}$ and $\DLL{s,a}$ play equal roles. 

\begin{pro} \label{pro:DL1}
	Let $s$ be a simple reflection and $a\in \R$ a real number. The operator $\DL{s,a}$ commute with the left Weyl group action.
\end{pro}
\begin{proof}
	The proof is analogous to \cite[Proposition~5.9~(a)]{MNS}. The left and right Weyl group actions commute (see \cite[Proposition~5.3~(c)]{MNS}). Moreover, the line bundle $\Ls$ is $G$-equivariant, so it is preserved by the left Weyl group action.
\end{proof}
\begin{pro} \label{lem:L}
	The left and right twisted Demazure-Lusztig operators commute. Let $s_1,s_2 \in W$ be simple reflections and $a,b\in \R$ real numbers. Then
	$$\DL{s_1,a}\circ \DLL {s_2,b}=
	\DLL {s_2,b}\circ \DL{s_1,a}\,.$$
\end{pro}

\begin{proof}
	The twisted Demazure-Lusztig operator $\DL{s_1,a}$ commutes with both the left Weyl group action and multiplication by a character from $S^{-1}\KTh_{\T}(pt)[y]$, see Proposition \ref{pro:DL1}.
\end{proof}
\begin{cor} 
	Let $s_1,s_2 \in W$ be simple reflections and $\lambda_1,\lambda_2 \in \Hom(\T,\C^*)\otimes\Q$ fractional characters. Then
	$$\DL{s_1,\lambda_1}\circ \DLL {s_2,\lambda_2}=
	\DLL {s_2,\lambda_2}\circ \DL{s_1,\lambda_1}\,.$$
\end{cor}
The twisted operators satisfy certain braid relations with parameters described in Lemma \ref{universal_braid} which allows us to define operator~$\DL{ w,\lambda}$ for an arbitrary element $w\in W$.
\begin{adf}
	Let $w\in W$ be a Weyl group element and \hbox{$\wu=(s_{i_1},s_{i_2},\dots ,s_{i_l})$} a reduced word decomposition of $w$. Let $w_{> k}\in W$ be the composition of the last $l-k$ letters of $\wu$, i.e. 
	$$w_{> k}=s_{i_{k+1}}s_{i_{k+2}}\dots s_{i_{l}} \in W \,.$$
	For a fractional character $\lambda$ we define the twisted Demazure-Lusztig operator $\DL{ \wu,\lambda}$ by
$$\DL{\wu,\lambda}:=\DL{s_{i_1},w_{> 1}\lambda}\circ
	\DL{s_{i_2},w_{>2}\lambda} \circ \dots \circ
	\DL{s_{i_{l-1}},s_{i_l}\lambda} \circ
	\DL{s_{i_l},\lambda}\,. $$
\end{adf}

\begin{pro} \label{pro:braid}
	Let $\lambda$ be a fractional character and $w\in W$ a Weyl group element.
	Let $\wu$~and~$\wu'$ be two reduced word representations of $w$.
	Then
	$$ \DL{ \wu,\lambda} =\DL{ \wu',\lambda}\,.$$
\end{pro}
\begin{proof}
	It is enough to check that the equality holds on some basis of $S^{-1}\KTh_\T(G/B)[y]$ over $S^{-1}\KTh_\T(pt)[y]$.
	The elements $\{\sigma^{\textsc l}(\mC(\id,\lambda))\}_{\sigma\in W}$
	form such a basis. Thus, it is enough to check that
$$ \DL{ \wu,\lambda}(\sigma^{\textsc l}(\mC(\id,\lambda))) =\DL{ \wu',\lambda}(\sigma^{\textsc l}(\mC(\id,\lambda)))\,.$$
	Proposition \ref{pro:DL1} and multiple use of Theorem \ref{tw:1} prove that both sides are equal to
$$\sigma^{\textsc l}(\mC(w^{-1},w\lambda))\,. $$
\end{proof}
The proof above uses the fact that the motivic Chern classes for a fixed $\lambda$ span $S^{-1}\KTh_\T(G/B)[y]$. An alternative proof is provided by using directly the braid relation. Eventually the result follows from the calculus of operations acting on Laurent polynomials, see Section \ref{abstract_algebra}.

\begin{ex}\label{przyklad_sl3} Let $G=\SL_3\subset \GL_3$, $\lambda=(\lambda_1,\lambda_2,\lambda_3)\in \ttt^*\simeq \R^3\subset \mathfrak{gl}_3$.
For the word $\wu=s_2s_1s_2$ we have:
$$
\DL{s_1s_2s_1,\lambda}=\DL{s_1,s_2s_1\lambda}\circ\DL{s_2,s_1\lambda}\circ\DL{s_1,\lambda}
=\DL{s_1,(\lambda_2,\lambda_3,\lambda_1)}\circ\DL{s_2,(\lambda_2,\lambda_1,\lambda_3)}\circ\DL{s_1,(\lambda_1,\lambda_2,\lambda_3)}\,.
$$
The corresponding parameters $\lceil\langle -w^{-1}_{< j}\lambda,\alpha_{s_{i_j}} \rangle\rceil$ in Definition \ref{df:DL} (4) are equal to:
$$\lceil\lambda_3-\lambda_2\rceil,\qquad \lceil\lambda_3-\lambda_1\rceil,\qquad \lceil\lambda_2-\lambda_1\rceil\,.$$
For the word $\wu=s_2s_1s_2$: 
$$
\DL{s_2s_1s_2,\lambda}=\DL{s_2,s_1s_2\lambda}\circ\DL{s_1,s_2\lambda}\circ\DL{s_2,\lambda}
=\DL{s_2,(\lambda_3,\lambda_1,\lambda_2)}\circ\DL{s_1,(\lambda_1,\lambda_3,\lambda_2)}\circ\DL{s_2,(\lambda_1,\lambda_2,\lambda_3)}\,.
$$
The corresponding parameters $\lceil\langle -w^{-1}_{< j}\lambda,\alpha_{s_{i_j}} \rangle\rceil$ are equal to:
$$\lceil\lambda_2-\lambda_1\rceil,\qquad \lceil\lambda_3-\lambda_1\rceil,\qquad \lceil\lambda_3-\lambda_2\rceil\,.$$
We obtain relation
$$
\DL{s_1,\lambda_3-\lambda_2}\circ\DL{s_2,\lambda_3-\lambda_1}\circ\DL{s_1,\lambda_2-\lambda_1}=
\DL{s_2,\lambda_2-\lambda_1}\circ\DL{s_1,\lambda_3-\lambda_1}\circ\DL{s_2,\lambda_3-\lambda_2} \,.
$$
\end{ex}

\begin{adf}
	Let $\lambda$ be a fractional character and $w \in W$ a Weyl group element.
	We define
	$$\DL{ w,\lambda}:= \DL{ \wu,\lambda}\,,$$
	for any reduced word representation $\wu$ of $w$.
\end{adf}
\begin{pro} \label{pro:tw1}
	Let $w,w' \in W$ be  Weyl group elements and $\lambda \in \Hom(\T,\C^*)\otimes\Q$ a~fractional character.
	\begin{enumerate}
		\item Suppose that $l(w')+l(w)=l(w'w)$, then
$$\DL{w^{-1},\lambda}(\mC(w',\lambda))=\mC(w'w,w^{-1}\lambda)\,. $$

		\item Suppose that $\lambda$ is general enough, then
	$$\DL{w^{-1},\lambda}(\mC(w',\lambda))=(-y)^{\frac{1}{2}(l(w')+l(w)-l(w'w))}\mC(w'w,w^{-1}\lambda)\,. $$
	\end{enumerate}
\end{pro}
\begin{proof}
	The first part follows from Theorem \ref{tw:1}. The second is a  consequence of Corollary~\ref{rem:tw1}.
\end{proof}
\begin{rem},,General enough'' means that $\langle\lambda,\alpha^\vee\rangle\not\in\Z$ for any root $\alpha$.\end{rem} 
\begin{cor} \label{cor:tw1}
	Let $w,w' \in W$ be Weyl group elements and $\lambda \in \Hom(\T,\C^*)\otimes\Q$ a~general enough fractional character. Then the corresponding stable envelopes are related by the recursive relation
	$$\DLq{w^{-1},\lambda}(\St^\lambda(w'))=q^{l(w)/2}\cdot\St^{w^{-1}\lambda}(w'w)\,. $$
\end{cor}
The above results imply \cite[Theorem 5.4]{SZZ2}. To see this we need the following proposition.
\begin{pro} \label{pro:antiample}
	Let $s \in W$ be a simple reflection and $w\in W$ a Weyl group element such that $l(sw)>l(w)$. Let $\lambda^-\in\Hom(\T,\C^*)\otimes\Q$ be a small anti-ample slope. Then $$\DL{s,sw\lambda^-}=\DL{s}\,.$$  
\end{pro}
\begin{proof}
	It is enough to show that
	$$\lceil-\langle sw\lambda^-,\alpha_s^\vee\rangle\rceil=\lceil\langle w\lambda^-,\alpha_s^\vee\rangle\rceil=0\,.$$
	The slope $\lambda^-$ is small enough, therefore we only need to prove that
	$\langle w\lambda^-,\alpha_s^\vee\rangle$ is negative.
	We have
	$$\langle w\lambda^-,\alpha_s^\vee\rangle=\langle \lambda^-,w^{-1}\alpha_s^\vee\rangle<0\,,$$
	where the last inequality follows from Proposition \ref{pro:walpha} implying that $w^{-1}\alpha_s^\vee$ is positive.
\end{proof}
\begin{cor}[{\cite[Theorem 5.4]{SZZ2}}] \label{cor:SZZ}
	Let $w,w' \in W$ be Weyl group elements and $\lambda^{-}\in\Hom(\T,\C^*)\otimes\Q$ a small anti-ample slope.
	Then
$$\DLq{w^{-1}}(\St^{w\lambda^-}(w'))={}q^{l(w)/2}\cdot\St^{\lambda^-}(w'w)\,. $$ 
\end{cor} 
\begin{proof} Set $\lambda=w\lambda^-$.
	 By the Corollary \ref{cor:tw1} 
		$$\DLq{w^{-1},w\lambda^-}(\St^{w\lambda^-}(w'))=q^{l(w)/2}\cdot\St^{\lambda^-}(w'w)\,. $$
It is enough to prove that $\DL{ w^{-1},w\lambda^-}=\DL{ w^{-1}}\,.$ Fix any reduced word decomposition $\wu=(s_{i_1},s_{i_2},\dots ,s_{i_m})$ of $w$. By definition of the twisted operators it is enough to prove that for any $k$  we have
	 $$\DL{s_{i_k},(w_{< k})^{-1}w\lambda^-}=\DL{s_{i_k},w_{\ge k}\lambda^-}=\DL{s_{i_k},s_{i_k}w_{> k}\lambda^-}=\DL{s_{i_k}}\,,$$
	 where $w_{< k}$ is the element corresponding to the first $k-1$ letters of the word $\wu$ and $w_{\geq k}=s_{i_k}w_{> k}=w_{< k}^{-1}w$. The last equation follows from Proposition 
\ref{pro:antiample}. 
\end{proof}

\begin{rem}In \cite{SZZ2}  a different notation is used. Namely our $\St^{w\lambda^-}(w')$
denotes
${\rm stab}^{+,w\nabla_-}_{w'}$. 
\end{rem}
\section{Left induction}
So far we have concentrated on the right Demazure-Lusztig operator. Below we study the left operations purely algebraically, not giving any geometric interpretation. The geometric source of the algebraic formulas is explained (at least in the $A_n$-case) by analysing resolutions of matrix Schubert varieties, see Section \ref{sec:matrix_resolution}.
 
\subsection{Recursive formula for $G/B$}
In this section we prove a counterpart of Theorem \ref{tw:1} concerning the left Weyl group action. The proof is similar to the proof \cite[Proposition 7.3]{RW}, see also \cite{MNS}.

\begin{atw} \label{tw:L1}
	Let $\lambda \in \Hom(\T,\C^*)\otimes\Q$ be a fractional character and $w \in (G/B)^\T\simeq W$ a fixed point.
	\begin{enumerate}
		\item Suppose that $l(sw)>l(w)$, then
		\begin{align*}
			\mC(sw,\lambda)=\DLL {s,w\lambda}(\mC(w,\lambda))\,.
		\end{align*}
		\item Suppose that $\lambda$ is general enough, then
		\begin{align*}
			(-y)^{\frac{1}{2}(l(w)-l(sw)+1)}\mC(sw,\lambda)=\DLL {s,w\lambda}(\mC(w,\lambda))\,.
		\end{align*}
	\end{enumerate}
\end{atw}
\begin{proof}
First we prove the theorem for $w=\id$. The Schubert variety of $\id$ is a single point. Thus, the class $\mC(\id,\lambda)$ does not depend on the character $\lambda$. It vanishes at all fixed points other than $\id$ and at the fixed point $\id$ we have
$$\mC(\id,\lambda)_{|\id}=\eu(T_{\id}G/B)
\,.$$
Restrictions of the relative tangent bundle $\Ls$ are given by ${\Ls}_{|\id}=\alpha_s^{-1}$ and ${\Ls}_{|s}=\alpha_s$. (We treat $\alpha_s$ as a character and we apply the multiplicative notation.) It follows that
	\begin{align} \label{wyr:L}
		\DLL {s,\lambda}(\mC(\id,\lambda))=
		\DL{s,s\lambda}(\mC(\id,\lambda))=
		\DL{s,s\lambda}(\mC(\id,s\lambda))=
		\mC(s,\lambda)\,.
	\end{align}
	This proves the theorem for $w=\id\,.$ \\
	Let us focus on the case $l(ws)>l(w)$. Then $l(ws)=l(w)+l(s)$. We have
	\begin{align*}
		\DLL {s,w\lambda}(\mC(w,\lambda))
		&=\DLL {s,w\lambda}\circ\DL{ w^{-1},w\lambda}(\mC(\id,w\lambda)) \\
		&=\DL{ w^{-1},w\lambda}\circ\DLL {s,w\lambda}(\mC(\id,w\lambda)) \\
		&=\DL{ w^{-1},w\lambda}(\mC(s,w\lambda)) \\
		&=\mC(sw,\lambda)
		\,,
	\end{align*}
	where the first and the fourth equality follow from Proposition \ref{pro:tw1} (1), the second from Lemma \ref{lem:L} and the third from equation \eqref{wyr:L}. \\
	The proof of the other case is analogous. In the last equality we need to use Proposition \ref{pro:tw1} (2) instead of Proposition \ref{pro:tw1} (1).
\end{proof}
\begin{rem} \label{rem:BP}
	Alternatively, we may prove Theorem \ref{tw:L1}  using another inductive construction of the Bott-Samelson varieties and reasoning similar to the proof of Theorem \ref{tw:1}. Let $w$ be
	a Weyl group element such that $sw$ is longer than $w$. Then the variety $Z_{\underline{sw}}$ admits a
	locally trivial fibration over the projective line $\PP^1$ with fiber $Z_\wu$. Compare the proof of Proposition \ref{lewaindukcja}.\end{rem}
\subsection{Recursion for $G/P$}
The left induction is also valid for homogeneous varieties $G/P$ where $P$ is an arbitrary parabolic subgroup. The left Weyl group action on $\KTh_\T(G/B)$ generalizes to the left Weyl group action $\KTh_\T(G/P)$, see \cite[Section 3]{MNS}.
Therefore, analogously to Definition \ref{df:DL}, we may define the twisted left Demazure-Lusztig operators on $\KTh_\T(G/P)$
$$
	\DLL {s,a}:=
	\frac{1+y\alpha^{-1}_s}{1-\alpha^{-1}_s}\cdot s^{\textsc l} -\frac{(1+y)\cdot\alpha_s^{-\lceil a\rceil}}{1-\alpha^{-1}_s} \cdot \id_{\KTh_\T(G/P)[y]}\,,\qquad \DLL{s,\lambda}=\DLL{s,\langle\lambda,\alpha_s^\vee\rangle}\,.
$$
Denote by $\pi$ the projection $\pi:G/B\to G/P$. The map  $\pi_*:\KTh_\T(G/B) \to \KTh_\T(G/P)$ commutes with the left Weyl group action \cite[Proposition 5.3 (d)]{MNS}, therefore the twisted left Demazure-Lusztig operators also commute with $\pi_*$. \\
Let $W_P$ be a subgroup of $W$ corresponding to $P$. The fixed point set $(G/P)^\T$ is in bijection with the set of left cosets $W/W_P$. Let $W^P\subset W$ be a set of minimal length representatives. For a coset $wW_P$ we consider the Schubert variety
$$ X_{wW_P}=\overline{B\pi(w)} \subset G/P\,, \qquad
X^\oo _{wW_P}=B\pi(w) \subset G/P\,, \qquad
\partial X_{wW_P}=X_{wW_P}\setminus X_{wW_P}^\oo \,.$$
Denote by $i_{wW_P}$ the inclusion $X_{wW_P}\to G/P$. The rational Picard group $\Pic_\Q(G/P)$  coincides with $\Pic_\Q(G/B)^{W_P}.$  For a coset $wW_P$ and a rational character $\lambda\in\Pic_\Q(G/P)$ let $\Delta^P_{wW_P,\lambda}$ be the unique $\T$-invariant divisor with support contained in the boundary $\partial X_{wW_P}$, which represents $\LL(\lambda)_{|X_{wW_P}}$. We use the  notation
$$\mC(wW_P,\lambda):=i_{wW_P*}\mC(X_{wW_P},\partial X_{wW_P}; \Delta^P_{wW_P,\lambda}) \in \KTh_\T(G/P)[y] \,.$$
\begin{lemma} \label{lem:G/P}
	Let $w\in W^P$ and $\lambda\in \Pic_\Q(G/P)$. Then
	$$
	\mC(wW_P,\lambda)=\pi_*\mC(w,\lambda)\,.$$
\end{lemma}
\begin{proof}
	Fix a reduced word decomposition $\wu$ of $w$. The element $w$ is a minimal length representative, therefore the map
	$$\pi_{|X_w}\colon X_w \to X_{wW_P}$$
	is birational. It is isomorphic on the Schubert cell and $\pi^{-1}_{|X_w}(\partial  X_{wW_P})=\partial  X_{w}.$ Thus
	$$(Z_\wu,\partial Z_{\wu}) \xto{p_{\wu}} (X_w,\partial X_w) \xto{\pi_{|X_w}} (X_{wW_P},\partial X_{wW_P})$$
	is a SNC resolution of singularities. \\
	The divisor $\pi_{|X_w}^*\Delta^P_{wW_P,\lambda}$ represents $\LL(\lambda)_{|X_w}$. It is $\T$-invariant and contained in the boundary of Schubert variety $X_w$. Therefore (cf. \cite[Section 10]{KonW})
	$$\pi_{|X_w}^*\Delta^P_{wW_P,\lambda}=\Delta_{w,\lambda}\,. $$
	It follows that both considered classes are equal to the push-forward of
	$$\mC(Z_\wu,\partial Z_\wu; p^*_{\wu}\Delta_{w,\lambda}) \,.$$
\end{proof}

\begin{lemma}[{\cite[Lemma 2.1]{Deo}}] \label{lem:Deo} 
	Suppose that $w\in W^P$ and $s$ is a simple root. There are three possibilities:
	\begin{enumerate}
	 \item $l(sw)<l(w)$, then $sw\in W^P$,
	 \item  $l(sw)>l(w)$ and $sw\in W^P$,
	 \item $l(sw)>l(w)$ and $sw\not\in W^P$, then there exists a simple reflection $\tilde s\in W_P$, such that $sw=w\tilde s$.
	 \end{enumerate} 
\end{lemma}

Let us analyse the first possibility:
\begin{pro}
	Consider $w\in W^P$. Let $s$ be a simple  reflection and $\lambda\in \Pic_\Q(G/P)$ a general enough fractional character.
		 Suppose that $l(sw)<l(w)$, then
		\begin{align*}
			\DLL {s,w\lambda}(\mC(wW_P,\lambda))=(-y)\mC(swW_P,\lambda)\,.
		\end{align*}
\end{pro}
\begin{proof} 
	 The proof is analogous to the proof of \cite[Theorem 4.3]{MNS}. We have
	 \begin{multline*}
	 	\DLL {s,w\lambda}(\mC(wW_P,\lambda))=
	 	\DLL {s,w\lambda}(\pi_*\mC(w,\lambda))= \\
	 	=\pi_*\DLL {s,w\lambda}(\mC(w,\lambda))=
	 	(-y)\cdot\pi_*\mC(sw,\lambda)=(-y)\cdot \mC(swW_P,\lambda)\,.
	 \end{multline*}
 The first equality follows from Lemma \ref{lem:G/P}. The second from commutation of  $\pi_*$ with the left Weyl group action \cite[Proposition 5.3 (d)]{MNS}, and the third from Theorem \ref{tw:L1}~(2). The last one follows from Lemmas \ref{lem:G/P} and \ref{lem:Deo}. 
\end{proof}

For a generic slope $\lambda$ we have $\DLL{s,w\lambda}\circ\DLL{s,sw\lambda}=(-y)id$. This is proven purely algebraically, as in Proposition \ref{pro:kwadrat1}, see also Section \ref{smallalgebra}. Replacing $w$ by $sw$ and using the identity  we obtain the corollary:
\begin{cor}
	Consider $w\in W^P$. Let $s$ be a simple reflection and $\lambda\in\Pic_\Q(G/P)$ a general enough fractional character.
		 Suppose that $l(sw)>l(w)$ and $sw\in W^P$, then
		\begin{align*}
			\DLL {s,w\lambda}(\mC(wW_P,\lambda))=\mC(swW_P,\lambda)\,.
		\end{align*}
\end{cor}

It remains to describe the result of the action of the left Demazure-Lusztig operator in the case when $swW_P=wW_P$.

\begin{pro}
	Consider $w\in W^P$. Let $s$ be a simple reflection and $\lambda\in \ttt_\Q^W\simeq\Pic_\Q(G/P)$ a general enough fractional character.
		 Suppose that $l(sw)>l(w)$ and $sw\not\in W^P$ then 
		\begin{align*}
			\DLL {s,w\lambda}(\mC(wW_P,\lambda))=-y\mC(wW_P,\lambda)\,.
		\end{align*}
\end{pro}

\begin{proof} We have \begin{multline*}\DLL {s,w\lambda}(\mC(wW_P,\lambda))=\pi_*\left(\DLL {s,w\lambda}(\mC(w,\lambda))\right)=
\pi_*\left(\mC(sw,\lambda)\right)=\\
=\pi_*\left(\mC(w\tilde s,\lambda)\right)
=\pi_*\left(\DL {\tilde s,\tilde s\lambda}(\mC(w,\tilde s\lambda))\right)=\pi_*\left(\DL {\tilde s,\lambda}(\mC(w,\lambda))\right)\,,\end{multline*}
where $\tilde s$ is given by Lemma \ref{lem:Deo}.(3). By our assumptions $\tilde s\lambda=\lambda$, hence the last equality. It remains to show that 
$$\pi_*\left(\DL {\tilde s,\lambda}(\mC(w,\lambda))\right)=-y\pi_*\left(\mC(w,\lambda)\right)\,.$$
The statement is a consequence of the following lemma.\end{proof}

\begin{lemma} Let $\tilde s$ be a simple reflection, such that $\tilde s\lambda=\lambda$. Consider a parabolic subgroup $P\subset G$ such that $P_{\tilde s} \subset P$. Denote by $\pi$ the projection $\pi:G/B\rightarrow G/P$. Then
\begin{enumerate}
\item
$$\DL{\tilde s,\lambda}=\DL{\tilde s}\,,$$
\item 
$$\pi_*\circ \DL{\tilde s}=-y\,\pi_*\,.$$
\end{enumerate}
\end{lemma}
\begin{proof} 
The part (1) follows from Proposition \ref{pro:defcor} (5), since $\langle\lambda, \alpha^\vee_{\tilde s}\rangle=0$. 
For the proof of (2) we can assume, that $P=P_{\tilde s}$. Then for $\xi\in \KTh_\T(G/B)[y]$
$$\pi_* \DL{\tilde s}(\xi)=\pi_*((1+y\LL_{\tilde s}^*)\cdot\pi^*\pi_*(\xi))-\xi)=
\pi_*((1+y\LL_{\tilde s}^*)\pi^*\pi_*(\xi))-\pi_*(\xi)
\,.$$
We set $\eta=\pi_*(\xi)$ and apply the projection formula 
$$\pi_*\DL{\tilde s}(\xi)=
\pi_*((1+y\LL_{\tilde s}^*)\cdot\pi^*\eta)-\eta=\pi_*(1+y\LL_{\tilde s}^*)\cdot\eta-\eta
\,.$$
Since $\pi$ is a $\PP^1$-fibration  $\pi_*(1+y\LL_{\tilde s}^*)=(1-y)$. Hence
$$\pi_* \DL{\tilde s}(\xi)=(1-y)\eta-\eta=-y\,\pi_*(\xi)\,.$$
\end{proof}
We sum up the properties of the left action:
\begin{atw} Let  $w\in W^P$. Let $s$ be a simple  reflection and $\lambda\in \Pic_\Q(G/P)$ a general enough fractional character. Then
\begin{align*}
			\DLL {s,w\lambda}(\mC(wW_P,\lambda))=(-y)^{\dim X^P_w-\dim X^P_{sw}+1}\,\mC(swW_P,\lambda)\,.
		\end{align*}
\end{atw}
\begin{rem}
	For $\lambda=0$ Theorem \ref{tw:L1} implies \cite[Theorem 7.6]{MNS}.
\end{rem}
\begin{rem}
	The left operators $\DLL {s,\lambda}$ are defined using the left Weyl group action on $\KTh_\T(G/P)$. The right Weyl group action on $\KTh_\T(G/B)$ does not generalize to $\KTh_\T(G/P)$, so there is no generalization of induction from Theorem \ref{tw:1} to $G/P$. A similar phenomenon is thoroughly discussed in \cite{MNS}.
\end{rem}

\subsection{Wall-crossing for a change of the coweight chamber}
Due to Theorem \ref{tw:otoczki} we obtained a recursive formula for twisted motivic Chern classes. After translation to stable envelopes it reads as follows:
\begin{cor} \label{cor:stabind2} 
	Let $\lambda \in \Hom(\T,\C^*)\otimes\Q$ be a general enough fractional character and $w \in (G/B)^\T\simeq W$ a fixed point. Then
	\begin{align} \label{w:3}
		q^{1/2}\cdot\St^{\lambda}(sw)={}\DLLq {s,w\lambda}(\St^\lambda(w))\,,
	\end{align}
	where $\DLLq {s,w\lambda}$ denotes the operator $\DLL {s,w\lambda}$ after the substitution $y=-q$.
\end{cor}

Stable envelopes depend on a choice of a coweight chamber $\mathfrak{C}$. Up to this point we considered only the positive coweight chamber $\mathfrak{C_+}$.
All other chambers are of the form $\sigma\mathfrak{C_+}$ for some $\sigma\in W$.
It follows from the naturality of stable envelopes (\cite[Lemma 8.2 (a)]{AMSS}) that for Weyl group elements $\sigma,w\in W$, arbitrary slope $\lambda$ and arbitrary coweight chamber $\mathfrak{C}$ we have
\begin{align} \label{w:5}
	\sigma^{\textsc l}(\operatorname{stab}_{\mathfrak{C}}^{\lambda}(w))=
	\operatorname{stab}_{\sigma\mathfrak{C}}^{\lambda}(\sigma w)\,.
\end{align}
For any choice of the coweight chamber $\mathfrak{C}$ and slope $\lambda$ the classes
$\{\StC_{\mathfrak{C}}^\lambda(w)\}_{w\in W}$ form a base of the localized K-theory $$S^{-1}\KTh_\T(G/B)[q^{1/2},q^{-1/2}]$$
over the localized Laurent polynomials ring
$$S^{-1}\KTh_\T(pt)[q^{1/2},q^{-1/2}]\,.$$
Let $\mathfrak{C}_1$ and $\mathfrak{C}_2$ be two coweight chambers. The wall-crossing formula (or the R-matrix) describes the base change matrix from the basis $\{\StC_{\mathfrak{C}_1}^{\lambda}(w)\}_{w\in W}$ to $\{\StC_{\mathfrak{C}_2}^{\lambda}(w)\}_{w\in W}$ (see \cite[Paragraph~9.2.11]{O2}). Recursion \eqref{w:3} may be used to obtain such a formula.

\begin{atw} \label{tw:wallL}
	 Let $\lambda \in \Hom(\T,\C^*)\otimes\Q$ be a general enough fractional character and $w,\sigma \in W$  Weyl group elements. Let $s \in W$ be a simple reflection. Then
	$$\operatorname{stab}_{\sigma s\mathfrak{C_+}}^{\lambda}(w)=
	 \frac{(1-q)\cdot\sigma\alpha_s^{\lfloor \langle w\lambda,\sigma\alpha_s \rangle\rfloor}}
 {1-q\sigma\alpha^{-1}_s} \operatorname{stab}_{\sigma \mathfrak{C_+}}^\lambda(\sigma s\sigma^{-1}w)
	 +\frac{q^{1/2}(1-\sigma\alpha^{-1}_s)}{1-q\sigma\alpha^{-1}_s}\operatorname{stab}_{\sigma \mathfrak{C_+}}^{\lambda}(w) \,.
	$$
\end{atw}
\begin{proof}
	Let $\tilde{w}=s\sigma^{-1}w$. Corollary \ref{cor:stabind2} for $\tilde{w}$ implies that
	$$q^{1/2}\cdot\St^{\lambda}(s\tilde{w})
	=\frac{1-q\alpha^{-1}_s}{1-\alpha^{-1}_s}\cdot s^{\textsc l}(\St^{\lambda}(\tilde{w}))
	-\frac{(1-q)\cdot\alpha_s^{-\lceil \langle\tilde{w}\lambda,\alpha_s\rangle \rceil}}{1-\alpha^{-1}_s} \cdot \St^{\lambda}(\tilde{w})\,, $$
	which may be rewritten as
	$$
	s^{\textsc l}(\St^{\lambda}(\tilde{w}))=
	\frac{(1-q)\cdot\alpha_s^{-\lceil \langle\tilde{w}\lambda,\alpha_s\rangle\rceil}} {1-q\alpha^{-1}_s} \St^\lambda(\tilde{w})+
	\frac{q^{1/2}(1-\alpha^{-1}_s)}{1-q\alpha^{-1}_s}\St^{\lambda}(s\tilde{w}) \,.
	$$
	We apply the left Weyl group action of $\sigma^{\textsc l}$ to this equation and use formula \eqref{w:5}. Note that $\sigma^{\textsc l}$ acts also on the coefficients $\alpha_s\in \KTh_\T(pt)$. We obtain
	$$
	\StC_{\sigma s\mathfrak{C}_+}^{\lambda}(\sigma s\tilde{w})=
	\frac{(1-q)\cdot\sigma\alpha_s^{-\lceil \langle\tilde{w}\lambda,\alpha_s\rangle\rceil}} {1-q\sigma\alpha^{-1}_s} \StC_{\sigma \mathfrak{C}_+}^\lambda(\sigma\tilde{w})
	+\frac{q^{1/2}(1-\sigma\alpha^{-1}_s)}{1-q\sigma\alpha^{-1}_s}\StC_{\sigma \mathfrak{C}_+}^{\lambda}(\sigma s\tilde{w}) \,.
	$$
	Substitution $\tilde{w}=s\sigma^{-1}w$ completes the proof. The exponent
$${-\lceil \langle s\sigma^{-1}w\lambda,\alpha_s \rangle\rceil}
={-\lceil -\langle \sigma^{-1}w\lambda,\alpha_s \rangle\rceil}$$ is equal to
$\lfloor \langle w\lambda,\sigma\alpha_s \rangle\rfloor$.
\end{proof}

The above theorem describes the wall-crossing formula from the coweight chamber $\sigma\mathfrak{C_+}$ to $\sigma s\mathfrak{C_+}$. All chambers are of the form $\sigma'\mathfrak{C_+}$ for some $\sigma'\in W$. Therefore, repetitive use of the above theorem determines the wall-crossing formula between arbitrary coweight chambers.

\section{Wall-crossing formula for a change of the slope} \label{s:Rmat} 
As an application of the action of the left and right Demazure-Lusztig operators we will give an easy proof of 
the wall-crossing formula for the slope.
\bigskip

For a root $\alpha$ and an integer $n\in \Z$ let
$$H_{\alpha,n}=\{\lambda\in \ttt^*|\langle\lambda,\alpha^\vee\rangle=n\}\subset \Hom(\T,\C^*)\otimes_\Z\R \simeq \Pic(G/B)\otimes_\Z\R \,.$$
\begin{adf}
	An alcove is a connected component of the complement of the union of the hyperplanes~$H_{\alpha,n}$.
\end{adf}
Let $\lambda_1$ and $\lambda_2$ be slopes belonging to two adjacent alcoves. The slope R-matrix  is the base change matrix from the basis $\{\St^{\lambda_1}(w)\}_{w\in W}$ to $\{\St^{\lambda_2}(w)\}_{w\in W}$ (see \cite[Section~2.2]{OS}).
In \cite{SZZ2} the slope R-matrix for $G/B$ was computed.
\begin{atw}[{\cite[Theorems 4.1 and 5.1]{SZZ2}}] \label{tw:SZZ} 
	Let $\alpha$ be a positive root and $s$ the corresponding reflection. Suppose that alcoves $\nabla_1$ and $\nabla_2$ are adjacent and separated by the wall $H_{\alpha,0}$ and the functional $\langle-,\alpha^\vee\rangle$ is positive on $\nabla_2$. For any fixed point $w\in  (G/B)^\T\simeq W$ we have
	$$
	\St^{\nabla_1}(w)=
	\begin{cases}
		\St^{\nabla_2}(w)+ (q^{1/2}-q^{-1/2})\cdot\St^{\nabla_2}(ws) & \text{ if } l(w)>l(ws)\,, \\
		\St^{\nabla_2}(w) & \text{ if } l(ws)>l(w)\,. \\
	\end{cases}
	$$
\end{atw}

We propose an alternative approach to this result based on the geometry of Bott-Samelson resolution. Theorem \ref{tw:otoczki} 
states that stable envelopes can be computed using twisted motivic Chern classes. We will prove that
\begin{atw} \label{tw:2}
	Let $\alpha$ be a positive root and $s$ the corresponding reflection. Suppose that alcoves $\nabla_1$ and $\nabla_2$ are adjacent and separated by the wall $H_{\alpha,0}$ and the functional $\langle-,\alpha^\vee\rangle$ is positive on $\nabla_2$.  Let $w\in (G/B)^\T\simeq W$ be a fixed point. Then
	\begin{enumerate}[{a)}]
		\item The twisted motivic Chern class is constant on  alcoves, i.e. if slopes $\lambda_1,\lambda_2$ belong to the same alcove, then
		$$\mC(w,\lambda_1)=\mC(w,\lambda_2)\,.$$
		\item Choose slopes $\lambda_1\in\nabla_1$ and $\lambda_2\in\nabla_2$. Suppose that $l(ws)>l(w)$ then
		$$\mC(w,\lambda_1)=\mC(w,\lambda_2)\,.$$
		\item Choose slopes $\lambda_1\in\nabla_1$ and $\lambda_2\in\nabla_2$. Suppose that $l(ws)<l(w)$ then
		$$\mC(w,\lambda_1)=\mC(w,\lambda_2)- (-y)^{\frac{1}{2}(l(w)-l(ws)-1)}(1+y)\cdot \mC(ws_{\alpha},\lambda_2)\,.$$
	\end{enumerate}
\end{atw}
\begin{rem}
	The above theorem is equivalent to Theorem \ref{tw:SZZ}.
\end{rem}

\begin{rem}
	In Theorems \ref{tw:SZZ} and \ref{tw:2} only the walls $H_{\alpha,0}$ are considered. The formulas for general walls $H_{\alpha,n}$ can be deduced from the periodicity of the twisted motivic Chern classes and the stable envelopes (cf. \cite[Lemma 8.2(a)]{AMSS} and \cite[Remark 2.4]{KonW}). For an integral weight $\mu\in \Hom(\T,\C^*)$ we have
	\begin{align*}
		\mC(w,\lambda+\mu)&=\LL(\mu)_{|w}^{-1}\LL(\mu)\cdot\mC(w,\lambda)\,, \\
		\St^{\nabla+\mu}(w)&=\LL(\mu)_{|w}^{-1}\LL(\mu)\cdot\St^{\nabla}(w) \,.
	\end{align*}
	See \cite[Corollary 5.3 and the discussion after Lemma 3.7]{SZZ2} for a detailed account.
\end{rem}

\begin{proof} [Proof of parts $a)$ and $b)$ of Theorem \ref{tw:2}]
	Choose a reduced world decomposition $\wu$ of $w$. By definition the twisted motivic Chern class $\mC(w,\lambda)$ is determined by the pullback bundle $p_\wu^*\LL(\lambda)$. This bundle is described by the Chevalley formula. The conclusion follows from Proposition \ref{pro:Chev}.
\end{proof}
The rest of this section is devoted to the proof of part $c)$ of Theorem \ref{tw:2}.
Due to the part $a)$ we may assume that $\lambda_2=s\lambda_1$.
\begin{lemma} \label{lem:Wall0}
	Consider the situation described in the part $c)$. Suppose that s is a simple reflection. Then the theorem holds, i.e.
	$$
	\mC(w,\lambda_1)=
	\mC(w,s\lambda_1)- (1+y)\mC(ws,s\lambda_1)\,.
	$$
\end{lemma}
\begin{proof} 
	Let $\Ls$ be the relative tangent bundle of the projection $G/B\to G/P_{s}$. Theorem \ref{tw:1} implies that
	\begin{align}
		\label{eq:1}
		\mC(w,\lambda_1)&=
		s^{\textsc r}\left(\mC(ws,s\lambda_1)\right) \cdot \frac{1+y\Ls^*}{1-\Ls}
		-\mC(ws,s\lambda_1) \cdot\frac{1+y}{1-\Ls}\,, \\
		\label{eq:2}
		\mC(w,s\lambda_1)&=
		s^{\textsc r}\left(\mC(ws,\lambda_1)\right) \cdot \frac{1+y\Ls^*}{1-\Ls}
		-\mC(ws,\lambda_1) \cdot\frac{(1+y)\cdot\Ls}{1-\Ls}\,.
	\end{align}
	Moreover, part $b)$ implies that $\mC(ws,\lambda_1)=\mC(ws,s\lambda_1)$.
	Subtracting equations \eqref{eq:1}--\eqref{eq:2} we obtain the desired equality.
\end{proof}
\begin{lemma} \label{lem:Wall2}
	Suppose $r,\tilde{s}\in W$ are reflections such that $s=r\tilde{s}r$, $r$ is simple and $r\neq \tilde{s}$. If part $c)$ of Theorem \ref{tw:2} holds for reflection $\tilde{s}$ then it also holds for $s$.
\end{lemma}
\begin{proof}
	Let
	\begin{align*}
		a&=\frac{1}{2}(l(w)-l(wr)-1)\,,&
		b&=\frac{1}{2}(l(wr)-l(wr\tilde{s})-1)\,,&
		c&=\frac{1}{2}(l(wr\tilde{s})-l(ws)+1)\,.
	\end{align*}
	Let $\alpha_r$ be the positive root defining the reflection $r$. We compute the difference using Corollary \ref{rem:tw1}
	$$\mC(w,\lambda_1)-\mC(w,s\lambda_1)
	=(-y)^a\left(\DL{r,r\lambda_1}\left(\mC(wr,r\lambda_1)\right)-\DL{r,rs\lambda_1}\left(\mC(wr,rs\lambda_1)\right)\right)\,.$$
	The exponents in the definition of $\DL{r,rs\lambda_1}$ and $\DL{r,r\lambda_1}$ (Definition \ref{df:DL}) are  equal to $\lceil -\langle r\lambda_1,\alpha_r\rangle\rceil$ and $\lceil -\langle rs\lambda_1,\alpha_r\rangle\rceil$ correspondingly.
	We have 
	$$-\langle r\lambda_1,\alpha_r^\vee\rangle=\langle \lambda_1,\alpha_r^\vee\rangle\,,
\qquad -\langle rs\lambda_1,\alpha_r^\vee\rangle=\langle s\lambda_1,\alpha_r^\vee\rangle\,.$$
	Since $\lambda_1$ and $s\lambda_1$ are separated from each other only by the wall $H_{\alpha,0}$ and $\alpha\neq\alpha_r$, thus
	$$\lceil\langle \lambda_1,\alpha_r^\vee\rangle\rceil=\lceil\langle s\lambda_1,\alpha_r^\vee\rangle\rceil\,,$$
	hence
	$$\lceil -\langle r\lambda_1,\alpha_r\rangle\rceil=\lceil -\langle rs\lambda_1,\alpha_r\rangle\rceil\,, \qquad\text{and}\qquad  \DL{r,r\lambda_1}=\DL{r,rs\lambda_1}\,.$$
	Therefore
	\begin{align*}
		\mC(w,\lambda_1)-\mC(w,s\lambda_1)
		=&(-y)^a\cdot\DL{r,r\lambda_1}\left(\mC(wr,r\lambda_1)-\mC(wr,\tilde{s}r\lambda_1)\right) \\
		=&(-1)\cdot(-y)^{a+b}\cdot(1+y)\cdot\DL{r,r\lambda_1}\left(\mC(wr\tilde{s},r\lambda_1)\right) \\
		=&(-1)\cdot(-y)^{a+b+c}\cdot(1+y)\cdot\mC(ws,\lambda_1)\,.
	\end{align*}
	The second equality follows from Theorem \ref{tw:2} for $\tilde{s}$ and the third from  Corollary \ref{rem:tw1}.
	Note that the assumption $l(wr)>l(wr\tilde{s})$ for c) is satisfied by Proposition \ref{lem:Wall1}. 
\end{proof}

\begin{proof}[Proof of the part c)]
	Choose $\sigma\in W$ of minimal length such that $s_0=\sigma^{-1}s\sigma$ is a simple reflection (cf. Proposition \ref{lem:Weyl}). Choose a reduced word decomposition $\underline{\sigma}$ of $\sigma$. Let $\sigma_{k} \in W$ be the product of the last $k$ letters of $\underline{\sigma}$. Let
	$$s_k=\sigma_{k} s_0 (\sigma_{k})^{-1}\,.$$
	We prove the theorem by induction on $k$. Lemma \ref{lem:Wall0} proves the inductive assumption for~$k=0$. Lemma \ref{lem:Wall2} proves the inductive step.
\end{proof}

\section{Upgrade from $\GL_n/B$ to $\End(\C^n)$}\label{sec:upgrade}
From now on (except the last section) we concentrate on the $A_n$-case. We show that the construction which we performed for the flag varieties can be lifted to the level of matrices.

\subsection{The Kirwan map}
In the study of the homogeneous spaces for $\SL_n$ it is convenient to consider $G=\GL_n$ instead of $\SL_n$.
Let $B$ be the standard Borel subgroup of $\GL_n$ consisting of the upper-triangular matrices.
Let $\T\subset \GL_n$ be the diagonal torus. Let $\alpha_1,\dots ,\alpha_n \in \ttt^*$ be the simple roots of $\SL_n$, i.e.
$$\alpha_k=(0,0,\dots,\underbrace{1,-1,}_{k,~~k+1}\dots,0)\in \R^n\simeq\ttt^* \,.$$
The root $\alpha_k$ corresponds to the simple reflection $s_k=(k,k+1)$. \\
 For $i\in \{1,2,\dots,n\}$ denote by $t_i$ the characters of $\T$ associated with coordinates. The equivariant K-theory of the point 
$$\KTh_\T(pt)=\Z[t_1^{\pm1},t_2^{\pm1},\dots, t_n^{\pm1}]$$
will be denoted by $\Z[\underline t^{\pm1}]$ for short.
The equivariant K-theory of the flag manifold 
$$\GL_n/B=\{V_0\subset V_1\subset\dots \subset V_n=\C^n\, :\,\dim V_i=i\text{ for }i\in \{0,1,\dots, n\}\}$$ 
is generated over $\Z[\underline t^{\pm1}]$ by the classes of the tautological line bundles $V_i/V_{i-1}$. Let $\underline x=\{ x_i\}_{i=1,2,\dots,n}$ be a set of variables.  The surjection
\begin{align*}\KTh_{\T\times\T}(\End(\C^n))\simeq\Z[\underline t^{\pm1},\underline x^{\pm1}]&\to \KTh_\T(\GL_n/B)\,,\\
x_i\;&\mapsto \;[V_i/V_{i-1}]\end{align*}
has a geometric interpretation. Let $\T^2$ be the product of tori and let $\underline t$ and $\underline x$ be the sets of coordinate characters. The first factor of $\TD$ acts on $\End(\C^n)$ by the left multiplication, the second factor acts by the right multiplication. The composition of the restriction map and the natural isomorphism
\begin{equation}\label{kirwan}\Z[\underline t^{\pm1},\underline x^{\pm1}]\simeq \KTh_{\TD}(\End(\C^n))\longrightarrow \KTh_{\TD}(\GL_n) \simeq \KTh_\T(\GL_n/B)\end{equation}
will be denoted by $\kappa$ and called the Kirwan map.
Let $L_{i,x}=\frac{x_{i+1}}{x_i}$ and $L_{i,t}=\frac{t_{i+1}}{t_i}$ for $i\in\{1,2,\dots,n-1\}$. We have \begin{align} \label{w:kirwan}
	\kappa(L_{i,t})=
	\LL(\alpha_i)\,, \qquad \kappa(L_{i,x})=
	\alpha^{-1}_i\,.
\end{align}
The bundle $\LL(\alpha_i)$ is the line bundle tangent  to the fibers of the elementary contraction $G/B\to G/P_i$. 
It is induced from the representation $\C_{-\alpha_i}$, see  Section \ref{notation}. For computation it is convenient to use the composition of $\kappa$ with the restriction to the fixed point set.  The composition is of the form
$$\overline\kappa:\Z[\underline t^{\pm 1},\underline x^{\pm 1},y]\to \KTh_\T((\GL_n/B)^\T[y])=\bigoplus_{\sigma\in S_n} \Z[\underline t^{\pm 1},y]\,,$$
$$f(\underline t,\underline x,y)\mapsto \bigoplus_{\sigma \in S_n}f(\underline t,\underline x,y)_{x_i:=t_{\sigma(i)}}\,.$$
Here $S_n$ denotes the group of permutations of $n$ elements, the Weyl group of $\GL_n$.
For a detailed description of the equivariant K-theory of the flag variety see e.g.~\cite{Uma}.

\subsection{Matrix Schubert varieties}
The action of $B\times B$ on $\End(\C^n)$ is given by 
$$(b_1,b_2)\cdot a=b_1ab_2^{-1}\,.$$
We identify $\End(\C^n)$ with the space of $n\times n$ matrices
 and we identify permutations with their matrices.  For $w\in S_n$ we denote by $\O_w$ the $B\times B$-orbit  of $w$. Let $\XX_w=\overline{\O_w}$ be its closure.  It is a maximal rank matrix Schubert variety. In general,
by the matrix Schubert variety we understand the closure of a $B\times B$ orbit in $\End(\C^n)$.
Let $pr:\GL_n\to \GL_n/B$ be the quotient map. Then 
$$\XX_w\cap \GL_n=pr^{-1}(X_w)\,.$$
We will discuss here only the maximal rank orbits. They are exactly the orbits of the permutation matrices. Matrix Schubert varieties and their characteristic classes in K-theory or cohomology of the vector space $\End(\C^n)$ were widely studied.
The coincidence of  Schubert polynomials with fundamental classes of matrix Schubert varieties was discovered in \cite{FR2}. Later, another proof based on the divided difference operators was given in \cite{AE}.
For more details see \cite[Chapter 15]{MillerSturmfels}.

\subsection{The lift of boundary divisors}
\label{divdef}
Let \hbox{$s_k=(k,k+1)$} be a simple reflection and let $w_0$ be the longest permutation. The Schubert variety $X_{w_0s_k}\subset \GL_n/B$ is of codimension one. By the Chevalley formula for a fractional character $\lambda\in\Hom(\T,\C^*)\otimes_\Z\Q\simeq\Q^n$ the line bundle $\LL(\lambda)$ can be written (nonequivariantly) in terms of Schubert divisors as
$$\LL(\lambda)=\O_{\GL_n/B}\left( \sum_{i=k}^{n-1}(\lambda_k-\lambda_{k+1}) X_{w_0s_k}\right)\,.$$
The analogous formula for the restriction of $\LL(\lambda)$ to Schubert variety $X_w$ is given e.g.~in \cite[Proposition 1.4.5]{BrionLec}.
We will repeat this construction in $\End(\C^n)$.
\medskip

For a matrix $a=\{a_{i,j}\}_{1\leq i,j\leq n}$ and $1\leq k\leq n$ let $$m_k(a)=\det\left(\{a_{i,j}\}_{1\leq i,j\leq k}\right)$$ be the first principal $k$-minor. Define 
\begin{equation}\label{defM}M_{w,k}=\{a\in \XX_w\;|\;m_k(w^{-1}a)=0\}\,.\end{equation}

\begin{rem}
The Schubert divisor $X_{w_0s_k}\subset \GL_n/B$  is pulled back from the Grassmannian $Gr_k(\C^n)$, hence it is the zero locus of a section of the bundle $\Lambda^kV_k$.
The divisor $M_{w_0,k}$ restricted to $\GL_n$ is the inverse image of the Schubert divisor in $X_{w_0s_k}$
$$M_{w_0,k}\cap\GL_n= pr^{-1}(X_{w_0s_k})\,.$$
The divisor $M_{w,k}$ is the restriction to $\XX_w$ of a left shift of $M_{w_0,k}$.
\end{rem}

\begin{pro} For $w\in S_n$ the divisors $M_{w,k}$ are contained in the boundary of $\XX_w$, i.e.~they do not intersect $\mathcal O_w$.\end{pro}
\begin{proof}First let us note that for $a\in\End(\C^n)$, $b\in B$  we have
\begin{equation}\label{gal1}m_k(ab)=0\;\Longleftrightarrow\;m_k(a)=0\,.\end{equation}
Also for $b\in B$, $w\in S_n$
\begin{equation}\label{gal2}m_k(w^{-1}bw)\neq 0\,.\end{equation}
We show that  $m_k(w^{-1}a)\neq 0$ for $a=b_1wb_2^{-1}$, $b_1,b_2\in B$. Since
$$m_k(w^{-1}b_1wb_2^{-1})=0\;\Longleftrightarrow\;m_k(w^{-1}b_1w)=0$$
and $b_1\in B$ by (\ref{gal1}--\ref{gal2}) we obtain the claim.
\end{proof}

For a fixed sequence of rational numbers $\lambda=(\lambda_1,\lambda_2,\dots,\lambda_n)$ let 
\begin{equation}\label{dywizor}D_{w,\lambda}=\sum_{k=1}^n (\lambda_{k}-\lambda_{k+1})M_{w,k}\,.\end{equation}
We set for convenience $\lambda_{n+1}=0$. In this way  for $w=\id$ 
$$D_{\id,\lambda}=\dv\Big(\prod_{i=1}^n a_{ii}^{\lambda_i}\Big)\subset \XX_{\id}=\overline B\,.$$
Here $\dv(f)$ stands for the $\Q$-divisor of the formal expression with rational exponents. 
The divisor $D_{w,\lambda}$ is defined so that
\begin{align} \label{w:1}
	(D_{w,\lambda})_{|\GL_n\cap \XX_w}=pr^{*}(\Delta_{w,\lambda})\,.
\end{align}
See Section \ref{s:mC} for the definition of $\Delta_{w,\lambda}$.

\begin{ex} For $n=2$, $S_2=\{\id,s_1\}$
\begin{align*}D_{\id,\lambda}&=(\lambda_1-\lambda_2)\dv(a_{11})+ \lambda_2\dv\det\!\Bigl(\hbox{\footnotesize$\begin{matrix}a_{11}&\hskip-0.7em a_{12}\\0&\hskip-0.7em a_{22}\end{matrix}$}\Bigr)\\
&=\lambda_1\dv(a_{11})+ \lambda_2\dv(a_{22})\,,\\
 D_{s_1,\lambda}&=(\lambda_1-\lambda_2)\dv(a_{21})+\lambda_2\dv\det\!\Bigl(\hbox{\footnotesize$\begin{matrix}a_{11}&\hskip-0.7em a_{12}\\a_{21}&\hskip-0.7em a_{22}\end{matrix}$}\Bigr)\,.\end{align*}
The boundary of the open orbit consists of the degenerate or upper-triangular matrices
$$\partial \XX_{s_1}=\{a\in \End(\C^2)\;:\; \det(a)=0\;\vee\;a_{21}=0\}\,.$$
Consider the map resolving the singularity of the boundary 
 $$\mu:P_1\times_B\overline B\to \End(\C^2)\,.$$ In one of the standard affine charts it is given by the formula 
$$\mu(c,b_{11},b_{12},b_{22})=\begin{pmatrix}1&0\\c&1\end{pmatrix}\cdot \begin{pmatrix}b_{11}&b_{12}\\0&b_{22}\end{pmatrix}=
\begin{pmatrix}b_{11}&b_{12}\\c\,b_{11}&c\,b_{12}+b_{22}\end{pmatrix}\,.$$
Then
\begin{align*}\mu^*(D_{s_1,\lambda})&=\dv\big( (c\,b_{11})^{\lambda_1-\lambda_2}(b_{11}b_{22})^{\lambda_2}\big)
\\&=
 (\lambda_1-\lambda_2)\dv(c)+\lambda_1\dv(b_{11})+\lambda_2\dv(b_{22})\,.\end{align*}

\end{ex}

\section{Universal algebra} \label{s:universal}
We define an algebra acting on the ring of Laurent polynomials. We will show that the generating operations satisfy the same quadratic and braid relations, as those acting on the level of K-theory of $\GL_n/B$.
The described algebra can be compared with the algebra generated by the Demazure-Lusztig operators which computes the motivic Chern classes of matrix Schubert varieties. 
\subsection{The small algebra}\label{smallalgebra}
Let us define operations 
$$\Tfr{i}{a}\colon \Z[z_1^{\pm 1},z_2^{\pm 1},\dots,z_n^{\pm 1},y]\longrightarrow \Z[z_1^{\pm 1},z_2^{\pm 1},\dots,z_n^{\pm 1},y]$$ 
on Laurent polynomials in $z_1,z_2,\dots, z_n$. The operations depend on $i\in\{1,2,\dots,n-1\}$ and $a\in \Q$. For $f\in \Z[z_1^{\pm 1},z_2^{\pm 1},\dots,z_n^{\pm 1},y]$ let 
$$L_i=\frac{z_{i+1}}{z_i}$$
and
$$\Tfr{i}{a}(f)
=-\frac{(1+y) L_i^{\lceil a\rceil}\cdot f- \big(1+y\, L_i^{-1}
\big)\cdot s_i f}{1-L_i }
\,,$$
where 
$$s_i f(z_1,z_2,\dots, z_i,z_{i+1},\dots,z_n,y)=f(z_1,z_2,\dots, z_{i+1},z_i,\dots,z_n,y)\,.$$

If we set $z_i=z_{i+1}$ then the numerator is equal to 0, therefore, it is divisible by $z_i-z_{i+1}$, which is equivalent to divisibility in Laurent polynomials by $1-\frac{z_{i}}{z_{i+1}}$. Hence the operators $\Tfr{i}{a}$ transform  Laurent polynomials into Laurent polynomials. The operators $\Tfr{i}{a}$ can be treated as variants of the isobaric divided differences, see e.g.~\cite[Section 2.1]{RimanyiSzenes}. Another way of writing the operators~$\Tfr{i}{a}$ is the following
$$\Tfr{i}{a}(f)=
\frac{\big(1+y\,L_i^{-1}
\big)\cdot s_i f}{1-L_i }
+
\frac{(1+y) L_i^{\lceil a\rceil-1}\cdot f}{1-L_i^{-1} }
\,.$$
This form is adapted to the localization formula.
A slight modification of the operator~$\Tfr{i}{a}$ will be useful 
\begin{align*}\Tfl{i}{a}(f)&= -\frac{(1+y) L_i^{\lceil a\rceil}\cdot f-\big(1+y\,L_i
\big)\cdot s_i f}{1-L_i }\\
&=\frac{\big(1+y\,L_i
\big)\cdot s_i f}{1-L_i }
+
\frac{(1+y) L_i^{\lceil a\rceil-1}\cdot f}{1-L_i^{-1} }
\,.\end{align*}
It is elementary to note that
$$\Tfl{i}{a}((1+y\,L_i)f)=(1+y\,L_i)~\Tfr{i}{a}(f)\,,$$
i.e. $\Tfl{i}{a}$ differs from $\Tfr{i}{a}$ by the conjugation with $(1+y\,L_i)$.

The properties described below hold for both $\Tfr{i}{a}$ and $\Tfl{i}{a}$, therefore we will simply use the notation $\Tf{i}{a}$.

\begin{lemma}[Quadratic relation]\label{quadratic}
Suppose $a\notin \Z$, then
$$\Tf{i}{-a}\circ \Tf{i}{a}=-y \,\id\,.$$
If $a\in \Z$, then
$$\Tf{i}{1-a}\circ \Tf{i}{a}=-y \,\id\,.$$
\end{lemma}

\begin{proof}The proof is identical as in Proposition \ref{pro:kwadrat}.\end{proof}

\begin{rem} For a formula involving the composition $\Tf{i}{b}\circ \Tf{i}{a}$ for arbitrary $a,b\in \Q$ see Theorem \ref{kwogolny}.\end{rem}

The operators $\Tf{i}{a}$ satisfy the  braid relation with parameters:

\begin{lemma}\label{universal_braid} For $\lambda_1,\lambda_2,\lambda_3\in \Q$ we have
$$\Tf{i}{\lambda_3 - \lambda_2}\circ \Tf{i+1}{\lambda_3 - \lambda_1}\circ \Tf{i}{\lambda_2 - \lambda_1}= 
    \Tf{i+1}{\lambda_2 - \lambda_1}\circ \Tf{i}{ \lambda_3 - \lambda_1}\circ \Tf{i+1}{\lambda_3 - \lambda_2}\,.$$
\end{lemma}

\begin{rem}The vector $\lambda=(\lambda_1,\lambda_2,\lambda_3)$ can be understood as a weight for the maximal torus of $\GL_3$, then the above formula coincides with the formula from Example \ref{przyklad_sl3}.
\end{rem}

\begin{proof} The formula can be checked by hand (or rather using computer algebra software). We can  assume that $n=3$, $i=1$. Since $\lceil \lambda_i-\lambda_j\rceil$ appears in the formulas, one is led to consider various cases depending on the fractional parts. Setting $a=\lambda_1-\lambda_2$ and $b=\lambda_2-\lambda_3$ it is enough to check two cases assuming for $a,b\in \Z$
$$\Tf{1}{ a}\circ \Tf{2}{ a + b}\circ \Tf{1}{ b}=
   \Tf{2}{b}\circ \Tf{1}{ a + b}\circ \Tf{2}{a}$$
and
$$\Tf{1}{ a}\circ \Tf{2}{ a + b+1}\circ \Tf{1}{ b}=
   \Tf{2}{b}\circ \Tf{1}{ a + b+1}\circ \Tf{2}{ a}\,.$$
Verification is an elementary, although lengthy, calculation.
\medskip

There is an alternative method of the proof, not demanding direct calculation, but relying on the fact that the Demazure-Lusztig operators generate an action of the Hecke algebra on the K-theory of $\GL_n/B$. 
It will be presented after Theorem \ref{podniesienie}.
\end{proof}

\subsection{Operators with parameters in $\ttt^*_n\subset\mathfrak{gl}_n^*$}
 \label{abstract_algebra}
\begin{adf}\label{df:univTcr}
	For $\lambda\in \ttt^*_n$ we define the operator $\Tcr{i}{\lambda}$ acting on $\Z[\underline x^{\pm1},y]$
$$\Tcr{i}{\lambda}=\Tfr{i}{-\langle\lambda,\alpha^\vee_i\rangle}\,,$$
after the substitution $x_i=z_i$. We extend the action of the operator $\Tcr{i}{\lambda}$ to 
$$\Tcr{i}{\lambda}\colon\Z[\underline t^{\pm 1},\underline x^{\pm 1},y]\to\Z[\underline t^{\pm 1},\underline x^{\pm 1},y]\,,$$
linearly with respect to $\Z[\underline t^{\pm 1}]$.
Explicitly for $f\in \Z[\underline x^{\pm 1},\underline t^{\pm 1},y]$ we have
\begin{equation}\label{opTR}\Tcr{i}{\lambda}(f)=\frac{(1+y)
   \big(\frac{x_i}{x_{i+1}}\big)
   {}^{1-\left\lceil \lambda_{i+1}-\lambda_i\right\rceil }
   }{1-\frac{x_i}{x_{i+1}}} f+
   \frac{
   \big(1+y\frac{x_i
   }{x_{i+1}}\big)
   }{1-\frac{x_{i+1}}{x_i}} s_i^xf\,,
   \end{equation}
where the reflection $s_i^x$ acts on $x$-variables, switching $x_i$ with $x_{i+1}$. 
\end{adf}

\begin{rem} The assignment $\lambda\mapsto \Tcr{i}{\lambda}$ is constant on the alcoves in $\ttt^*$.\end{rem}

Directly from Lemmas \ref{quadratic} and \ref{universal_braid} we obtain:  
\begin{pro}\label{uniwersalny_braid}
For a generic $\lambda$ we have 
$$\Tcr{i}{\lambda}\circ \Tcr{i}{s_i\lambda}=-y \,\id\,.$$
and for arbitrary $\lambda$
$$\Tcr{i}{s_{i+1}s_{i}\lambda}\circ \Tcr{i+1}{s_i \lambda}\circ \Tcr{i}{\lambda}= 
 \Tcr{i+1}{s_{i}s_{i+1}\lambda}\circ \Tcr{i}{s_{i+1} \lambda}\circ \Tcr{i+1}{\lambda}\,.$$
\end{pro}

We define the left operator using $\Tcl{i}{\lambda}$  acting on $t$--variables. 

\begin{adf}\label{df:univTcl}
		For $\lambda\in \ttt^*_n$ we define the operator $\Tcl{i}{\lambda}$ acting on $\Z[\underline t^{\pm1},y]$
	$$\Tcl{i}{\lambda}=\Tfl{i}{\langle\lambda,\alpha^\vee_i\rangle}\,,$$
	after the substitution $t_i=z_i$. We extend the action of the operator $\Tcl{i}{\lambda}$ to 
	$$\Tcl{i}{\lambda}\colon\Z[\underline t^{\pm 1},\underline x^{\pm 1},y]\to\Z[\underline t^{\pm 1},\underline x^{\pm 1},y]\,,$$
	linearly with respect to $\Z[\underline x^{\pm 1}]$.
	Explicitly for $f\in \Z[\underline x^{\pm 1},\underline t^{\pm 1},y]$ we have
\begin{equation}
	\label{opTL}\Tcl{i}{\lambda}(f)=\frac{(1+y)
		\big(\frac{t_i}{t_{i+1}}\big)
		{}^{1-\lceil\lambda_{i}-\lambda_{i+1}\rceil
	}}{1-\frac{t_i}{t_{i+1}}}f+\frac{
		\big(1+y\frac{t_{i+1}}{t_i}\big)}{1-\frac{t_{i+1}}{t_i}} s_i^t(f)\,.
\end{equation}
	where the reflection $s_i^t$ acts on $t$-variables, switching $t_i$ with $t_{i+1}$.
\end{adf}

Directly from Lemmas \ref{quadratic} and \ref{universal_braid}  
we obtain braid and quadratic relations for $\Tcl{i}{\lambda}$ operators.
\begin{pro}\label{universalny_braid_l}For a generic $\lambda$ we have 
$$\Tcl{i}{s_i\lambda}\circ \Tcl{i}{\lambda}=-y \,\id\,.$$
and for arbitrary $\lambda$
$$\Tcl{i}{s_{i+1}s_i\lambda}\circ \Tcl{i+1}{s_i \lambda}\circ \Tcl{i}{\lambda}= 
 \Tcl{i+1}{s_is_{i+1}\lambda}\circ \Tcl{i}{s_{i+1} \lambda}\circ \Tcl{i+1}{\lambda}\,.$$
\end{pro}

\begin{ex}
	For $n=2$ and $\lambda=(\lambda_1,\lambda_2)$
	we have
	\begin{multline*}\Tcr{1}{\lambda}(f)(t_1,t_2,x_1,x_2,y)=
		\\=
		\frac{(1+y)
			\big(\frac{x_1}{x_2}\big)
			{}^{1-\left\lceil \pmb{\color{purple}\lambda_2-\lambda_1}\right\rceil }
		}{1-\frac{x_1}{x_2}}f\left(t_1,t_2,x_1,x_2,y\right)+\frac{
			1+y{\color{purple}\pmb{\frac{x_1
					}{x_2}}}}{1-\frac{x_2}{x_1}}
		f\left(t_1,t_2,x_2,x_1,y\right)\,,\end{multline*}
	while 
	\begin{multline*}\Tcl{1}{\lambda}(f)(t_1,t_2,x_1,x_2,y)=
		\\=\frac{(1+y)
			\big(\frac{t_1}{t_2}\big)
			{}^{1-\left\lceil \pmb{\color{purple}\lambda_1-\lambda_2}\right\rceil }}{1-\frac{t_1}{t_2}}
		f\left(t_1,t_2,x_1,x_2,y\right)
		+
		\frac{1+y\pmb{\color{purple}\frac{t_2
				}{t_1}}}{1-\frac{t_2}{t_1}}f\left(t_2,t_1,x_1,x_2,y\right)\,.\end{multline*}
\end{ex}

\subsection{Comparison with $\KTh_\T(\GL_n/B_n)$}

The main reason why we introduced the operators $\Tcr{i}{\lambda}$ and $\Tcl{i}{\lambda}$ is the following theorem:

\begin{atw}\label{podniesienie} The operators $\Tcr{i}{\lambda}$ are  lifts of the operators $\DL {s_i,\lambda}$ and the operators  $\Tcl{i}{\lambda}$ are  lifts of the operators $\DLL {s_i,\lambda}$:
\[\begin{matrix}\begin{tikzcd}
    \Z[\underline t^{\pm 1},\underline x^{\pm 1},y]
    \arrow[r,"\Tcr{i}{\lambda}"]\arrow[d,"\kappa"'] & \Z[\underline t^{\pm 1},\underline x^{\pm 1},y]\arrow[d,"\kappa"] \\
    \KTh_\T(\GL_n/B)[y]\arrow[r,"\DL {s_i,\lambda}"]&\KTh_\T(\GL_n/B)[y],
\end{tikzcd}
&~~&
\begin{tikzcd}
    \Z[\underline t^{\pm 1},\underline x^{\pm 1},y]
    \arrow[r,"\Tcl{i}{\lambda}"]\arrow[d,"\kappa"'] & \Z[\underline t^{\pm 1},\underline x^{\pm 1},y]\arrow[d,"\kappa"] \\
    \KTh_\T(\GL_n/B)[y]\arrow[r,"\DLL {s_i,\lambda}"]&\KTh_\T(\GL_n/B)[y].
\end{tikzcd}
\end{matrix}
\]
\end{atw}  

\begin{proof}
We compare the equations \eqref{opTR} and \eqref{opTL}  with Definition \ref{df:DL}. We use the fact that the Kirwan map commutes with both Weyl group actions, i.e.
$$\kappa\circ s_i^x= s^{\textsc r}_{\alpha_i}\circ\kappa\,, \qquad \kappa\circ s_i^t= s^{\textsc l}_{\alpha_i}\circ\kappa\,,$$
and  formula \eqref{w:kirwan}, i.e.
$$\kappa(L_{i,t})=\kappa\left(\frac{t_{i+1}}{t_i}\right)=\alpha_i^{-1}\,, \qquad
\kappa(L_{i,x})=\kappa\left(\frac{x_{i+1}}{x_i}\right)=\LL(\alpha_i)\,.
$$
\end{proof}

\begin{rem}We have four different types of operators and we denote  them using different fonts: 
\begin{enumerate}[label=(\roman*)]

\item $\DL{s,a}$ and $\DLL{s,a}$ act on the K-theory $\KTh_\T(G/B)[y]$ (Definitions \ref{df:DL}.1 and \ref{df:DL}.3). The parameter $a$ is a rational number. Real parameter makes sense as well. 
\item $\DL{s,\lambda}$ and $\DLL{s,\lambda}$ act on the K-theory $\KTh_\T(G/B)[y]$ for a semisimple group $G$, (Definitions \ref{df:DL}.2 and \ref{df:DL}.4). The group $\GL_n$ is allowed as well. The parameter $\lambda$ is a weight of the fixed maximal torus $\T\subset B$ and $s$ is a simple reflection.
\item $\Tfr{i}{a}$ and $\Tfl{i}{a}$ act on Laurent polynomials or rational functions, $a\in \Q$ is a parameter (Section \ref{smallalgebra}). 
\item $\Tcr{i}{\lambda}$ and $\Tcl{i}{\lambda}$ act on Laurent polynomials or rational functions, $\lambda\in \ttt^*=\R^n$ is a fractional character of the standard maximal torus in $\GL_n$ (Definitions \ref{df:univTcr} and \ref{df:univTcl}).
\end{enumerate}
\end{rem}
\begin{ex}\label{ex:zlystart}
Let $n=2$, 
and consider the composition of $\kappa$ with the restriction to the fixed points
$$\overline\kappa:\Z[t_1^{\pm 1},t_2^{\pm 1},x_1^{\pm 1},x_2^{\pm 1},y]\to \KTh_\T((\GL_2/B)^\T[y])=\bigoplus_{\sigma\in S_2} \Z[t_1^{\pm 1},t_2^{\pm 1},y]\,,$$
$$f(t_1,t_2,x_1,x_2,y)\mapsto (\,f(t_1,t_2,t_1,t_2,y)\,,\,f(t_1,t_2,t_2,t_1,y)\,)\,.$$
The polynomial 
$f_0=1-\tfrac{x_1}{t_2}$ is sent to $(1-\tfrac{t_1}{t_2},0)$, which is the restriction of the class of the point $X_{\id}$.
Applying $\Tcr{1}{\lambda}$ we obtain
$$\frac{(1+y) 
   \big(\frac{x_1}{x_2}\big){}^{1-\lceil \lambda_2-\lambda_1\rceil
   }}{1-\frac{x_1}{x_2}}\big(1-\frac{x_1}{t_2}\big)+
\frac{1+y\frac{x_1 }{x_2}}{1-\frac{x_2}{x_1}}\big(1-\frac{x_2}{t_2}\big)\,.$$
After restriction to the fixed point set we arrive at the formula for the twisted motivic Chern class of $\C\subset \PP^1$
$$\big(\,(1+y) \big(\tfrac{t_1}{t_2}\big){}^{1-\lceil \lambda_2-\lambda_1\rceil}\,,
\,1+y\tfrac{t_2 }{t_1}\,\big)\,.$$
By Theorem \ref{podniesienie} and the inductions given in Theorem \ref{tw:1}
$$\kappa(\Tcr{1}{\lambda}(f_0))=\DL {s_1,\lambda}(\kappa(f_0))=\mC(s_1,s_1\lambda)\,.$$
On the other hand applying $\Tcl{1}{s_1\lambda}$ we obtain
$$\frac{(1+y)
   \big(\frac{t_1}{t_2}\big)
   {}^{1-\left\lceil {\lambda_2-\lambda_1}\right\rceil }}{1-\frac{t_1}{t_2}}\left(1-\tfrac{x_1}{t_2}\right)
+\frac{1+y {\frac{t_2
   }{t_1}}
   }{1-\frac{t_2}{t_1}}\left(1-\tfrac{x_1}{t_1}\right
   )\,.$$
After restriction to the fixed point and simplifying we obtain
  the twisted motivic Chern class of $\C\subset \PP^1$
$$\big(\,(1+y) \big(\tfrac{t_1}{t_2}\big){}^{1-\lceil \lambda_2-\lambda_1\rceil}\,,
\,1+y\tfrac{t_2 }{t_1}\,\big)\,,$$
which  by Theorem \ref{tw:L1} is equal to
$$\kappa(\Tcl{1}{s_1\lambda}(f_0))=\DLL{s_1,s_1\lambda}(\kappa(f_0))=\mC(s_1,\lambda)\,.$$
Although we have equality 
$$\kappa(\Tcr{1}{\lambda}(f_0))=\kappa(\Tcl{1}{s_1\lambda}(f_0))\,,$$ but one can check that $$\Tcr{1}{\lambda}(f_0)\neq\Tcl{1}{s_1\lambda}(f_0)\,.$$
This discrepancy can be repaired by a different choice of the initial polynomial $f_0$.
A better choice for the beginning of the induction is proposed in the next section.
\end{ex}
	
\begin{proof}[Alternative proof of Lemma \ref{universal_braid}]
We know that the braid relations hold in $\KTh_\T(\GL_n/B)$, see Proposition \ref{pro:braid} and Example \ref{przyklad_sl3}. Take $n=3$ and apply the Kirwan map \eqref{kirwan}  composed with restriction to the fixed point corresponding to $[\id]\in\GL_3/B$
\begin{multline*}\Z[z_1^{\pm1},z_2^{\pm1},z_3^{\pm1},y]\stackrel{z_i\mapsto x_i}\hookrightarrow \KTh_{\T\times \T}(\End(\C^3))[y]\stackrel{\kappa}\longrightarrow \\ 
\stackrel{\kappa}\longrightarrow \KTh_\T(\GL_3/B)[y]
\stackrel{\iota^*}\longrightarrow \KTh_\T(pt)[y]=\Z[t_1^{\pm1},t_2^{\pm1},t_3^{\pm1},y]\,.\end{multline*}
The composition sends the variables $x_i$ to the equivariant variables $t_i$, therefore the map
$$\Z[z_1^{\pm1},z_2^{\pm1},z_3^{\pm1},y]\to \KTh_\T(\GL_3/B)[y]$$
is injective. 
The operators $\Tcr{i}{\lambda}$ are lifts of the twisted Demazure-Lusztig operators.
Hence the braid relations hold in $\Z[z_1^{\pm1},z_2^{\pm1},z_3^{\pm1},y]$.
\end{proof}

\section{Comparison with the matrix Schubert varieties} \label{s:matrixSchubert}
We will show that the purely algebraic definition of the operator $\Tcr{i}{\lambda}$ and $\Tcr{i}{\lambda}$ have geometric meaning. They describe how the twisted motivic Chern classes change when we pass from $\XX_w$ to $\XX_{ws}$ or to $\XX_{sw}$. We start with untwisted classes and reformulate the results of \cite{RudnickiW}.
\subsection{Corollaries from the work on square-zero matrices}

 A relation between motivic Chern classes of invariant subvarieties of a $G$-manifold and motivic Chern classes of  quotients is studied in  \cite[Section 8]{FRWproc}.  In our situation we consider $\GL_n/B$ as a quotient $\End(\C^n)/\!/B$ (quotient by the right action).
Let $$\BB=\mC(B\subset \overline B)$$
with $\T$ acting on $B$ by conjugation, expressed in $x$-variables.
Here $\overline B$ is the closure of $B$ in $\End(\C^n)$. The closure $\overline B$ is the vector space consisting of upper-triangular matrices. 
In terms of the coordinate torus characters 
\begin{equation}\label{eq:BB}\BB=(1+y)^n\prod_{1\leq i<j\leq n}(1+y\tfrac{x_j}{x_i})\in \KTh_\T(\overline B)[y]=
\Z[\underline x,y]
\,.\end{equation}
It follows that
\begin{pro}\cite[Theorems 8.9 and 8.12]{FRWproc}\label{dzielenie0} For a   subvariety $Y\subset  \End(\C^n)$ which is invariant with respect to the right $B$ action and left $\T$ action
$$\mC((Y\cap \GL_n)/B\subset \GL_n/B)=\kappa(\BB^{-1}\mCD(Y\subset \End(\C^n)))\,,$$
where $\kappa$ is the Kirwan map \eqref{kirwan}. 
\end{pro}
The formulation in \cite{FRWproc} is given in terms of Segre classes. That is why we divide the motivic Chern class by $\BB$. We will give an alternative proof,  which is valid for twisted motivic Chern classes as well, see Proposition \ref{dzielenie}.
\medskip

We can identify the space $\End(\C^n)$ with the subspace of square-zero upper-triangular matrices of the dimension $2n\times 2n$ of the block form {\footnotesize $\begin{pmatrix}0\ a\\0\,\ 0\end{pmatrix}$}.
The results of \cite{RudnickiW} for $B$-orbits of square-zero matrices apply. In particular the motivic Chern classes (for the trivial slope) can be computed inductively. We translate that result to our situation. 
The Euler class of   
$\End(\C^n)$ is equal to
$$\EE=\prod_{i=1}^n\prod_{j=1}^n(1-\tfrac{x_{j}}{t_i})\,.$$
The localized motivic Chern class of a $\T^2$--invariant subvariety $Y\subset \End(\C^n)$ is defined
as the quotient
$$\mClocD(Y\subset\End(\C^n))=\EE^{-1}\mCD(Y\subset\End(\C^n))\,.$$
By \cite[Corollary 5.4]{RudnickiW} for $i\in\{1,2,\dots,n-1\}$
\begin{equation}\label{lewy}\mClocD(\O_{s_iw}\subset \End(\C^n))=A_i^{t}(\mClocD(\O_{w}\subset \End(\C^n)))\,,\end{equation}
\begin{equation}\label{prawy}\mClocD(\O_{ws_i}\subset \End(\C^n))=A_i^x(\mClocD(\O_{w}\subset \End(\C^n)))\,.\end{equation}
The operations $A_i^t$ and $A_i^x$ are defined in a more general context of square zero matrices  by a uniform formula\footnote{Setting $t_{n+i}:=x_i$ we do not have to distinguish between $t$--variables and $x$--variables. Instead, in \cite{RudnickiW}, the superscript is used to indicate that we deal with K-theory.}. In our situation
$$A_i^t(f)=\frac{(1+y)\tfrac{t_{i}}{t_{i+1}}}{1-\tfrac{t_{i}}{t_{i+1}}}f+\frac{1+y\tfrac{t_{i+1}}{t_{i}}}{1-\tfrac{t_{i+1}}{t_{i}}}s^t_if\,,$$
$$A_i^x(f)=\frac{(1+y)\tfrac{x_{i}}{x_{i+1}}}{1-\tfrac{x_{i}}{x_{i+1}}}f+\frac{1+y\tfrac{x_{i+1}}{x_{i}}}{1-\tfrac{x_{i+1}}{x_{i}}}s^x_if$$
for $i\in\{1,2,\dots ,n-1\}$. Above, the reflections $s^t_i$  act on $t$--variables and  $s^x_i$ on $x$--variables.

Since $\EE$ is symmetric with respect to both sets of variables $\{t_i\}_{i=1,2,\dots,n}$ and $\{x_i\}_{i=1,2,\dots,n}$ separately and $\mCD(-)=\EE\cdot\mClocD(-)$, the localized classes in the formulas (\ref{lewy}--\ref{prawy}) can be replaced by the usual motivic Chern classes.
The operation $A_i^t$ coincides with $\Tcl{i}{0}$. It commutes with multiplication by $\BB$, since $\BB$ is expressed only by $x$-variables. On the other hand $A_{i}^x$ differs from $\Tcr{i}{0}$ by the conjugation with $1+y\frac{x_{i+1}}{x_i}$. 
For any $f$ we have
$$\Tcr{i}{0}(\BB^{-1} f)=\BB^{-1}A_{i}^x( f)\,. $$
Therefore we obtain
\begin{pro} Let $1\leq i\leq n-1$. If $l(s_iw)>l(w)$, then 
$$\BB^{-1}\mCD(\O_{s_iw}\subset \End(\C^n))=\Tcl{i}{0}\left(\BB^{-1}\mCD(\O_{w}\subset \End(\C^n))\right)\,.$$
If $l(ws_i)>l(w)$, then 
$$\BB^{-1}\mCD(\O_{ws_i}\subset \End(\C^n))=\Tcr{i}{0}\left(\BB^{-1}\mCD(\O_{w}\subset \End(\C^n))\right)\,.$$
\end{pro}
Setting $f_w=\BB^{-1}\mCD(\O_{ws_i}\subset \End(\C^n))$ we obtain the following corollary:
\begin{cor}There exist rational functions $f_w$ for $w\in S_n$, such that
\begin{enumerate}
\item $$f_\id=\BB^{-1}\prod_{i=1}^n \big((1+y)\tfrac{x_i}{t_i}\big)
\prod_{1\leq i<j\leq n}\big(1+y\tfrac{x_j}{t_i}\big)\prod_{1\leq j<i\leq n}\big(1-\tfrac{x_j}{t_i}\big)\,.$$
\item if $l(ws_i)>l(w)$, then
$$f_{ws_i}=\Tcr{i}{0}(f_w)\,.$$
\item if $l(s_iw)>l(w)$, then
$$f_{s_iw}=\Tcl{i}{0}(f_w)\,.$$
\item for each permutation $\sigma\in S^n$ after the substitution $x_i=t_{\sigma(i)}$ the function $f_w$ specializes to $\mC(w,0)_{|\sigma}$, i.e.
$$\kappa(f_w)=\mC(w,0)\,. $$ 
\end{enumerate}
\end{cor}

\begin{rem}The function $f_w$ differs from the weight function of \cite{RTV}, which does not satisfy the property (2). 
The weight function is related to a different presentation of the flag variety: either as the quiver variety associated with $A_n$-quiver or as the quotient $\Hom(\C^{n-1},\C^n)/B_{n-1}$.
See the explanation in \cite[Section 7]{RudnickiW}.\end{rem}

\subsection{Twisted motivic Chern classes of matrix Schubert varieties}
We will describe the twisted motivic Chern classes of $\XX_w$ with respect to a family of divisors $D_{w,\lambda}$ which we defined in Section \ref{divdef}.
Let us introduce a notation for the images of the twisted motivic Chern classes in $\KTh_{\T^2}(\End(\C^n))[y]$. Let
$$\mCm(w,\lambda)=\widehat\iota_{w*}\mCD(\XX_w,\partial \XX_w;D_{w,\lambda}) \in\Z[\underline{x}^\pm,\underline{t}^\pm,y]\,,$$
where $\widehat\iota_w:\XX_w\to \End(\C^n)$ is the embedding. 

First we generalize Proposition \ref{dzielenie0} to the case of twisted classes of Schubert varieties: 
\begin{pro}\label{dzielenie} With the notation as above and the notation of Definition \ref{mc-defndf}
 
$$\kappa\left(\BB^{-1}\mCm(w,\lambda)\right)=\mC(w,\lambda)\,.$$
\end{pro}

\begin{proof}We compare the restrictions of both classes at the fixed points. First, we concentrate on the point $[\id]\in G/B$. 
Let $N_-$ be the group consisting of lower-triangular unipotent matrices. The composition of the inclusion and the quotient map $$N_-\hookrightarrow \GL_n\to  \GL_n/B$$ is an isomorphism to the image which is a neighbourhood of $[\id]$. Over that neighbourhood the bundle $\GL_n\to \GL_n/B$ is trivial. The trivialization map
$$\begin{matrix}&(n,b)&\mapsto & nb\\[5pt]
\widehat\varphi_{\id}:&N_-\times B&\longrightarrow & \GL_n\\[5pt] 
&\big\downarrow&&\big\downarrow~~^{pr}\\[5pt]
\varphi_{\id}:&N_-&\longrightarrow & \GL_n/B
\end{matrix}
$$
is equivariant with respect to the diagonal torus $\T\subset \T^2$ acting by conjugation on $N_-$, $B$ and $\GL_n$.
Thus $N_-\hookrightarrow \GL_n$ is a local $\T$-equivariant slice of the $B$-bundle \hbox{$pr:\GL_n\to \GL_n/B$.}
The inverse image $\widehat\varphi_{\id}^{-1}(\XX_w)$ is of the product form $\varphi_{\id}^{-1}(X_w)\times B$.
 Moreover the divisor $D_{w,\lambda}$ is the preimage of $\Delta_{w,\lambda}$, see formula \eqref{w:1}. Therefore \cite[Proposition~4.7~(ii)]{KonW} implies that
$$\big(\mCm(w,\lambda)_{x_i:=t_i}\big){}_{|N_-\cdot B}= pr^*\mC(w,\lambda)_{|N_-} \boxtimes \mC(\id_B)\,.$$
We have $B\simeq \T\times N_+$ and the conjugation action on the torus is trivial. We restrict the above class to $\id\in \GL_n$ and obtain
$$
\big(\mCm(w,\lambda)_{x_i:=t_i}
\big){}_{|\id}=\mC(w,\lambda)_{|[\id]}\cdot (1+y)^n\cdot \prod_{1\leq i<j \leq n} (1+\tfrac {t_j}{t_i})\,.$$
The second and the third factor give $\BB$ with $x_i$ substituted by $t_i$. Thus
$$\mCm(w,\lambda)_{x_i:=t_i}=\kappa(\BB)_{|[\id]} \cdot\mC(w,\lambda)_{|[\id]}\,.$$

For an arbitrary fixed point $\sigma\in (G/B)^\T$ we consider the translated slice $\sigma N_-$. It is invariant with respect to the subtorus $\T^\sigma\subset \T^2$ consisting of pairs $(t, \sigma^{-1} t \sigma)$.
Indeed for $\sigma g\in \sigma N_-$ $$ t (\sigma g)(\sigma^{-1} t\sigma)^{-1}=\sigma (\sigma^{-1} t\sigma)g  (\sigma^{-1} t\sigma)^{-1}\in \sigma N_-\,.$$
In particular $g=\sigma$ is a fixed point for the action of $\T^\sigma$. We apply the previous computation of $\mCm(w,\lambda)$ restricted to the torus $\T^\sigma$ at the point $\sigma$.
The resulting substitution is $x_i\mapsto t_{\sigma(i)}$. Therefore we obtain
$$\mCm(w,\lambda)_{x_i:=t_{\sigma(i)}}=\kappa(\BB)_{|\sigma} \cdot\mC(w,\lambda)_{|\sigma}\,.$$
\end{proof}

The main result relating the twisted motivic Chern classes with the algebra described in Section \ref{abstract_algebra} is the following:

\begin{atw}\label{indukcja_macierzowa} The left and right recursions hold
$$\Tcl{i}{w\lambda}\big(\BB^{-1} \mCm(w,\lambda)\big)=\BB^{-1}\mCm(s_iw,\lambda)\quad \text{if }~~l(s_iw)>l(w)\,,$$
$$\Tcr{i}{\lambda}\big(\BB^{-1} \mCm(w,\lambda)\big)=\BB^{-1}\mCm(ws_i,s_i\lambda)\quad \text{if }~~l(ws_i)>l(w)\,,$$
where $w(\lambda_1,\lambda_2,\dots,\lambda_n)=(\lambda_{w(1)},\lambda_{w(2)},\dots,\lambda_{w(n)})$.
\end{atw}

The theorem is proven in Section \ref{dowodindukcji}.

\begin{rem} A good  starting point for the induction of twisted classes is  $\XX_\id=\overline B\subset\End(\C^n)$. The rational function 
	$$\BB^{-1}\mCD(\id,\lambda)=\BB^{-1}\prod_{i=1}^n \big((1+y)\big(\tfrac{x_i}{t_i}\big)^{1-\lceil \lambda_i\rceil}
	\big)
	\prod_{1\leq i<j\leq n}\big(1+y\tfrac{x_j}{t_i}\big)
	\prod_{1\leq j<i\leq n}\big(1-\tfrac{x_j}{t_i}\big)$$
	has the property that equally well one can apply the left or right Demazure-Lusztig operators, obtaining distinguished representatives of higher dimensional cells. That was not so for the \emph{easy} representative of $\mC(\id,\lambda)$ given in Example \ref{ex:zlystart}.
	The price to pay is that $\BB^{-1}\mCD(\id,\lambda)$ is not a Laurent polynomial.
\end{rem}

\begin{rem}
	The class $\mCm(\id,\lambda)$ which is the start of induction depends on $\lambda$. On the other hand the class  $\mC(\id,\lambda)=[X_\id] \in \KTh_{\T}(G/B)[y]$ does not depend on $\lambda$.
\end{rem}

\section{Resolution of matrix Schubert varieties}
\label{sec:matrix_resolution}
Before proceeding with the proof of Theorem \ref{indukcja_macierzowa} let us describe the geometry of  resolutions of $\XX_w$.
We will extend the construction of the Bott-Samelson resolution in our case, basically applying the more general method of \cite{BenderPerrin}.

\subsection{The left resolution}
For a reduced word $\wu=s_{i_1}s_{i_2}\dots s_{i_l}$ let
$$\ZL \wu=P_{i_1}\times_BP_{i_2}\times_B\dots\times_BP_{i_l}\times_B \overline B\,,$$
where $P_i$ is the minimal parabolic subgroup  containing $B$ and $s_i$. 
The resolution map 
$$\mu_\wu:\ZL \wu\longrightarrow \XX_w\,,$$
is defined as the product of coordinates
$$\mu_\wu([p_1,p_2,\dots,p_l,b])=p_1\, p_2\dots p_l\, b\,,\qquad p_j\in P_{i_j}\,,\quad b\in \overline B\,.$$
The Borel group acts on $\ZL \wu$ from the left and from the right. 
Let 
$$q_\wu:\ZL \wu\longrightarrow Z_\wu= P_{i_1}\times_BP_{i_2}\times_B\dots\times_BP_{i_l}\times_B \{*\}$$
be the map induced by the contraction of $\overline B$ to a point.
The map $q_\wu:\ZL \wu\to Z_\wu$ is a vector bundle over $Z_\wu$. Denote by $\iota_\wu$ its zero section.
The $B\times B$-orbit of $[s_{i_1},s_{i_2},\dots,s_{i_l},\id]\in\ZL w$ is open in $\ZL \wu$ and  maps isomorphically to $\O_w$. The boundary
$$
\partial \ZL \wu= \mu^{-1}_\wu(\partial X_w)
$$
consists of the inverse image  $q_\wu^{-1}(\partial Z_\wu)$ and the \emph{$B$-boundary} 
$$\partial_B\ZL \wu= P_{i_1}\times_BP_{i_2}\times_B\dots\times_BP_{i_l}\times_B \partial B\,,\qquad  \partial B= \overline B\setminus B\,.$$
The divisor $\partial_B\ZL \wu$ is the sum of $n$ irreducible components. For $j\in\{1,2,\dots, n\}$ set 
$$ \partial_{B,j}\ZL \wu= P_{i_1}\times_BP_{i_2}\times_B\dots\times_BP_{i_l}\times_B (\{a_{jj}=0\}\cap\overline B)\,.$$
Moreover $q_\wu^{-1}(\partial Z_\wu)$ decomposes into irreducible components $\partial_j\ZL \wu=q_\wu^{-1}(\partial_j Z_\wu)$ corresponding to omitting the $j$-th letter in the word $\wu$. We obtain a decomposition of the boundary
$$\partial \ZL \wu=
\partial_B \ZL \wu\cup q_\wu^{-1}(\partial Z_\wu)=
\bigcup_{j=1}^{n}\partial_{B,j} \ZL \wu \cup \bigcup_{j=1}^{l(w)}\partial_{j} \ZL \wu\,. $$
Note that $\partial B$ is a SNC divisor in $\overline B$. It follows that $\partial \ZL \wu$ is a SNC divisor.

Each fixed point $p_\eheta\in(\ZL \wu)^\T$ belongs to the zero section $\iota_\wu(Z_\wu)$ and it corresponds to a binary sequence $\eheta=(\epsilon_1,\epsilon_2,\dots,\epsilon_l)$ 
$$p_\eheta=[s_{i_1}^{\epsilon_1},s_{i_2}^{\epsilon_2},\dots,s_{i_l}^{\epsilon_l},0]$$
(the last $0$ denotes the  zero matrix belonging to $\overline B$). Further we will denote the fixed points simply by $\eheta$, dropping $p$ from the notation. 
We will compute the multiplicity
$\mu_\wu^*(M_{w,k})$ along the component $\partial_j\ZL \wu$ (see \eqref{defM} for the definition of $M_{w,k}$).
Let us test the multiplicity on a curve, which is transverse at the point 
$$[s_{i_1},\dots,s_{i_{j-1}},\id,s_{i_{j+1}},\dots,s_{i_l},\id]\,.$$
Let $E_i(u)$ be the elementary matrix of the form
$$\hbox{\footnotesize
$\begin{matrix}i\\i+1\end{matrix}\begin{pmatrix}
1&\\
&\ddots \\
&&1&0\\
&&u&1&\\
&&&&\ddots \\
&&&&& &1
\end{pmatrix}$}\,.$$
The testing curve is defined as
$$u\mapsto [s_{i_1},\dots,s_{i_{j-1}},E_{i_j}(u),s_{i_{j+1}},\dots,s_{i_l},\id]\,.$$
The multiplicity is equal to the order of vanishing of the function
\begin{equation}\label{krotnosckrzywej}u\mapsto m_k(w^{-1}s_{i_1}\dots s_{i_{l-1}}E_{i_j}(u)s_{i_{l+1}}\dots s_{i_l})=m_k(w_{>j}^{-1}\,s_{i_j}\,E_{i_j}(u)\,w_{>j})\,,\end{equation}
where $w_{>j}=s_{i_{j+1}}s_{i_{j+2}}\dots s_{i_l}$.
The resulting formula is fairly explicit. 
The matrix $w^{-1}_{>j}\,s_{i_j}\,E_{i_j}(u)\,w_{>j}$ differs from the identity matrix only at four  entries having indices in the set
$$\{w^{-1}_{>j}(i_j)\;,\;w_{>j}^{-1}(i_j+1)\}\,.$$ 
The key $2\times2$ submatrix is equal to
$$\begin{pmatrix}u&1\\1&0\end{pmatrix}\,.$$
It follows, that the coefficient of $\partial_j \ZL {\wu}$ in $\mu^*_{\wu}(M_{w,k})$  
 is 1 if $$w^{-1}_{>j}(i_j)\leq k<w_{>j}^{-1}(i_j+1)$$ and otherwise it is equal to 0. Note that since the word $\wu$ is reduced, 
$$w_{>j}^{-1}(i_j)<w_{>j}^{-1}(i_j+1)\,.$$
The multiplicity of  $\mu^*_\wu(M_{w,k})$ at $\partial_{B,j} \ZL {\wu}$ is equal to 1 if $k\geq j$ and otherwise it is equal to 0.
By the above discussion we find the multiplicities of the pullback of the divisor  $D_{w,\lambda}$ (see the formula \eqref{dywizor} for the definition of $D_{w,\lambda}$).
\begin{pro}\label{krotnosc}
We have
$$\mu_{\wu}^*(D_{w,\lambda})=
\sum_{j=1}^{n} \lambda_j\cdot\partial_{B,j} \ZL \wu +
\sum_{j=1}^{l(w)}
\langle w_{>j}\lambda,\alpha_{i_j}^\vee \rangle\cdot\partial_{j} \ZL \wu\,.$$
\end{pro}
Note that 
$\langle w\lambda,\alpha_i^\vee \rangle=\lambda_{w^{-1}(i)}-\lambda_{w^{-1}(i+1)}\,. $

\begin{cor} \label{cor:krotnosc2}
	Let $s$ be a simple reflection corresponding to the root $\alpha_i$. Suppose that $w\in W$ is a Weyl group element such that $l(sw)>l(w)$. Let $\eheta$ be a binary sequence of length $l(w)$ and $\ehet=(\delta,\eheta)$ for $\delta\in\{0,1\}$.
	\begin{enumerate}
		\item For $\delta=0$ we have
		$$
		\O_{\ZL {\sw}}(\lceil\mu_{\sw}^*(D_{sw,\lambda})\rceil)_{|\ehet}
		=\O_{\ZL {\wu}}(\lceil\mu_{\wu}^*(D_{w,\lambda})\rceil)_{|\eheta}
		\cdot \O_{\ZL {\sw}}({\lceil\langle w\lambda,\alpha_{i} ^\vee\rangle\rceil}\partial_{1} \ZL \sw)_{|\ehet}\,.
		$$
		\item
		For $\delta=1$ we have
$$
		\O_{\ZL {\sw}}(\lceil\mu_{\sw}^*(D_{sw,\lambda})\rceil)_{|\ehet}
		=s^t(\O_{\ZL {\wu}}(\lceil\mu_{\wu}^*(D_{w,\lambda})\rceil)_{|\eheta})\,.
		$$	
	\end{enumerate}
Here $s^t$ denotes the action on $t$--variables of $\KTh_{\T^2}(pt)=\Z[\underline t^{\pm 1},\underline x^{\pm 1}].$
\end{cor}
\begin{proof}
	Let  $\eheta=(\epsilon_1,\epsilon_2,\dots,\epsilon_l) \in (\ZL{\wu})^\T$ be a fixed point. The character in the normal direction to $\partial_{B,j} \ZL {\wu}$ at the point $\eheta$,  is equal to 
	$$t_{w^\eheta(j)}/x_j\,,$$
	where $$w^\eheta=s_{i_1}^{\epsilon_1}\,s_{i_2}^{\epsilon_2}\dots s_{i_l}^{\epsilon_l}\,.$$
	The remaining characters are the same as for the Bott-Samelson resolution $Z_w$.
	Therefore for $j\in\{1,\dots,n\}$ we have 
	$$
	\O_{\ZL {\sw}}(\partial_{B,j} \ZL {\sw})_{|\ehet}=
	\begin{cases}
		\O_{\ZL {\wu}}(\partial_{B,j} \ZL {\wu})_{|\eheta} & \text{ when } \delta=0\,, \\
		s^t(\O_{\ZL {\wu}}(\partial_{B,j} \ZL {\wu})_{|\eheta}) & \text{ when } \delta=1\,.
	\end{cases}
	$$
	Moreover for $j\in\{1,\dots,l(w)\}$, as in \cite[Section 3.2]{RW}
	$$
	\O_{\ZL {\sw}}(\partial_{j+1} \ZL {\sw})_{|\ehet}=
	\begin{cases}
		\O_{\ZL {\wu}}(\partial_{j} \ZL {\wu})_{|\eheta} & \text{ when } \delta=0\,, \\
		s^t(\O_{\ZL {\wu}}(\partial_{j} \ZL {\wu})_{|\eheta}) & \text{ when } \delta=1\,.
	\end{cases}
	$$
	We apply Proposition \ref{krotnosc} to deduce the recursive formula of Corollary \ref{cor:krotnosc2}.
\end{proof}

\begin{ex}Let $n=4$, $w=s_2s_3$
	$$w(1)=1,\quad w(2)=3,\quad w(3)=4,\quad w(4)=2\,.$$
	Let us analyse the divisors $M_{w,k}$
	$$P_2\times_B P_3\times_B \overline B\to \End(\C^4)$$
	in the neighbourhood of the empty word
	\begin{multline*}s_3\cdot s_2\cdot\hbox{\footnotesize$\left(
			\begin{array}{cccc}
				1 & 0 & 0 & 0 \\
				0 & 1 & 0 & 0 \\
				0 & u & 1 & 0 \\
				0 & 0 & 0 & 1 \\
			\end{array}
			\right)
			\cdot\left(
			\begin{array}{cccc}
				1 & 0 & 0 & 0 \\
				0 & 1 & 0 & 0 \\
				0 & 0 & 1 & 0 \\
				0 & 0 & v & 1 \\
			\end{array}
			\right)
			\cdot\left(
			\begin{array}{cccc}
				a_{11} & a_{12} & a_{13} & a_{14} \\
				0 & a_{22} & a_{23} & a_{24} \\
				0 & 0 & a_{33} & a_{34} \\
				0 & 0 & 0 & a_{44} \\
			\end{array}
			\right)$}=
		\\
		=\hbox{\footnotesize$ \left(
			\begin{array}{cccc}
				a_{11} & a_{12} & a_{13} & a_{14} \\
				0 & u a_{22} & u a_{23}+a_{33} & u
				a_{24}+a_{34} \\
				0 & 0 & v a_{33} & v a_{34}+a_{44} \\
				0 & a_{22} & a_{23} & a_{24} \\
			\end{array}
			\right)$}\,.
	\end{multline*}
	The  consecutive minors are equal to
	$$a_{11},\quad u\,a_{11} a_{22} ,\quad u v\,a_{11}
	a_{22} a_{33} ,\quad a_{11} a_{22}
	a_{33} a_{44}\,.$$
	Therefore in the neighbourhood of the point corresponding to the empty word the divisor $D_{w,\lambda}$ is equal to
		$$\dv\big(a_{11}^{\lambda_1-\lambda_2}( u\,a_{11} a_{22})^{\lambda_2-\lambda_3}(u v\,a_{11}
		a_{22} a_{33})^{\lambda_3-\lambda_4}(a_{11} a_{22}
		a_{33} a_{44})^{\lambda_4}\big)=$$
		$$=(\lambda_2-\lambda_4)\dv(u)+(\lambda_3-\lambda_4)\dv(v)+\sum_{i=1}^4\lambda_i \dv(a_{ii})\,.
		$$
\end{ex}

\subsection{The right resolution}\label{dualresolution}
Multiplying from the right the upper-triangular matrices by the  matrices belonging to $P_i$ we obtain a different resolution 
$$\nu_\wu:\ZR \wu\longrightarrow \XX_w\,,$$
$$\ZR \wu=\overline B\times_B P_{i_1}\times_B P_{i_1}\times_B\dots\times_B P_{i_l}\,,$$
$$\nu_\wu([b,p_1,p_2,\dots,p_l])=b\,p_1\, p_2\dots p_l\,,\qquad b\in \overline B\,,\quad  p_j\in P_{i_j}.$$
As before the boundary  $\partial\ZR \wu=\nu_\wu^{-1}(\partial\XX_w)$ is a SNC divisor. It is a sum of irreducible components
$$\partial \ZR \wu=
\bigcup_{j=1}^{n}\partial_{B,j} \ZR \wu \cup \bigcup_{j=1}^{l(w)}\partial_{j} \ZR \wu\,. $$
The divisor $\partial_j\ZR \wu$ is defined by the condition $p_j\in B$, the $B$-boundary divisor $\partial_{B,j}\ZR \wu$ is defined by $a_{jj}=0$.
The argument as before leads us to the following conclusion

\begin{pro} \label{krotnoscdual}

We have 
	$$\nu_{\wu}^*(D_{w,\lambda})=
	\sum_{j=1}^{n} (w\lambda)_{j}\cdot\partial_{B,j} \ZR \wu +
	\sum_{j=1}^{l(w)}
	\langle w_{>j}\lambda,\alpha_{i_j} ^\vee\rangle\cdot\partial_{j} \ZR \wu\,.$$
\end{pro}
Note that for $\partial_{B,j}\ZR {\wu}$ the indices of $\lambda$ are permuted, since we compute the degree of $m_k(w^{-1}E_k(u)w)$, not $m_k(E_k(u))$ as in the previous case.

\begin{cor} \label{cor:krotnosc1}
	Let $s$ be a simple reflection corresponding to the root $\alpha_i$. Suppose $w\in W$ is a Weyl group element such that $l(ws)>l(w)$. Let $\eheta$ be a binary sequence of length $l(w)$ and $\ehet=(\eheta,\delta)$ for $\delta\in\{0,1\}$.
	\begin{enumerate}
		\item For $\delta=0$ we have
		$$
		\O_{\ZR {\ws}}(\lceil\mu_{\ws}^*(D_{ws,s\lambda})\rceil)_{|\ehet}
		=\O_{\ZR {\wu}}(\lceil\mu_{\wu}^*(D_{w,\lambda})\rceil)_{|\eheta}
		\cdot \O_{\ZR {\ws}}({\lceil \langle s\lambda,\alpha_{i}^\vee \rangle \rceil}\partial_{l(ws)} \ZR \ws)_{|\ehet}\,.
		$$
		\item
		For $\delta=1$ we have
		$$
		\O_{\ZR {\ws}}(\lceil\nu_{\ws}^*(D_{ws,s\lambda})\rceil)_{|\ehet}
		=s^x(\O_{\ZR {\wu}}(\lceil\nu_{\wu}^*(D_{w,\lambda})\rceil)_{|\eheta})\,.
		$$		
	\end{enumerate}
	Here $s^x$ denotes the action on $x$--variables of $\KTh_{\T^2}(pt)=\Z[\underline t^{\pm 1},\underline x^{\pm 1}].$
\end{cor}
\begin{proof}
	Let  $\eheta=(\epsilon_1,\epsilon_2,\dots,\epsilon_l) \in (\ZR{\wu})^\T$ be a fixed point. The character in the normal direction to $\partial_{B,j} \ZR {\wu}$ at the point $\eheta$,  is equal to 
	$$t_{j}/x_{(w^{\eheta})^{-1}(j)}\,.$$
Therefore for $j\in\{1,\dots,n\}$ we have 
$$
\O_{\ZR {\ws}}(\partial_{B,j} \ZR {\ws})_{|\ehet}=
\begin{cases}
	\O_{\ZR {\wu}}(\partial_{B,j} \ZR {\wu})_{|\eheta} & \text{ when } \delta=0\,, \\
	s^x(\O_{\ZR {\wu}}(\partial_{B,j} \ZR {\wu})_{|\eheta}) & \text{ when } \delta=1\,.
\end{cases}
$$
Moreover for $j\in\{1,\dots,l(w)\}$ (see e.g.~\cite[Section 3.2]{RW})
$$
\O_{\ZR {\ws}}(\partial_{j} \ZR {\ws})_{|\ehet}=
\begin{cases}
	\O_{\ZR {\wu}}(\partial_{j} \ZR {\wu})_{|\eheta} & \text{ when } \delta=0\,, \\
	s^x(\O_{\ZR {\wu}}(\partial_{j} \ZR {\wu})_{|\eheta}) & \text{ when } \delta=1\,.
\end{cases}
$$
The claim follows from Corollary \ref{krotnoscdual}.
\end{proof}

\begin{ex}
	For the dual resolution 
	$$\overline B\times_B P_2\times_B P_3\to \End(\C^4)$$
	we have
	\begin{multline*}s_3\cdot s_2\cdot \hbox{\footnotesize $
			\left(
			\begin{array}{cccc}
				a_{11} & a_{12} & a_{13} & a_{14} \\
				0 & a_{22} & a_{23} & a_{24} \\
				0 & 0 & a_{33} & a_{34} \\
				0 & 0 & 0 & a_{44} \\
			\end{array}
			\right)
			\cdot
			\left(
			\begin{array}{cccc}
				1 & 0 & 0 & 0 \\
				0 & 1 & 0 & 0 \\
				0 & u & 1 & 0 \\
				0 & 0 & 0 & 1 \\
			\end{array}
			\right)
			\cdot
			\left(
			\begin{array}{cccc}
				1 & 0 & 0 & 0 \\
				0 & 1 & 0 & 0 \\
				0 & 0 & 1 & 0 \\
				0 & 0 & v & 1 \\
			\end{array}
			\right)$}=\\
		=\hbox{\footnotesize $
			\left(
			\begin{array}{cccc}
				a_{11} & a_{12}+u a_{13} & a_{13}+v
				a_{14} & a_{14} \\
				0 & u a_{33} & a_{33}+v a_{34} &
				a_{34} \\
				0 & 0 & v a_{44} & a_{44} \\
				0 & a_{22}+u a_{23} & a_{23}+v a_{24}
				& a_{24} \\
			\end{array}
			\right)$}\,.\end{multline*}
	The minors are equal to
	$$a_{11},\quad u\,a_{11} a_{33} ,\quad u v\,a_{11}
	a_{33} a_{44} ,\quad a_{11} a_{22}
	a_{33} a_{44}\,.$$
	Therefore in the neighbourhood of the point corresponding to the empty word the divisor $D_{w,\lambda}$ is equal to
		$$\dv\big(a_{11}^{\lambda_1-\lambda_2}( u\,a_{11} a_{33})^{\lambda_2-\lambda_3}(u v\,a_{11}
		a_{33} a_{44})^{\lambda_3-\lambda_4}(a_{11} a_{22}
		a_{33} a_{44})^{\lambda_4}\big)=$$
		$$=(\lambda_2-\lambda_4)\dv(u)+(\lambda_3-\lambda_4)\dv(v)+\lambda_1 \dv(a_1)+\lambda_4 \dv(a_2)+\lambda_2 \dv(a_3)+\lambda_3 \dv(a_4).
		$$
\end{ex}

\section{Proof of Theorem \ref{indukcja_macierzowa}}
\label{dowodindukcji}
Let us recall the notation
$$L_{i,x}=\frac{x_{i+1}}{x_i}\,, \qquad L_{i,t}=\frac{t_{i+1}}{t_i}\,. $$
\subsection{Left induction}
For an arbitrary element $w\in W$ with a reduced word decomposition $\wu$ we define
\begin{align*}
	\mClocDL{\wu}{\lambda}&:=
	\frac{\mC((\ZL{\wu})^\oo \subset \ZL {\wu})\cdot \O_{\ZL {\wu}}(\lceil \mu_{\wu}^*(D_{\sigma,\lambda})\rceil)}{\eu(T\ZL {\wu})}  \,.
\end{align*}
 Let $s=s_i$ be a simple reflection corresponding to the root $\alpha_i$. Suppose $w\in W$ is a Weyl group element such that $l(sw)>l(w)$. Fix a reduced word decomposition $\wu$ of $w$ and the resulting decomposition $\sw$ of $sw$.
\begin{lemma} \label{lem:ind1}
	Let $\eheta$ be a binary sequence corresponding to a fixed point in $\ZL \wu$. Let $\ehet$ be a binary sequence of the form $\ehet=(\delta,\eheta)$, where $\delta\in\{0,1\}$. We have the following equalities in the localized K-theory of a point $S^{-1}\KTh_{\T^2}(pt)[y]$.
	\begin{enumerate}
		\item $$\eu(T\ZL \sw)_{|\pehet}=
		\begin{cases}
			\eu(T \ZL \wu)_{|\peheta}
			\cdot(1-L^{-1}_{i,t}) &\text{ when } \delta=0\,, \\
			s^t(\eu(T \ZL \wu)_{|\peheta})
			\cdot(1-L_{i,t})&\text{ when } \delta=1\,.
		\end{cases}
	$$
		\item
		$$
		O_{\ZL {\sw}}(\lceil\mu_{\sw}^*(D_{sw,\lambda})\rceil)_{|\ehet}=
		\begin{cases}
			\O_{\ZL {\wu}}(\lceil\mu_{\sw}^*(D_{w,\lambda})\rceil)_{|\eheta}
			\cdot L_{i,t}^{\lceil \langle w\lambda,\alpha_{i}^\vee \rangle
				\rceil}   & \text{ when } \delta=0\,, \\
			s^t(\O_{\ZL {\wu}}(\lceil\mu_{\sw}^*(D_{w,\lambda})\rceil)_{|\eheta}) & \text{ when } \delta=1\,.
		\end{cases}
		$$
		
		\item $$\mC((\ZL\sw)^\oo  \subset \ZL \sw)_{|\pehet}=
		\begin{cases}
			\mC((\ZL \wu)^\oo  \subset \ZL \wu)_{|\peheta}\cdot(1+y)L^{-1}_{i,t} & \text{ when } \delta=0\,, \\
			s^t(\mC((\ZL \wu)^\oo  \subset \ZL \wu)_{|\peheta})\cdot\left(1+yL_{i,t}\right)  & \text{ when } \delta=1\,.
		\end{cases}$$
		
		\item
		$$\mClocDL{\sw}{s\lambda}_{|\pehet}=
		\begin{cases}
			\mClocDL{\wu}{\lambda}{}_{|\peheta}\cdot
			(1+y)\cdot\frac{ L_{i,t}^{\left\lceil \langle w\lambda,\alpha_{i}^\vee \rangle
			\right\rceil-1}}{1-L^{-1}_{i,t}}
			& \text{ when } \delta=0\,,\\
			s^t(\mClocDL{\wu}{\lambda}{}_{|\peheta})\cdot
			\frac{1+yL_{i,t}}{1-L_{i,t}} & \text{ when } \delta=1\,.
		\end{cases}
		$$
	\end{enumerate}
\end{lemma}

\begin{proof} The  following map is well defined and it is $B\times B$-equivariant:
$$\varpi:\ZL {\ws}\longrightarrow P_i/B\simeq \PP^1
\,,\qquad 
\varpi([p_0,p_1,p_2,\dots,p_l,x])=[p_0]\,.$$
Let $F_{\id}, F_s$ be the fibers of $\varpi$ over $[\id],[s]\in P_{i}/B$ respectively. The subvariety $F_{\id} \subset \ZL {sw}$ coincides with divisor $\partial_1\ZL \sw$. It is equivariantly isomorphic to $\ZL \wu$. The variety $F_s$ is also isomorphic to $\ZL \wu$. This isomorphism is equivariant after the twist of the $\T^2$-torus action by~$s^t$. 
The fixed point $\pehet$ lies either in $F_{\id}$ when $\delta=0$ or in $F_{s}$ when $\delta=1$.
\\
\\
{\bf 1)} Let $ T_\varpi$ be the relative tangent bundle of $\varpi$.  We have
$$\eu(T\ZL \sw)=\eu(T_\varpi)(1-\varpi^*[T^*P_i/B])\,. $$
We observe that
$$
(T_\varpi)_{\pehet}=
\begin{cases}
	(T\ZL \wu)_{|\peheta} & \text{ for } \delta=0 \,, \\
	s^t(T\ZL \wu)_{|\peheta} & \text{ for } \delta=1 \,,
\end{cases}
\quad\text{and}\quad
\varpi^*(T^* P_{i}/B)_{|\varpi(\pehet)}=
\begin{cases}
	L^{-1}_{i,t} & \text{ for } \delta=0 \,, \\
	L_{i,t} & \text{ for } \delta=1 \,.
\end{cases}
$$
The claim follows from the above formulas. \\
\\
{\bf 2)} We have
$$\O_{\ZL \sw}(\partial_1 \ZL \sw)=\O_{\ZL \sw}(F_{\id}) =\varpi^*\O_{P_i/B}([\id])\,. $$
Therefore
$$\O_{\ZL \sw}(\partial_1 \ZL \sw)_{|\pehet}=\O_{P_i/B}([\id])_{|\varpi(\pehet)}=
\begin{cases}
	L_{i,t} & \text{ for } \delta=0 \,, \\
	1 & \text{ for } \delta=1 \,.
\end{cases}$$
The claim follows from Proposition \ref{cor:krotnosc2}. \\
\\
{\bf 3)}
Suppose that $\delta=0$. Then $\pehet \in F_{\id}$. The divisor $\partial Z_\sw$ is SNC. Therefore (cf. \cite[Lemma 9.7]{KonW})
\begin{align*}
	\mC((\ZL\sw)^\oo  \subset \ZL \sw)_{|F_{\id}} 
	&=(1+y)\cdot \O_{\ZL \sw}(-F_{\id})_{|F_{\id}} \cdot \mC(F_{\id}^\oo  \subset F_{\id})\\
	&=(1+y)L^{-1}_{t,i} \cdot \mC((\ZL \wu)^\oo  \subset \ZL \wu) \,.
\end{align*}
Suppose that $x=1$. The divisor $\partial Z_\sw + F_s$ is SNC, therefore 
\begin{align*} 
	\mC((\ZL\sw)^\oo  \subset \ZL \sw)_{|F_s}
	&=(1+y\O_{\ZL \sw}(-F_{s})_{|F_{s}}) \cdot \mC(F_{s}^\oo  \subset F_{s})\\
	&=(1+yL_{t,i}) \cdot s^t(\mC((\ZL \wu)^\oo  \subset \ZL \wu)) \,.
\end{align*}
{\bf 4)} The last point follows from (2), (3) and Definition \ref{df:twistedmC}. 
\end{proof}
\begin{lemma} \label{lewaindukcja}
	We have
	$$\Tcl{i}{w\lambda}\big(\mCm(w,\lambda)\big)=\mCm(sw,\lambda)\,.$$
\end{lemma}
\begin{proof} Since the Euler  class $\EE$ is symmetric, hence the multiplication by $\EE$ commutes with $\Tcl{i}{w\lambda}$, 
it is enough to prove the equality
$$\Tcl{i}{w\lambda}\big(\mClocD(w,\lambda)\big)=\mClocD(sw,\lambda)\,.$$
	By Definition \ref{df:DL}
	$$\mCm(sw,\lambda)=\mu_{\sw*} \left( \mC((\ZL\sw)^\oo \subset \ZL \sw)\cdot \O_{Z_\sw}(\lceil \mu_\ws^*D_{sw,\lambda}\rceil)\right)\,. $$
	The LRR formula (Theorem \ref{tw:LRR}) implies that 
	the local twisted motivic Chern class $\mClocD(sw,\lambda)$ is equal to the sum
	$$\mClocD(sw,\lambda)=\sum_{\ehet\in (\ZL \sw)^\T} \mClocDL{\sw}{\lambda}_{|\ehet}\,.$$
	The above sum splits as the sum of two subsums: those with $\ehet=(0,\eheta)$ and with $\ehet=(1,\eheta)$, cf. proof of Lemma \ref{lem:BSind}.
The first summand is equal to
\begin{align} \label{w:6}
	\sum_{\eheta\in (\ZL \wu)^\T}\frac{(1+y)
   L_{i,t}^
{\lceil
	\langle w\lambda,\alpha_{i}^\vee \rangle
	\rceil-1}
   }{1-L_{i,t}^{-1}}
\cdot
\mClocDL{\wu}{\lambda}_{|\eheta}=
\frac{(1+y)\cdot
   L_{i,t}^{\lceil \langle w\lambda,\alpha_{i}^\vee \rangle
   	\rceil-1}
   }{1-L_{i,t}^{-1}}
\cdot
\mClocD(w,\lambda)\,.
\end{align}
The second summand is equal to
\begin{align} \label{w:7}
	\sum_{\eheta\in (\ZL \wu)^\T}
\frac{1+yL_{i,t}}
{1-L_{i,t}}\cdot
s^t(\mClocDL{\wu}{\lambda})_{|\eheta}=
\frac{1+yL_{i,t}}
{1-L_{i,t}}\cdot
s^t
\mClocD(w,\lambda)\,.
\end{align}
The sum \eqref{w:6}+\eqref{w:7} is exactly the definition of $\Tcl{i}{\lambda}\big(\mClocD(w,\lambda)\big)$.
\end{proof}
\begin{proof}[Proof of Theorem \ref{indukcja_macierzowa} (1)]
	The operator $\Tcl{i}{w\lambda}$ commutes with multiplication by the classes $\EE$ and~$\BB$. Therefore the result follows from Lemma \ref{lewaindukcja}.
\end{proof}
\begin{rem}It is possible to write a closed formula for $\mClocDL{\wu}{\lambda}$. The restriction at the point  $\eheta=(\epsilon_1,\epsilon_2,\dots,\epsilon_l)\in\{0,1\}^l$ is equal to
	$$(1+y)^n\cdot w^\epsilon\!\!\left(\prod_{j=1}^n \big(\tfrac{x_j}{t_{j}}\big)^{1-\lceil \lambda_i\rceil} \cdot\prod_{1\leq i< j\leq n}\frac{1+y\, {x_j}/{t_i}}{1- {x_j}/{t_i}}\right)
	\cdot\prod_{j=1}^lw^\epsilon_{<j}(\psi(\wu,\eheta,j))\,,$$
	$$\psi(\wu,\eheta,j)=
	\begin{cases} 
		\frac{(1+y)({t_{i_j}}/{t_{i_j+1}})
			^{1-\lceil a
				\rceil}}
		{1- {t_{i_j}}/{t_{i_j+1}}}
		& \text{ if } \epsilon_j=0\,,~
		a=\lambda_{w^{-1}_{>j}(i_j)}-\lambda_{w_{>j}^{-1}(i_j+1)}\,,\\
		\frac{1+y\, {t_{i_j+1}}/{t_{i_j}}}
		{1- {t_{i_j+1}}/{t_{i_j}}}
		& \text{ if } \epsilon_j=1\,,\\
	\end{cases}
	$$
	where $w^\epsilon$ and $w^\epsilon_{<j}=s_{i_1}^{\epsilon_1}\,s_{i_2}^{\epsilon_2}\dots s_{i_{j-1}}^{\epsilon_{j-1}}$ are acting on the indices of $t$--variables.
	The proof of this formula is a direct application of Lemma \ref{lem:ind1}.\end{rem}
\subsection{Right induction}
For an arbitrary element $w\in W$ with a reduced word decomposition $\underline{w}$ we define
\begin{align*}
	\mClocDR{\underline{w}}{\lambda}&:=
	\frac{\mC((\ZR{\underline{w}})^\oo \subset \ZR {\underline{w}})\cdot \O_{\ZR {\underline{w}}}(\lceil \nu_{\underline{w}}^*(D_{w,\lambda})\rceil)}{\eu(T\ZR {\underline{w}})}\,.
\end{align*}
Let $s=s_i$ be a simple reflection corresponding to the root $\alpha_i$. Suppose $w\in W$ is a Weyl group element such that $l(ws)>l(w)$. Fix a reduced word decomposition $\wu$ of $w$. It induces a reduced word decomposition $\ws$ of $ws$.
\begin{lemma} \label{lem:ind2}
	Let $\eheta$ be a binary sequence corresponding to a fixed point in $\ZR \wu$. Let $\ehet$ be a binary sequence of the form $\ehet=(\eheta,\delta)$, where $\delta\in\{0,1\}$. We have the following equalities in the localized K-theory of a point $S^{-1}\KTh_{\T^2}(pt)[y]$.
	\begin{enumerate}
		\item $$\eu(T\ZR \ws)_{|\pehet}=
		\begin{cases}
			\eu(T \ZR \wu)_{|\peheta}
			\cdot(1-L^{-1}_{i,x}) &\text{ when } \delta=0\,, \\
			s^t(\eu(T \ZR \wu)_{|\peheta})
			\cdot(1-L_{i,x})&\text{ when } \delta=1\,.
		\end{cases}
		$$
		\item
		$$
		O_{\ZR {\ws}}(\lceil\nu_{\ws}^*(D_{ws,s\lambda})\rceil)_{|\ehet}=
		\begin{cases}
			\O_{\ZR {\wu}}(\lceil\nu_{\ws}^*(D_{w,\lambda})\rceil)_{|\eheta}
			\cdot L_{i,x}^{\lceil -\langle \lambda,\alpha_{i}^\vee \rangle
				\rceil}  &  \text{ when } \delta=0\,, \\
			s^x(\O_{\ZR {\wu}}(\lceil\nu_{\ws}^*(D_{w,\lambda})\rceil)_{|\eheta}) & \text{ when } \delta=1\,.
		\end{cases}
		$$
		
		\item $$\mC((\ZR\ws)^\oo  \subset \ZR \ws)_{|\pehet}=
		\begin{cases}
			\mC((\ZR \wu)^\oo  \subset \ZR \wu)_{|\peheta}\cdot(1+y)L^{-1}_{i,x} & \text{ when } \delta=0\,, \\
			s^x(\mC((\ZR \wu)^\oo  \subset \ZR \wu)_{|\peheta})\cdot\left(1+yL_{i,x}\right)  & \text{ when } \delta=1\,.
		\end{cases}$$
		
		\item
		$$\mClocDR{\ws}{s\lambda}_{|\pehet}=
		\begin{cases}
			\mClocDR{\wu}{\lambda}{}_{|\peheta}\cdot
			(1+y)\cdot\frac{ L_{i,x}^{\left\lceil -\langle \lambda,\alpha_{i} ^\vee\rangle
					\right\rceil}-1}{1-L^{-1}_{i,x}} &
			\text{ when } \delta=0\,,\\
			s^x(\mClocDR{\wu}{\lambda}{}_{|\peheta})\cdot
			\frac{1+yL_{i,x}}{1-L_{i,x}} & \text{ when } \delta=1\,.
		\end{cases}
		$$
	\end{enumerate}
\end{lemma}

\begin{proof} The proof differs from the proof of Lemma \ref{lem:ind1} by  switching right with left actions.
We employ the natural map to the left coset $B\backslash P_i\simeq \PP^1$:  
	$$\varpi:\ZR {\ws}\longrightarrow B\backslash P_i\,,\qquad 
	\varpi([x,p_1,p_2,\dots,p_{l(w)+1}])=[p_{l(w)+1}]\,.$$
The map $\varpi$ is $B\times B$-equivariant. Let $F_{\id}, F_s$ be the fibers of $\varpi$ over $[\id],[s]\in P_{i}/B$ respectively. The subvariety $F_{\id} \subset \ZR {ws}$ coincides with the divisor $\partial_{l(ws)}\ZR \ws$. It is equivariantly isomorphic to $\ZR \wu$. The variety $F_s$ is also isomorphic to $\ZR \wu$. This isomorphism is equivariant after the twist of the $\T^2$-torus action by~$s^x$. 
	The fixed point $\pehet$ lies either in $F_{\id}$ when $\delta=0$ or in $F_{s}$ when $\delta=1$.
\medskip

The rest of the proof is essentially a repetition of the proof for the left induction. The small differences are caused by the fact that the divisors $D_{w,\lambda}$ are not symmetric with respect to switching from the left to the right action. We leave the details to the reader. 
\end{proof}
\begin{adf}
	Let
	$$\Tcrr{i}{\lambda}(f)=\frac{1+y\,L_{i,x}
		}{1-L_{i,x}}\cdot s^x
	+
	\frac{(1+y) \cdot L_{i,x}^{\lceil -\langle\lambda,\alpha_i^\vee\rangle\rceil-1}}{1-L_{i,x}^{-1} } \cdot \id
	\,.$$
	\end{adf}
Note that 
	$\Tcrr{i}{\lambda}$ is equal to $\Tfl{i}{-\langle \lambda,\alpha_{i}^\vee \rangle}$ acting on $x$--variables.
\begin{rem}
	The operator $\Tcrr{i}{\lambda}$ differs from $\Tcr{i}{\lambda}$ by conjugation by $1+y\,L_{i,x}$, thus it satisfies
\begin{equation}\label{sprzezenie}\BB^{-1}\Tcrr{r}{\lambda}(f)=\Tcr{r}{\lambda}(\BB^{-1}f)\,.\end{equation}
\end{rem}
Repeating the proof of Lemma \ref{lewaindukcja} we obtain
\begin{lemma} \label{prawaindukcja}
	We have
	$$\Tcrr{i}{\lambda}\big(\mClocD(w,\lambda)\big)=\mClocD(ws,s\lambda)\,.$$
\end{lemma}

\begin{proof}[Proof of Theorem \ref{indukcja_macierzowa} (2)] We apply Proposition \ref{prawaindukcja} and the formula \eqref{sprzezenie}. The Euler class $\EE$ commutes with the operator $\Tcr{i}{\lambda}$, hence we  obtain the recursion for $\EE\BB^{-1}\mClocD(w,\lambda)=\BB^{-1}\mCm(w,\lambda)$.
\end{proof}

\section{Twisted double Hecke algebra of the general type}\label{final}
Let $G$ be a semisimple, simply-connected Lie group with a maximal torus $\T$. Consider the representation  ring of~$\T^2$. The characters of the first factor are decorated with the superscript $t$, the characters of the second factor with $x$. 
For a simple reflection $s$ let 
$$L_{s,t}=\alpha_s^{-1}\otimes 1\in R(\T)\otimes R(\T)\simeq R(\T^2)\,,$$
$$L_{s,x}=1\otimes \alpha_s^{-1}\in R(\T)\otimes R(\T)\simeq R(\T^2)\,,$$
$$\Tfl{s}{a}(f)
=\frac{1+y\, L_{s,t}}{1-L_{s,t}} s^tf+ \frac{(1+y)L_{s,t}^{\lceil a\rceil-1}}{1-L_{s,t}^{-1} }
 f\,,$$

$$\Tfr{s}{a}(f)
=\frac{1+y\, L_{s,x}^{-1}}{1-L_{s,x}} s^xf+ \frac{(1+y)L_{s,x}^{\lceil a\rceil-1}}{1-L_{s,x}^{-1} }
 f\,,$$
$$\Tcl{s}{\lambda}=\Tfl{s}{\langle\lambda,\alpha_s^\vee\rangle}\,,\qquad\Tcr{s}{\lambda}=\Tfr{s}{-\langle\lambda,\alpha_s^\vee\rangle}\,.$$
Here the left and right operations are indexed by simple reflections. It is justified to call
the algebra generated by both types of operations (depending on parameters) 
{\it the twisted double Hecke algebra of the general type}.

\begin{rem}
	For $G=\SL_n$ we obtain operators coinciding with the operators from Section \ref{s:universal}, up to a slight change of notation. For the simple reflection $s_i=(i,i+1)$ we have
	$$
	\Tcl{s_i}{\lambda}:=\Tcl{i}{\lambda}\,,\qquad \Tcr{s_i}{\lambda}:=\Tcr{i}{\lambda}\,.
	$$
	The operators defined above are a generalization to the case of a general semisimple, simply-connected Lie group.
\end{rem}
As in the $A_n$ case the right operators commute with left operators.
By construction $\Tcr{s}{\lambda}$ and $\Tcl{s}{\lambda}$ are lifts of the operators $\DL{s,\lambda}$ and $\DLL{s,\lambda}$ given in Definition \ref{df:DL}. Precisely, let $$\kappa:R(\T^2)\to \KTh_\T(G/B)$$
be the surjection sending $L_{s,x}$ to $\LL(\alpha_s)$ and $L_{s,t}$ to $\alpha_s^{-1}=\C_{-\alpha_s}$ viewed as an element of the coefficient ring $\KTh_\T(pt)\simeq R(\T)$. Then
$$\DL{s,\lambda}\circ\kappa=\kappa\circ\Tcr{s}{\lambda}\,,\qquad 
\DLL{s,\lambda}\circ\kappa=\kappa\circ\Tcl{s}{\lambda}\,.$$ 
The operators $\Tcr{s}{\lambda}$ and $\Tcl{s}{\lambda}$ satisfy quadratic relations from Propositions \ref{uniwersalny_braid} and \ref{universalny_braid_l}, because when only one reflection is involved,  one can assume that  $G=\SL_2$.
\begin{atw}\label{kwogolny} Suppose $a,b\in\Q$. If
$a\not\in\Z$, then
$$\Tf{s}{b}\circ\Tf{s}{a}=(y+1)\frac{L^{\lceil -a\rceil}-L^{\lceil b\rceil}}{1-L}\Tf{s}{a}-y\,\id\,.$$
If $a\in \Z$
$$\Tf{s}{b}\circ\Tf{s}{a}=(y+1)\frac{L^{\lceil 1-a\rceil}-L^{\lceil b\rceil}}{1-L}\Tf{s}{a}-y\,\id\,,$$
where $L=L_{s,x}$ or $L_{s,t}$ and $\Tf{s}{c}=\Tfr{s}{c}$ or $\Tfl{s}{c}$ correspondingly, $c\in\{a,b\}$,
\end{atw}
\begin{proof}
Directly by the definition above
$$
\Tf{s}{b}=\Tf{s}{-a}+(y+1)\frac{L^{\lceil -a\rceil}-L^{\lceil b\rceil}}{1-L}\id\,.$$
Hence, if $a\not \in \Z$, by Lemma \ref{quadratic}
$$\Tf{s}{b}\circ\Tf{s}{a}=\Tf{s}{-a}\circ\Tf{s}{a}+(y+1)\frac{L^{\lceil -a\rceil}-L^{\lceil b\rceil}}{1-L}\Tf{s}{a}=-y\,\id+(y+1)\frac{L^{\lceil -a\rceil}-L^{\lceil b\rceil}}{1-L}\Tf{s}{a}\,.$$
For $a\in \Z$ we  replace $-a$ by $1-a$ in the above formula.\end{proof}

The braid relations 
for simply-laced groups already follow from the $\SL_3$ case. It remains to check the braid relations for the $C_2$ and $G_2$ cases. The proofs are straightforward generalizations of the proof of Proposition \ref{universal_braid}. We state the braid relation without the superscripts $\textsc l$, $\textsc r$:  

\subsection{Type $C_2$}
$$\Tcc{1}{  s_2 s_1 s_2 \lambda}\, 
  \Tcc{2}{ s_1 s_2 \lambda}\, \Tcc{1}{ s_2 \lambda}\, \Tcc{2}{  \lambda}
=
 \Tcc{2}{ s_1 s_2 s_1 \lambda}\, 
  \Tcc{1}{ s_2 s_1 \lambda}\, \Tcc{2}{  s_1 \lambda}\, \Tcc{1}{ \lambda}\,.
$$
In the standard coordinates of $\ttt\subset \mathfrak{sp}_2$ the simple roots are the following
$$\alpha_1=(1,-1)\,,\qquad \alpha_2=(0,2)\,.$$ 
Writing the  weight $\lambda=(a,b)$ in coordinates we have the identity

$$\Tf{1}{a-b}\,\Tf{2}{a}\,\Tf{1}{a+b}\,\Tf{2}{b}\,=\,\Tf{2}{b}\,\Tf{1}{a+b}\,\Tf{2}{a}\,\Tf{1}{a-b}\,.$$

\subsection{Type $G_2$}

\begin{multline*}\Tcc{1}{ s_2 s_1 s_2 s_1 s_2 \lambda}
  \Tcc{2}{ s_1 s_2 s_1 s_2 \lambda} 
   \Tcc{1}{ s_2 s_1 s_2 \lambda} 
    \Tcc{2}{ s_1 s_2 \lambda} 
     \Tcc{1}{  s_2 \lambda} \Tcc{2}{ \lambda} =\\
= \Tcc{2}{ s_1 s_2 s_1 s_2 s_1 \lambda} 
  \Tcc{1}{  s_2 s_1 s_2 s_1 \lambda} 
   \Tcc{2}{ s_1 s_2 s_1 \lambda} 
    \Tcc{1}{ s_2 s_1 \lambda} \Tcc{2}{s_1 \lambda} \Tcc{1}{ \lambda}\,.
\end{multline*}
For the weight $\lambda=(a,b)$  written in the basis of simple roots $\alpha_1,\alpha_2$ with $|\alpha_1|>|\alpha_2|$ we obtain an equivalent identity
$$\Tf{1}{a}\,\Tf{2}{3 a+3 b}\,\Tf{1}{2 a+3 b}\,\Tf{2}{3 a+6 b}\,\Tf{1}{a+3 b}\,\Tf{2}{3
   b}
=\Tf{2}{3 b}\,\Tf{1}{a+3 b}\,\Tf{2}{3 a+6 b}\,\Tf{1}{2 a+3 b}\,\Tf{2}{3
   a+3 b}\,\Tf{1}{a}\,.$$
\bigskip

We hope to provide a geometric meaning of $\Tcr{s}{\lambda}$ and $\Tcl{s}{\lambda}$ for a semisimple group $G$ of a general type. The matrix Schubert varieties would be replaced by Borel orbits contained in the nilpotent cone. This topic will be explored in future research. It goes far beyond the scope of the current paper.

\end{document}